\documentclass[a4paper,reqno,12pt]{amsart}

\usepackage{tikz,pgf,amscd,enumerate,multicol}
\usepackage[pagebackref,colorlinks,linkcolor=magenta,citecolor=teal,urlcolor=blue,hypertexnames=false]{hyperref}
\usepackage{amsmath,amssymb,tikz-cd,tkz-graph}

\usepackage[margin=1.2cm]{geometry}
\usetikzlibrary{decorations.markings,arrows}

\tikzset{midarrow/.style = {postaction=decorate, decoration={markings,mark=at position .5 with \arrow{stealth}}}}

\pgfdeclarelayer{background}
\pgfsetlayers{background,main}
\usepackage[all,2cell]{xy}

\newcommand{\Zz}{\mathcal{Z}}
\newcommand{\N}{\mathbf{n}}
\newcommand{\M}{\mathbf{m}}
\newcommand{\Xs}{\mathsf{X}}

\newcommand{\Hom}{\operatorname{Hom}}

\newcommand{\col}{\operatorname{col}}

\newcommand{\KP}{\operatorname{KP}}

\newcommand{\ea}{\mathbf{e}_1}
\newcommand{\eb}{\mathbf{e}_2}
\newcommand{\ec}{\mathbf{e}_3}

\newtheorem{thm}{Theorem}[section]
\newtheorem{cor}[thm]{Corollary}
\newtheorem{lem}[thm]{Lemma}
\newtheorem{prop}[thm]{Proposition}

\theoremstyle{definition}
\newtheorem{dfn}[thm]{Definition}

\theoremstyle{remark}
\newtheorem{rmk}[thm]{Remark}
\newtheorem{rmks}[thm]{Remarks}
\newtheorem{example}[thm]{Example}

\newcommand{\image}{\operatorname{Im}}
\newcommand{\Ker}{\operatorname{Ker}}

\newcommand{\tr}{\operatorname{tr}}

\newcommand{\End}{\operatorname{End}}

\def\mT{\mathcal{T}}

\newcommand{\gr}{\operatorname{gr}}
\newcommand{\ev}{\operatorname{ev}}

\allowdisplaybreaks

\title[A textile framework for $3$-graphs]{Higher-dimensional symbolic dynamics: A textile framework for $3$-graphs}

\author{Roozbeh Hazrat}
\address{Roozbeh Hazrat: Centre for Research in Mathematics and Data Science, Western Sydney University, Australia} \email{r.hazrat@westernsydney.edu.au}

\author{Promit Mukherjee}
\address{Promit Mukherjee: Department of Mathematics, Jadavpur University, Kolkata-700032, India} \email{promitmukherjeejumath@gmail.com}

\author{Sujit Kumar Sardar}
\address{Sujit Kumar Sardar: Department of Mathematics, Jadavpur University, Kolkata-700032, India} \email{sujitk.sardar@jadavpuruniversity.in, sksardarjumath@gmail.com}

\subjclass[2020]{Primary: 37B10, 05C20; Secondary: 19D55, 22A22}

\keywords{Textile systems, Textile algebra, Diagonal graph, Shift of finite type, Three-dimensional textile, Higher-rank graphs, Pullback squares, Homology of a textile system}

\begin{document}

\begin{abstract}
Textile systems are best known to model two-dimensional shifts of finite type. In this article, we associate a discrete algebra with a textile system and provide a groupoid model for it. When the textile system is left-resolving, this algebra coincides with the Kumjian--Pask algebra of the associated $2$-graph. The main objective of this paper is to extend textile systems to dimension $3$ so that the resulting structures can, on the one hand, capture all three-dimensional shifts of finite type and, on the other hand, provide a textile-like framework for $3$-graphs extending the well-known connection between $2$-graphs and left-resolving textile systems. We introduce a model of a three-dimensional textile system and investigate the interplay between such textile systems and $3$-graphs. In particular, our investigation shows that the conditions required to form a $3$-graph from a $3$-colored graph (including the delicate associativity condition on tricolored paths), can be encoded in terms of simple pullback diagrams arising from the textile data. We also define homology groups for three-dimensional textiles and prove that these groups coincide with the homology groups of the associated $3$-graph, thus establishing that our construction is homologically consistent with $3$-graphs.  
\end{abstract}

\maketitle

\tableofcontents

\section{Introduction}\label{sec introduction}
Throughout the advancement of symbolic dynamics, one of the challenging and yet engaging programs has been to extend various dynamical concepts beyond one-dimension. Pursuing this often requires novel techniques as several difficulties may arise in the transition from dimension $1$ to higher; for example, the undecidability of the \emph{nonemptiness} and the \emph{extension} problems (see \cite{Lind-Marcus,JM}). The urge to overcome these difficulties attracted many researchers to higher-dimensional symbolic dynamics and notable works have been done over the years in this direction, see for example \cite{Aso, JM, JM2, Kitchens, MP, Quas-Trow, Schmidt}.

Textile systems, introduced by Nasu \cite{Nasu}, play important roles in understanding the dynamical properties of endomorphisms and automorphisms of topological Markov shifts. A textile system naturally gives rise to a two-dimensional shift of finite type, often referred to as a \emph{textile shift}, whose elements may be viewed as arrays of Wang tiles paved using certain horizontal and vertical concatenability conditions. Moreover, Johnson and Madden \cite{JM} showed that any two-dimensional matrix shift can be realized essentially as a textile shift. As any one-dimensional shift of finite type is conjugate to the \emph{edge shift} of a directed graph (\cite[Theorem 2.3.2]{Lind-Marcus}), the result of Johnson and Madden, undoubtedly, establishes textile systems as the perfect substitutes for directed graphs in dimension $2$ representing two-dimensional shifts of finite type.  

In 2000, Kumjian and Pask \cite{Kumjian-Pask} introduced the concept of higher-rank graphs, or $k$-graphs. These are certain small categories that generalize directed graphs by encoding paths with multi-dimensional degrees and provide a combinatorial framework for studying higher-dimensional Cuntz--Krieger algebras \cite{Robertson}. Apart from their operator algebraic importance, higher-rank graphs can also be used to produce higher-dimensional dynamical systems. In particular, $2$-graphs are closely related to two-dimensional shift spaces as evident from several interesting works (see \cite{BGGLPP,Deaconu, KPW, PRW,SZ,Tang}). In \cite[Proposition 3.7]{Deaconu}, Deaconu first established that the two coordinate graphs of a $2$-graph can be used to determine a textile system and then Tang \cite{Tang} extensively studied the connection between left-resolving textile systems and $2$-graphs.  

The purpose of this article is to develop a textile framework for $3$-graphs, extending the established relationship between $2$-graphs and left-resolving textile systems. To this end, we propose a model of a three-dimensional textile system and investigate the correspondence between such systems and rank-$3$ graphs. It is known that to obtain a $2$-graph from a given $2$-colored directed graph, it suffices to fix the factorization rules for bicolored paths. However, this drastically fails in dimensions larger than two---mere commutativity of bicolored paths is not enough. Hazlewood, Raeburn, Sims and Webster \cite{HRSW} showed that one needs a complete set of commutative squares that satisfy a specific associativity criterion on tricolored paths. It is quite interesting to translate this associativity condition into other combinatorial languages; for example, Yang \cite{Yang} phrased this using the \emph{Yang-Baxter} equations (YBE). In the case of $2$-graphs, path factorization (i.e., the commutativity of squares) can be captured by the unique path lifting properties of the component homomorphisms of the associated textile system. A natural question arises: can this be done for $3$-graphs? In other words, what will be the textile picture encoding the associativity of tricolored paths? Our approach seeks to identify the right textile structure which can reflect the geometry and path structure of $3$-graphs. More precisely, we show that all the conditions required to form a $3$-graph from a given $3$-colored graph can be phrased in terms of simple pullback diagrams originating from our prescribed three-dimensional textile system. We also define homology groups for three-dimensional textile systems and establish that the resulting textile homology agrees with the homology of the associated $3$-graph \cite{HomKPS}. The results presented here demonstrate that the structural connections observed in dimension $2$ between textile systems and higher-rank graphs persist, in an appropriate form, in dimension $3$, and open the possibility of further investigations of higher-dimensional textile systems and their algebraic and dynamical invariants.

Inspired by the rich connection between combinatorial dynamical systems (edge shifts of directed graphs or, equivalently, matrix shifts) and combinatorial algebras (Cuntz algebras, graph $C^*$-algebras and Leavitt path algebras) established in \cite{ALPS, Bates-Pask,Carlsen,CK,Hazdyn} and numerous other works, in this paper, apart from defining three-dimensional textile system, we also associate a discrete algebra (called \emph{textile algebra}) with a usual two-dimensional textile system. When a directed graph is plugged in as a trivial textile, our algebra becomes the usual Leavitt path algebra. Moreover, the textile algebra of a left-resolving textile system recovers the Kumjian--Pask algebra of the associated $2$-graph. In \cite[Lemma 3.2]{Matsumoto}, Matsumoto established that the two-dimensional subshift arising from a $C^*$-textile dynamical system can be determined completely by a one-dimensional subshift. Algebraising this, we show here that the exact same thing happens for the first-quadrant two-dimensional shift of a left-resolving textile system $T$ by explicitly constructing a directed graph, which we call the \emph{diagonal graph} of $T$. We also investigate the relationship between a textile system and its diagonal graph at the level of algebras by showing that there is a natural $\mathbb{Z}$-graded injective homomorphism from the Leavitt path algebra of the diagonal graph into the textile algebra that preserves the respective diagonal subalgebras. 

We now describe the structure of the paper. The germane background material is included in Section \ref{sec preliminaries}. 
In Section \ref{sec algebra of TS}, we define the textile algebra (Definition \ref{def the textile algebra}) from a two-dimensional textile system, prove some basic structural results, and remark on its connections with other algebras (Remark \ref{rem connecting with other algebras}). Then, in Section \ref{sec the diagonal graph}, we introduce the diagonal graph of a textile system and observe that the two-dimensional first-quadrant shift space of a left-resolving textile system is uniquely determined by the edge shift of the diagonal graph. We also show that the diagonal graph can be used to decide whether the textile shift is nonempty (Proposition \ref{pro characterizing nonemptyness}). The main result of this section (Theorem \ref{th the connection between A(T) and L(Delta_T)}) establishes the Leavitt path algebra of the diagonal graph as a $\mathbb{Z}$-graded diagonal-preserving subalgebra of the textile algebra. We show in Section \ref{sec groupoid of TS} that the first-quadrant shift space of a left-resolving textile system together with the shift maps forms a rank-$2$ Deaconu--Renault system and the Steinberg algebra of the associated Deaconu--Renault groupoid recovers the textile algebra (Corollary \ref{cor textile algebra as Steinberg algebra}). In Section \ref{sec 3D textiles}, we propose our model of a three-dimensional textile system (Definition \ref{def 3D textile}) and establish that any three-dimensional shift of finite type can be represented as a three-dimensional textile shift (Proposition \ref{pro realizing 3D SFT as 3-textile SFT}). We prove that a three-dimensional textile is precisely a functor satisfying specific properties from the parallel arrow category $\downdownarrows$ to the category of usual textile systems (Theorem \ref{th alternate description of 3D textile}). This ensures that our description of three-dimensional textile systems is in line with that of the two-dimensional counterparts. Then we investigate the mutual relationship between three-dimensional textiles and $3$-graphs in Section \ref{sec 3-graphs and 3D textiles}. We show that any $3$-graph naturally gives rise to a three-dimensional textile satisfying nice properties depicted in terms of pullback squares (Propositions \ref{pro 3-graph to 3D textile} and \ref{pro path liftings for 3-graph textile}), and conversely, any textile with such pullback squares can be used to construct a $3$-graph (Theorem \ref{th 3D textile to 3-graph}). We finish this section by providing a constructive way to produce a three-dimensional textile (and associated $3$-graph) from two given two-dimensional textiles with a common base graph. The paper concludes in Section \ref{sec homology of 3D textile}, where we define the notion of a covering textile and, using this, introduce the \emph{graded homology} of a two-dimensional textile system. We show that the graded homology coincides with the usual homology of the skew-product of the associated $2$-graph (Theorem \ref{th connecting graded homology with homology of skew-product}). Finally, we define homology groups for a three-dimensional textile system and verify that these are isomorphic to the homology of the corresponding $3$-graph (Theorem \ref{th the homologies are isomorphic}).  

\section{Background information}\label{sec preliminaries}
This section provides a brief overview of the preliminary concepts that will be used in the coming sections. For a detailed and comprehensive understanding of symbolic dynamics and textile systems, readers may consult \cite{BGGLPP, Lind-Marcus, Quas-Trow, Tang}. 

\textbf{Notations:} Throughout the article, we denote by $\mathbb{N}$ the set of all non-negative integers. For any positive integer $k$, $\mathbb{N}^k$ ($\mathbb{Z}^k$) is considered simultaneously as a partially ordered commutative monoid (group) with respect to the component-wise ordering. For $i=1,2,\ldots,k$, $\mathbf{e}_i$ denotes the canonical basic vector in $\mathbb{Z}^k$ with $1$ at the $i$-th position and $0$ elsewhere. By $\mathbf{1}$, we denote the $k$-tuple with all entries equal to $1$. A directed graph $E$ is always represented as a $4$-tuple $E=(E^0,E^1,r_E,s_E)$, where $E^0,E^1$ are the sets of vertices and edges respectively, and $r_E$, $s_E$ are the range and source maps. For two directed graphs $F,E$, a graph homomorphism $f:F\longrightarrow E$ is written as $f=(f^0,f^1)$, where $f^0:F^0\longrightarrow E^0$ and $f^1:F^1\longrightarrow E^1$ denote the vertex map and the edge map, respectivelu. 

\subsection{Higher-dimensional symbolic dynamics}\label{ssec symbolic dynamics}
Let $\mathcal{A}$ be a finite set, which we call the \emph{alphabet} and the elements of which are called \emph{symbols}. For any positive integer $k$, the \emph{full $k$-dimensional shift} is defined as \[\mathcal{A}^{\mathbb{Z}^k}:=\{x=(x_\N)_{\N\in \mathbb{Z}^k}~|~x_\N\in \mathcal{A}~\text{for all}~\N\in \mathbb{Z}^k\}.\] Equivalently, the full shift is the collection of all maps $\mathbb{Z}^k\longrightarrow \mathcal{A}$. The elements of the full shift are generally called \emph{arrays} in $k$-dimensions. If we consider the alphabet $\mathcal{A}$ to be equipped with the discrete topology, then the full shift under the product topology becomes a compact $0$-dimensional metrizable space. The dynamics inside the full shift is caused by the \emph{shift maps}: for any $\N\in \mathbb{Z}^k$ the \emph{shift map} $\sigma_\N$ is a homeomorphism on $\mathcal{A}^{\mathbb{Z}^k}$ such that $(\sigma_\N(x))_\M:=x_{\M+\N}$ for all $\M\in \mathbb{Z}^k$. For any $\N=(n_1,n_2,\ldots,n_k)\in \mathbb{Z}^k$, note that $\sigma_\N=\sigma_{\mathbf{e}_1}^{n_1}\circ \sigma_{\mathbf{e}_2}^{n_2}\circ \cdots \circ \sigma_{\mathbf{e}_k}^{n_k}$. By a \emph{$k$-dimensional shift space} over the alphabet $\mathcal{A}$, we mean a closed subspace $\mathsf{X}$ of $\mathcal{A}^{\mathbb{Z}^k}$ such that $\sigma_\N(\mathsf{X})=\mathsf{X}$ for all $\N\in \mathbb{Z}^k$ or equivalently, $\sigma_{\mathbf{e}_i}(\Xs)=\Xs$ for all $i=1,2,\ldots,k$. 

We fix some notations. Suppose $\Xs$ is a $k$-dimensional shift space over the alphabet $\mathcal{A}$. 
\begin{enumerate}
\item[$\bullet$] Let $\N=(n_1,n_2,\ldots,n_k)$ be a $k$-tuple of positive integers. By $\mathcal{B}_\N(\Xs)$, we denote the set of all blocks (precisely \emph{hypercuboids}) of dimension $n_1\times n_2\times \cdots \times n_k$ that appear in some array of $\Xs$. For $\omega\in \mathcal{B}_\N(\Xs)$ and $\mathbf{q}=(q_1,q_2,\ldots,q_k)$ with $1\le q_i\le n_i$ for all $i=1,2,\ldots,k$, $\omega_\mathbf{q}$ denotes the symbol which appear in the $(q_1,q_2,\ldots,q_k)$-th position of $\omega$. 

\item[$\bullet$] For $x\in \Xs$ and $\N,\M\in \mathbb{Z}^k$ with $\N\le \M$, we denote $x_{[\N,\M]}$ to be the element of $\mathcal{B}_{\M-\N+\mathbf{1}}(\Xs)$ such that $(x_{[\N,\M]})_\mathbf{q}=x_{\mathbf{q}+\N-\mathbf{1}}$ for all $\mathbf{1}\le \mathbf{q}\le \M-\N+\mathbf{1}$.  

\item[$\bullet$] For any block $\omega$ of dimension $\N\ge \mathbf{1}$, we denote \[\mathcal{Z}(\omega):=\{x\in \Xs~|~x_{[0,\N-\mathbf{1}]}=\omega\}.\] In particular, for any symbol $a\in \mathcal{A}$, $\mathcal{Z}(a)=\{x\in \Xs~|~x_0=a\}$. 
\end{enumerate}
Two $k$-dimensional shift spaces $\Xs$ and $\mathsf{Y}$ (possibly over different alphabets) are said to be \emph{(topologically) conjugate} if there exists a homeomorphism $h:\Xs\longrightarrow \mathsf{Y}$ such that $\sigma_\N \circ h=h\circ \sigma_\N$ for all $\N\in \mathbb{Z}^k$. In this case, we write $\Xs\cong \mathsf{Y}$.

Let $\Xs$ be a $k$-dimensional shift space over some alphabet $\mathcal{A}$ and $\N$ a $k$-tuple of positive integers. Define a new alphabet $\mathcal{U}=\mathcal{B}_\N(\Xs)$ and consider the map $\beta_\N:\Xs\longrightarrow \mathcal{U}^{\mathbb{Z}^k}$ defined by \[(\beta_\N(x))_\M:=x_{[\M,\M+\N-\mathbf{1}]}\] for all $\M\in \mathbb{Z}^k$. The \emph{$\N$-th higher-block presentation} of $\Xs$ is denoted by $\Xs^{[\N]}$ and is defined as $\image(\beta_\N)$. This is a $k$-dimensional shift space over the alphabet $\mathcal{U}$. It can be shown that $\beta_\N:\Xs\longrightarrow \Xs^{[\N]}$ is a shift commuting homeomorphism and so $\Xs\cong \Xs^{[\N]}$. 

Suppose $A_i$, $i=1,2,\ldots,k$ are $0-1$ matrices indexed by the alphabet $\mathcal{A}$. Then the \emph{$k$-dimensional matrix shift} determined by the matrices $A_i$ is denoted by $\Xs(A_1,A_2,\ldots,A_k)$ and is defined as \[\Xs(A_1,A_2,\ldots,A_k):=\{x\in \mathcal{A}^{\mathbb{Z}^k}~|~A_i(x_\N,x_{\N+\mathbf{e}_i})=1~\text{for all}~\N\in \mathbb{Z}^k~\text{and}~i=1,2,\ldots,k\}.\] Clearly, $\Xs(A_1,A_2,\ldots,A_k)$ is a shift space over $\mathcal{A}$. A \emph{$k$-dimensional shift of finite type (SFT)} is a shift space $\Xs$ such that $\Xs\cong \Xs(A_1,A_2,\ldots,A_k)$ for some matrices $A_1,A_2,\ldots,A_k$. We remark that there is an independent definition of a shift of finite type in literature (see \cite{BGGLPP,Lind-Marcus,Quas-Trow}) which is more appropriate to justify its name; however, it is a standard fact that any shift of finite type is conjugate to a matrix shift. That is why we choose the above definition.

\subsection{Higher-rank graphs}\label{ssec k-graphs}
For a positive integer $k$, a \emph{$k$-graph} is defined as a pair $(\Lambda,d)$ where $\Lambda$ is a (countably) small category and $d:\Lambda\longrightarrow \mathbb{N}^k$ is a functor (viewing the monoid $\mathbb{N}^k$ as a small category with one object) satisfying the \emph{unique factorization property}: if $d(\lambda)=\N+\M$ for $\N,\M\in \mathbb{N}^k$, then there exist unique $\alpha,\beta\in \Lambda$ such that $d(\alpha)=\N$, $d(\beta)=\M$ and $\lambda=\alpha\beta$. The functor $d$ is usually called the \emph{degree} functor. Often we refer a $k$-graph only by mentioning the category $\Lambda$ keeping in mind that there is a certain degree functor. The \emph{domain} and \emph{co-domain} maps of $\Lambda$ are denoted by $s_\Lambda$ and $r_\Lambda$ respectively. 

Instead of going into more details of $k$-graphs, which one can avail from \cite{Pino,HMPS2,Kumjian-Pask} etc., we now list the notations regarding $k$-graphs, which are going to be used frequently in the article. 

\begin{enumerate}
\item[$\bullet$] For any $\N\in \mathbb{N}^k$, $\Lambda^\N:=d^{-1}(\N)$. By virtue of the unique factorization property, $\Lambda^0$ coincides with the set of objects of $\Lambda$. 

\item[$\bullet$] For $u,v\in \Lambda^0$, $u\Lambda^\mathbf{n}:=r_\Lambda^{-1}(u)\cap \Lambda^\mathbf{n}$, $\Lambda^\mathbf{n} v:=\Lambda^\mathbf{n}\cap s_\Lambda^{-1}(v)$ and $u\Lambda^\mathbf{n} v:=u\Lambda^\mathbf{n} \cap \Lambda^\mathbf{n} v$. 

\item[$\bullet$] For $\lambda\in \Lambda$ and $0\le \N\le \M\le d(\lambda)$, $\lambda(\N,\M)$ denotes the unique element in $\Lambda$ such that $\lambda=\alpha \lambda(\N,\M)\beta$ for some $\alpha\in \Lambda^\N$ and $\beta\in \Lambda^{d(\lambda)-\M}$. 

\item[$\bullet$] For $\N,\M\in \mathbb{N}^k$ with $0\le\N<\M$, $\ev_{\N,\M}$ denotes the following map:
\begin{align*}
\ev_{\N,\M}:\displaystyle{\bigcup_{\mathbf{q}\ge \M}}\Lambda^\mathbf{q}&\longrightarrow \Lambda^{\M-\N}\\
\lambda&\longmapsto \lambda(\N,\M).
\end{align*} 
\end{enumerate}

By a $k$-\emph{colored graph}, we mean a directed graph $E$ with a map $\col:E^1\longrightarrow \{\mathbf{e}_i~|~i=1,2,\ldots,k\}$. If $E^*$ is the free category (path category) of $E$, then the map $\col$ can be extended to a functor $d:E^*\longrightarrow \mathbb{N}^k$ by defining $d(v)=0$ for all $v\in E^0$ and $d(\lambda):=\displaystyle{\sum_{i=1}^{\ell}} \col(\lambda_i)$ for any path $\lambda=\lambda_1 \lambda_2\cdots \lambda_\ell$ of $E$. 

There is a constructive way to obtain a $k$-graph from a $k$-colored graph $E$. For this, we need an equivalence relation $\mathcal{R}$ on $E^*$ such that whenever $\alpha\mathcal{R}\beta$, then $r(\alpha)=r(\beta)$, $s(\alpha)=s(\beta)$ and $d(\alpha)=d(\beta)$. Moreover, suppose the following hold.

$(KG0)$ $\mathcal{R}$ is a congruence on $E^*$: if $\lambda=\lambda_1\lambda_2\in E^*$ and $\mu_1\mathcal{R}\lambda_1$, $\mu_2\mathcal{R}\lambda_2$, then $\lambda\mathcal{R}\mu_1\mu_2$. 

$(KG1)$ $\mathcal{R}|_{E^1}=\Delta_{E^1}$, the diagonal relation on $E^1$.

$(KG2)$ For any $\lambda=\lambda_1\lambda_2\in E^2$ with $d(\lambda_1)=\mathbf{e}_i$ and $d(\lambda_2)=\mathbf{e}_j$, there is a unique $\mu=\mu_1\mu_2\in E^2$ such that $d(\mu_1)=\mathbf{e}_j$, $d(\mu_2)=\mathbf{e}_i$ and $\lambda\mathcal{R}\mu$. 

$(KG3)$ For any $i\neq j\neq \ell\neq i$ and a tricolored path $\alpha\beta\gamma\in E^3$ with color sequence $\mathbf{e}_i,\mathbf{e}_j,\mathbf{e}_\ell$, the two tricolored paths $\gamma''\beta''\alpha''$ and $\gamma_2\beta_2\alpha_2$ with reversed color sequence obtained by following the two routes: 

\begin{align*}
(R1)~ &\alpha\beta\gamma\longrightarrow \alpha \gamma'\beta'\longrightarrow \gamma''\alpha'\beta'\longrightarrow \gamma''\beta''\alpha'',\\
(R2)~ &\alpha\beta\gamma\longrightarrow \beta_1 \alpha_1 \gamma\longrightarrow \beta_1 \gamma_1 \alpha_2\longrightarrow \gamma_2 \beta_2 \alpha_2;
\end{align*}
are the same. Then it was established in \cite[Theorem 4.5]{HRSW} that the quotient category $\Lambda:=E^*/\mathcal{R}$ is a $k$-graph with $d([\alpha]_\mathcal{R})=d(\alpha)$ for any $\alpha\in E^*$. Note that for $2$-colored graphs, $(KG3)$ is vacuously satisfied; as a consequence, to get a $2$-graph from a $2$-colored graph, we just need to specify the factorization rules (see $(KG2)$ above). 

We recall the homology of $k$-graphs as introduced in \cite{HomKPS}. Let $\Lambda$ be a $k$-graph. For $\lambda\in \Lambda$ with $d(\lambda)=(n_1,n_2,\ldots,n_k)$, the \emph{length} of $\lambda$ is defined as $|\lambda|:=\displaystyle{\sum_{i=1}^{k}}n_i$. Now for each $n\in \mathbb{N}$, define \[Q_n(\Lambda):=\{\lambda\in \Lambda~|~d(\lambda)\le \mathbf{1}~\text{and}~|\lambda|=n\}.\] Choose any $\lambda\in Q_n(\Lambda)$ and write $d(\lambda)=\displaystyle{\sum_{t=1}^{n}}\mathbf{e}_{i_t}$ such that $i_1< i_2 <\cdots < i_n$. Then for each $t\in \{1,2,\ldots,n\}$, define \[F_t^0(\lambda):=\lambda(0,d(\lambda)-\mathbf{e}_{i_t}),~~F_t^1(\lambda):=\lambda(\mathbf{e}_{i_t},d(\lambda)).\] For each $n\in \mathbb{N}$, define $C_n(\Lambda):=\mathbb{Z}Q_n(\Lambda)$, the free abelian group generated by $Q_n(\Lambda)$. Note that $C_0(\Lambda)=\mathbb{Z}\Lambda^0$, $C_1(\Lambda)=\mathbb{Z}\left(\displaystyle{\bigsqcup_{i=1}^{k}}~\Lambda^{\mathbf{e}_i}\right)$, $C_2(\Lambda)=\mathbb{Z}\left(\displaystyle{\bigsqcup_{1\le i<j\le k}}\Lambda^{\mathbf{e}_i+\mathbf{e}_j}\right)$ and so on. Also $C_n(\Lambda)=\{0\}$ for all $n>k$. For $1\le n\le k$, the \emph{$n$-th boundary map} is defined on generators as 
\begin{align*}
\partial_n^\Lambda:C_n(\Lambda)&\longrightarrow C_{n-1}(\Lambda)\\
\lambda&\longmapsto \displaystyle{\sum_{l=0}^{1}}\left(\displaystyle{\sum_{t=1}^{n}}(-1)^{i+l}F_t^l(\lambda)\right),
\end{align*}
and for $n> k$, $\partial_n^\Lambda:=\theta$, the zero homomorphism. Then the sequence of abelian groups and group homomorphisms \[\{0\}\overset{\theta}{\longleftarrow} C_0(\Lambda)\overset{\partial_1^\Lambda}{\longleftarrow} C_1(\Lambda)\overset{\partial_2^\Lambda}{\longleftarrow} C_2(\Lambda)\overset{\partial_3^\Lambda}{\longleftarrow}\cdots \overset{\partial_k^\Lambda}{\longleftarrow} C_k(\Lambda)\overset{\theta}{\longleftarrow} \{0\}\] forms a chain complex (\cite[Lemma 3.3]{HomKPS}). The \emph{$n$-th integral homology} of $\Lambda$ is defined as $H_n(\Lambda):=\Ker(\partial_n^\Lambda)/\image(\partial_{n+1}^\Lambda)$. Obviously, $H_0(\Lambda)=\mathbb{Z}\Lambda^0/\image(\partial_1^\Lambda)$ and $H_n(\Lambda)=\{0\}$ for all $n>k$. 

\subsection{Textile systems and their connection with $2$-graphs}\label{ssec textile systems and 2-graphs}
Textile systems were invented by Nasu \cite{Nasu} in order to study dynamical characteristics of endomorphisms and automorphisms of a topological Markov shift. A \emph{textile system} is defined as a quadruple $T=(F,E,p,q)$, where $F,E$ are directed graphs and $p,q:F\longrightarrow E$ are graph homomorphisms such that the correspondence 
\begin{align*}
    F^1 &\longrightarrow F^0\times F^0 \times E^1\times E^1\\
    f &\longmapsto (s_F(f),r_F(f),p(f),q(f))
\end{align*}
is injective. This amounts to say that we can identify each edge $f$ of $F$ uniquely by the following square, commonly referred as a \emph{Wang tile}:

\[
\begin{tikzpicture}[scale=2.0]
\node[inner sep=0.5pt, circle,draw,fill=black] (A1) at (0,0) {};
\node[inner sep=0.5pt, circle,draw,fill=black] (A2) at (1,0) {};
\node[inner sep=0.5pt, circle,draw,fill=black] (A3) at (0,1) {};
\node[inner sep=0.5pt, circle,draw,fill=black] (A4) at (1,1) {};

\path[->,blue, >=latex,thick] (A1) edge [] node[]{} (A2);
\path[->,blue, >=latex,thick] (A1) edge [] node[]{} (A3);
\path[->,blue, >=latex,thick] (A2) edge [] node[]{} (A4);
\path[->,blue, >=latex,thick] (A3) edge [] node[]{} (A4);

\node at (0.5,-0.2) {$p(f)$};
\node at (0.5,1.2) {$q(f)$};
\node at (-0.3,0.5) {$s_F(f)$};
\node at (1.3,0.5) {$r_F(f)$};

\node at (0.5,0.5) {$T_f$};
\end{tikzpicture}
\]
Note that if $F$ and $E$ are directed graphs such that $F$ has no parallel edges (a pair of edges $f_1,f_2\in F^1$ such that $s_F(f_1)=s_F(f_2)$ and $r_F(f_1)=r_F(f_2)$), then for any two homomorphisms $p,q:F\longrightarrow E$, the quadruple $(F,E,p,q)$ forms a textile system. Therefore, $p,q$ come play their roles only when source and range maps fail to distinguish edges of $F$. 

Corresponding to each textile system $T$, there is a \emph{dual} textile system which is denoted by $\overline{T}=(\overline{F},\overline{E},\overline{p},\overline{q})$. The components of $\overline{T}$ are defined as: \[\overline{F}=(E^1,F^1,q^1,p^1),~~~ \overline{E}=(E^0,F^0,q^0,p^0),\] and  \[\overline{p}=(s_E,s_F),~~\overline{q}=(r_E,r_F).\] Any textile system $T$ gives rise to a two-dimensional shift space over the alphabet $F^1$ which is defined as follows
\[ \mathsf{X}_T:=\{x=(x_\N)_{\N\in \mathbb{Z}^2}\in (F^1)^{\mathbb{Z}^2}~|~r_F(x_\mathbf{n})=s_F(x_{\mathbf{n}+\mathbf{e}_1})~\text{and}~q(x_\mathbf{n})=p(x_{\mathbf{n}+\mathbf{e}_2})~\text{for all}~\mathbf{n}\in \mathbb{Z}^2\}.\] Each element $x\in \mathsf{X}_T$ is called a \emph{textile weaved by} $T$. One can realize $\mathsf{X}_T$ as the two-dimensional matrix shift $\mathsf{X}(A_F,A_{\overline{F}})$, where $A_F$ (resp., $A_{\overline{F}}$) is the edge connection matrix of $F$ (resp., $\overline{F}$) and it plays the role of the horizontal (resp., vertical) transition matrix. In 1999, Johnson and Madden \cite{JM} proved that any two-dimensional shift of finite type can be realized essentially as $\mathsf{X}_T$ for some suitably defined textile system. 

There is also a \emph{first-quadrant} version of $\mathsf{X}_T$. Let \[\mathsf{X}_T^+:=\{x:\mathbb{N}^2\longrightarrow F^1~|~r_F(x_\mathbf{n})=s_F(x_{\mathbf{n}+\mathbf{e}_1})~\text{and}~q(x_\mathbf{n})=p(x_{\mathbf{n}+\mathbf{e}_2})~\text{for all}~\mathbf{n}\in \mathbb{N}^2\}.\] We call $\Xs_T^+$, the \emph{first quadrant two-dimensional shift space} associated with $T$. This is clearly a shift of finite type (see \cite{BGGLPP}). 

Suppose $F,E$ are directed graphs. A graph homomorphism $p:F\longrightarrow E$ is said to have \emph{unique $s$-path lifting} (resp., \emph{unique $r$-path lifting}) if  $p^0(v)=s_E(e)$ (resp., $p^0(v)=r_E(e)$) for some $v\in F^0$ and $e\in E^1$ implies the existence of a unique $f\in F^1$ such that $s_F(f)=v$ (resp., $r_F(f)=v$) and $p^1(f)=e$.

Recall that a commutative square 
\[
\begin{tikzpicture}[scale=1.0]
\node[] (A1) at (0,0) {$P$};
\node[] (A2) at (2,0) {$X$};

\node[] (B1) at (0,-2) {$Y$};
\node[] (B2) at (2,-2) {$Z$};

\path[->, >=latex,thick] (A1) edge [] node[]{} (A2);
\path[->, >=latex,thick] (B1) edge [] node[]{} (B2);

\path[->, >=latex,thick] (A1) edge [] node[]{} (B1);
\path[->, >=latex,thick] (A2) edge [] node[]{} (B2);

\node at (1,0.3) {$p_X$};
\node at (1,-2.3) {$g$};

\node at (-0.3,-1) {$p_Y$};
\node at (2.3,-1) {$f$};
\end{tikzpicture}
\]

in a category $\mathcal{C}$ is called a \emph{pullback square} if given any object $Q$ of $\mathcal{C}$ and morphisms $q_X\in \Hom(Q,X)$, $q_Y\in \Hom(Q,Y)$ such that $f q_X=g q_Y$, there exists a unique morphism $h\in \Hom(Q,P)$ such that $p_X h=q_X$ and $p_Y h=q_Y$. It is interesting to note that the unique path lifting property of a graph homomorphism can be described effectively in terms of pullback squares. For example, it is easy to see that a morphism $p:F\longrightarrow E$ has unique $r$-path lifting (resp., unique $s$-path lifting) if and only if the square in the left (resp., right) below
\[
\begin{tikzpicture}[scale=1]
\node[] (A1) at (0,0) {$F^1$};
\node[] (A2) at (2,0) {$F^0$};
\node[] (A3) at (4,0) {$F^1$};
\node[] (A4) at (6,0) {$F^0$};

\node[] (B1) at (0,-2) {$E^1$};
\node[] (B2) at (2,-2) {$E^0$};
\node[] (B3) at (4,-2) {$E^1$};
\node[] (B4) at (6,-2) {$E^0$};

\path[->, >=latex,thick] (A1) edge [] node[]{} (A2);
\path[->, >=latex,thick] (B1) edge [] node[]{} (B2);
\path[->, >=latex,thick] (A3) edge [] node[]{} (A4);
\path[->, >=latex,thick] (B3) edge [] node[]{} (B4);

\path[->, >=latex,thick] (A1) edge [] node[]{} (B1);
\path[->, >=latex,thick] (A2) edge [] node[]{} (B2);
\path[->, >=latex,thick] (A3) edge [] node[]{} (B3);
\path[->, >=latex,thick] (A4) edge [] node[]{} (B4);

\node at (1,0.3) {$r_F$};
\node at (5,0.3) {$s_F$};
\node at (1,-2.3) {$r_E$};
\node at (5,-2.3) {$s_E$};

\node at (-0.3,-1) {$p^1$};
\node at (2.3,-1) {$p^0$};
\node at (3.7,-1) {$p^1$};
\node at (6.3,-1) {$p^0$};
\end{tikzpicture}
\]
is a pullback square in the category \textbf{Set}. We will heavily use this description in Section \ref{sec 3-graphs and 3D textiles}. 

A textile system $T=(F,E,p,q)$ is called \emph{left-resolving} (in short, $LR$) if the morphism $p$ has unique $r$-path lifting and the morphism $q$ has unique $s$-path lifting. There is a one-to-one correspondence between $LR$-textile systems and $2$-graphs (see \cite[Chapter 3]{Tang} for details). As we aim to set up a similar connection in three-dimensions, we feel that it is good to record this here. 

Suppose $T=(F,E,p,q)$ is an $LR$-textile system. Define a graph $G_T$ with $G_T^0:=E^0$, $G_T^1:=F^0\sqcup E^1$. The range and source maps are defined as follows.

\[r(x):=
	\left\{
	\begin{array}{ll}
		q^0(x) & \mbox{if }  x\in F^0,\\
        r_E(x) & \mbox{if }  x\in E^1;
	\end{array}
	\right.~~~
s(x):=
	\left\{
	\begin{array}{ll}
		p^0(x) & \mbox{if }  x\in F^0,\\
        s_E(x) & \mbox{if }  x\in E^1.
	\end{array}
	\right. 
\]
Define $\col:G_T^1\longrightarrow \{\ea,\eb\}$ by 

\[\col(x):=
	\left\{
	\begin{array}{ll}
		\eb & \mbox{if }  x\in F^0,\\
        \ea & \mbox{if }  x\in E^1.
	\end{array}
	\right. 
\] 
Then $G_T$ becomes a $2$-colored graph with respect to the color map $\col$. Now consider the congruence $\mathcal{R}$ on the free category $G_T^*$ generated by the pairs of bicolored paths $(s_F(f)q^1(f),p^1(f)r_F(f))$; $f\in F^1$. Then the quotient category $\Lambda_T:=G_T^*/\mathcal{R}$ is a $2$-graph (see \cite[Theorem 3.8]{Tang} and \cite[Theorem 4.6]{BGGLPP}). 

On the other hand, if we have a $2$-graph $\Lambda$, then we can define a textile system as follows. Let $F_\Lambda$ and $E_\Lambda$ be directed graphs defined as \[F_\Lambda:=(\Lambda^{\eb},\Lambda^{\ea+\eb},\ev(\ea,\ea+\eb),\ev(0,\eb)),~E_\Lambda:=(\Lambda^0,\Lambda^{\ea},s_\Lambda,r_\Lambda).\] Now consider the graph homomorphisms $p_\Lambda,q_\Lambda:F_\Lambda\longrightarrow E_\Lambda$ with \[p_\Lambda:=(r_\Lambda,\ev(0,\ea)),~q_\Lambda:=(s_\Lambda,\ev(\eb,\ea+\eb)).\] Then once can check that $T_\Lambda:=(F_\Lambda,E_\Lambda,p_\Lambda,q_\Lambda)$ is indeed an $LR$-textile system (\cite[Definition 3.5 \& Lemma 3.6]{Tang}). Moreover, $\Lambda_{T_\Lambda}=\Lambda$ and $T_{\Lambda_T}=T$. 

\section{Algebra associated with a textile system}\label{sec algebra of TS}
In this section, given a textile system $T$ and a field $\mathsf{k}$, we define a certain $\mathsf{k}$-algebra $\mathbf{A}(T)$ using generators and relations based on the textile data. This can be viewed as a textile system analogue of the Kumjian--Pask algebra of a $2$-graph.

\begin{dfn}\label{def the textile algebra}
Let $T=(F,E,p,q)$ be any textile system and $\mathsf{k}$ an arbitrary field. For each $w\in F^0$, we introduce a ghost counterpart $w^*$ and set $G(F^0):=\{w^*~|~w\in F^0\}$. Similarly, for each $e\in E^1$, $e^*$ denotes the ghost copy of $e$ and we set $G(E^1):=\{e^*~|~e\in E^1\}$. We denote by $\mathbf{A}(T)$, the associative $\mathsf{k}$-algebra generated by $E^0\cup F^0\cup E^1\cup G(F^0)\cup G(E^1)$ subject to the following relations:

$(T1)$ $uv=\delta_{u,v}u$ for all $u,v\in E^0$;

$(T2)$ $s_E(e)e=e=er_E(e)$, $r_E(e)e^*=e^*=e^*s_E(e)$ for all $e\in E^1$;

\hspace{0.9cm}$p(w)w=w=wq(w)$, $q(w)w^*=w^*=w^*p(w)$ for all $w\in F^0$;

$(T3)$ $e^*e'=\delta_{e,e'}r_E(e)$ for all $e,e'\in E^1$,

\hspace{0.9cm}$w^*w'=\delta_{w,w'}q(w)$ for all $w,w'\in F^0$;

$(T4)$ $u=\displaystyle{\sum_{e\in s_E^{-1}(u)}}ee^*$ for all $u\in E^0$ with $0<|s_E^{-1}(u)|< \infty$,

\hspace{0.9cm}$v=\displaystyle{\sum_{w\in p^{-1}(v)}} ww^*$ for all $v\in E^0$ with $0<|p^{-1}(v)|<\infty$;

$(T5)$ $we=e'w'$ and $e^*w^*=w'^*e'^*$ whenever there is $f\in F^1$ such that $w=s_F(f)$, $w'=r_F(f)$, $e=q(f)$ and $e'=p(f)$.

The algebra $\mathbf{A}(T)$ is called the \emph{textile algebra} of $T$.
\end{dfn}

\begin{rmk}\label{rem LPA as A(T)}
Any directed graph $E$ can be seen as a two-dimensional textile system $T_E:=(E^1,E^0,r_E,s_E)$ where we view $E^1$ and $E^0$ as $0$-graphs (or, \emph{null graphs}), i.e., $(E^1)^0:=E^1$, $(E^1)^1:=\emptyset$, $(E^0)^0:=E^0$ and $(E^0)^1:=\emptyset$. Note that $\mathbf{A}(T_E)=L_\mathsf{k}(E)$. Thus, Leavitt path algebras are trivially textile algebras.     
\end{rmk}
If $X=E^0\cup F^0\cup E^1\cup G(F^0)\cup G(E^1)$, $\mathbb{F}_\mathsf{k}(\omega(X))$ is the free $\mathsf{k}$-algebra on $X$ and $I$ is the ideal of $\mathbb{F}_\mathsf{k}(\omega(X))$ given by the relations $(T1)-(T5)$, then $\mathbf{A}(T)=\mathbb{F}_\mathsf{k}(\omega(X))/I$. The following proposition showcases the universal property of the algebra $\mathbf{A}(T)$.
\begin{prop}\label{pro universal property of the algebra}
Let $T$ be a textile system. Then the algebra $\mathbf{A}(T)$ has the following universal property: if $A$ is any $\mathsf{k}$-algebra with a family of elements $\{P_u,S_w,T_e,S_{w^*},T_{e^*}~|~u\in E^0,w\in F^0,e\in E^1\}$ satisfying relations $(T1)-(T5)$ with $u\in E^0$, $w\in F^0,e\in E^1$ ($w^*\in G(F^0),e^*\in G(E^1)$) replaced by $P_u$, $S_w$ and $T_e$ ($S_{w^*}$, $T_{e^*}$) respectively\footnote{we call such a family a \emph{textile $T$-family} inside $A$.}, then there is a unique $\mathsf{k}$-algebra homomorphism \[\pi_{P,S,T}:\mathbf{A}(T)\longrightarrow A\] such that \[\pi_{P,S,T}(u)=P_u,~~~\pi_{P,S,T}(w)=S_w~\text{and}~\pi_{P,S,T}(e)=T_e,\] for all $u\in E^0,w\in F^0$ and $e\in E^1$. Furthermore, if $T$ is an $LR$-textile system such that $F$ has no sinks, $p^0:F^0\longrightarrow E^0$ is surjective and $p$ has $s$-path lifting property, then $u\neq 0$ in $\mathbf{A}(T)$ for all $u\in E^0$. 
\end{prop}
\begin{proof}
By the universal property of the free algebra $\mathbb{F}_\mathsf{k}(\omega(X))$, there is a unique $\mathsf{k}$-algebra homomorphism $f_{P,S,T}:\mathbb{F}_\mathsf{k}(\omega(X))\longrightarrow A$ such that $f_{P,S,T}(u):=P_u$, $f_{P,S,T}(w):=S_w$, $f_{P,S,T}(e):=T_e$, $f_{P,S,T}(w^*):=S_{w^*}$ and $f_{P,S,T}(e^*):=T_{e^*}$. By the hypothesis, $I\subseteq \ker(f_{P,S,T})$. Thus $f_{P,S,T}$ induces a well-defined $\mathsf{k}$-algebra homomorphism $\pi_{P,S.T}:\mathbf{A}(T)\longrightarrow A$ such that \[\pi_{P,S,T}(u)=P_u,~~~\pi_{P,S,T}(w)=S_w~\text{and}~\pi_{P,S,T}(e)=T_e,\] for all $u\in E^0,w\in F^0$ and $e\in E^1$. The uniqueness of $\pi_{P,S,T}$ follows from the fact that any algebra homomorphism from $\mathbb{F}_\mathsf{k}(\omega(X))$ to $A$ is uniquely determined by the images of elements of $X$. 

For the remaining part, first observe that under the given conditions, $\mathsf{X}_T^+$ is nonempty. We consider $\mathsf{k}\mathsf{X}_T^+$, the free $\mathsf{k}$-vector space with basis $\mathsf{X}_T^+$. We construct a textile $T$-family inside the $\mathsf{k}$-algebra $\End(\mathsf{k}\mathsf{X}_T^+)$. For this we first make an easy observation. Let $w\in F^0$ and $x\in \mathsf{X}_T^+$ be such that $q(w)=p(s_F(x_{(0,0)}))$. By repeated applications of the unique $s$-path lifting of $q$, it follows that there is a unique $x^w\in \mathsf{X}_T^+$ such that $s_F(x^w_{(0,0)})=w$ and $\sigma_{\mathbf{e}_2}(x^w)=x$. Similarly, for any $e\in E^1$ and $x\in \mathsf{X}_T^+$ with $r_E(e)=s_E(p(x_{(0,0)}))$, there is a unique $x^e\in \mathsf{X}_T^+$ such that $p(x^e_{(0,0)})=e$ and $\sigma_{\mathbf{e}_1}(x^e)=x$. Now for $v\in E^0$, $w\in F^0$ and $e\in E^1$, there are unique linear endomorphisms $P_v,S_w,T_e,S_{w^*},T_{e^*}\in \End(\mathsf{k}\mathsf{X}_T^+)$ such that 
$$P_v(x):=
	\left\{
	\begin{array}{ll}
		x  & \mbox{if } p(s_F(x_{(0,0)}))=v, \\
		0 & \mbox{otherwise};
	\end{array}
	\right.$$

$$S_w(x):=
	\left\{
	\begin{array}{ll}
		x^w  & \mbox{if } q(w)=p(s_F(x_{(0,0)})), \\
		0 & \mbox{otherwise};
	\end{array}
	\right.$$

$$T_e(x):=
	\left\{
	\begin{array}{ll}
		x^e  & \mbox{if } r_E(e)=s_E(p(x_{(0,0)})), \\
		0 & \mbox{otherwise};
	\end{array}
	\right.$$

$$S_{w^*}(x):=
	\left\{
	\begin{array}{ll}
		\sigma_{\mathbf{e}_2}(x)  & \mbox{if } s_F(x_{(0,0)})=w, \\
		0 & \mbox{otherwise};
	\end{array}
	\right.$$

$$T_{e^*}(x):=
	\left\{
	\begin{array}{ll}
		\sigma_{\mathbf{e}_1}(x)  & \mbox{if } p(x_{(0,0)})=e, \\
		0 & \mbox{otherwise}.
	\end{array}
	\right.$$
A routine verification yields that $(P,S,T)$ forms a textile $T$-family inside $\End(\mathsf{k}\mathsf{X}_T^+)$. Also, it is not hard to observe that for each $v\in E^0$ there is some $x\in \mathsf{X}_T^+$ such that $p(s_F(x_{(0,0)}))=v$. Therefore, $P_v\neq 0$ for all $v\in E^0$. Now, by the universal property of $\mathbf{A}(T)$, there is a unique $\mathsf{k}$-algebra homomorphism $\pi_{P,S,T}:\mathbf{A}(T)\longrightarrow \End(\mathsf{k}\mathsf{X}_T^+)$ such that \[\pi_{P,S,T}(u)=P_u,~~~\pi_{P,S,T}(w)=S_w~\text{and}~\pi_{P,S,T}(e)=T_e,\] for all $u\in E^0,w\in F^0$ and $e\in E^1$. It follows that $u\neq 0$ for all $u\in E^0$. 
\end{proof}

When multiplying two elements of $\mathbf{A}(T)$, we may come across terms like $e^*e',w^*w',e^*w, w^*e$ where $w,w'\in F^0$ and $e,e'\in E^1$. The first two can be dealt with using relations $(T3)$. The following lemma tells that under certain situations, we can effectively reduce the other two.

\begin{lem}\label{lem swapping real and ghost}
Let $T=(F,E,p,q)$ be a textile system such that $F,E$ are finite graphs, $E$ has no sink, $p^0:F^0\longrightarrow E^0$ is surjective, $p$ has $r$-path lifting and $q$ has $s$-path lifting. Let $w\in F^0$ and $e\in E^1$ be such that $s_E(e)=p(w)$. Then 

$$e^*w:=
	\left\{
	\begin{array}{ll}
		\displaystyle{\sum_{f\in p^{-1}(e)\cap s_F^{-1}(w)}}r_F(f)q(f)^*  & \mbox{if } p^{-1}(e)\cap s_F^{-1}(w)\neq \emptyset, \\
		0 & \mbox{otherwise};
	\end{array}
	\right.$$
and 
$$w^*e:=
	\left\{
	\begin{array}{ll}
		\displaystyle{\sum_{f\in p^{-1}(e)\cap s_F^{-1}(w)}}q(f)r_F(f)^*  & \mbox{if } p^{-1}(e)\cap s_F^{-1}(w)\neq \emptyset, \\
		0 & \mbox{otherwise};
	\end{array}
	\right.$$
Also we have $ee^*ww^*=ww^*ee^*$. 
\end{lem}
\begin{proof}
We have 
\begin{align*}
e^*w&=r_E(e)e^*wq(w)\\
&=\left(\displaystyle{\sum_{w'\in p^{-1}(r_E(e))}} w'w'^*\right)e^*w \left(\displaystyle{\sum_{e'\in s_E^{-1}(q(w))}} e'e'^*\right)\\
&=\displaystyle{\sum_{(w',e')\in p^{-1}(r_E(e))\times s_E^{-1}(q(w))}} w'(w'^*e^*)(we')e'^*
\end{align*}
Fix any pair $(w',e')\in p^{-1}(r_E(e))\times s_E^{-1}(q(w))$. Since $q$ has $s$-path lifting, $s_F^{-1}(w)\cap q^{-1}(e')\neq \emptyset$. Note that $(w'^*e^*)(we')\neq 0$ if and only if there is some $f\in s_F^{-1}(w)\cap q^{-1}(e')$ such that $p(f)=e$ and $r_F(f)=w'$. In this case, the corresponding summand reduces to $w'e'^*=r_F(f)q(f)^*$. From this observation, it follows that $e^*w\neq 0$ if and only if $p^{-1}(e)\cap s_F^{-1}(w)\neq \emptyset$, in which case \[e^*w=\displaystyle{\sum_{f\in p^{-1}(e)\cap s_F^{-1}(w)}}r_F(f)q(f)^*.\]
The case of $w^*e$ is exactly the same. Using $(T5)$, now we have 
\begin{align*}
    ee^*ww^*&= e\left(\displaystyle{\sum_{f\in p^{-1}(e)\cap s_F^{-1}(w)}} r_F(f)q(f)^*\right)w^*\\
    &= \displaystyle{\sum_{f\in p^{-1}(e)\cap s_F^{-1}(w)}} er_F(f)q(f)^*w^*\\
    &= \displaystyle{\sum_{f\in p^{-1}(e)\cap s_F^{-1}(w)}} wq(f)r_F(f)^*e^*\\
    &= w\left(\displaystyle{\sum_{f\in p^{-1}(e)\cap s_F^{-1}(w)}} q(f)r_F(f)^*\right) e^*=ww^*ee^*.
\end{align*}
\end{proof}
\begin{rmk}\label{rem necessity of injectivity}
The fact that each edge $f\in F^1$ is uniquely presented by the Wang tile with sides $s_F(f),p(f)$, $r_F(f),q(f)$ has a significant role in the above result. Without this, the correspondence $f\longmapsto (r_F(f),q(f))$ from $p^{-1}(e)\cap s_F^{-1}(w)$ to the set \[\{(w',e')\in p^{-1}(r_E(e))\times s_E^{-1}(q(w))~|~w'^*e^*we'\neq 0\},\] would not be injective.     
\end{rmk}
We next define a certain subalgebra of $\mathbf{A}(T)$. For this, we observe that any $\alpha\in \omega(F^0\cup E^1)$ which is nonzero in $\mathbf{A}(T)$, can be expressed as $\alpha=e_1e_2\cdots e_m w_1w_2\cdots w_n$ where $r_E(e_i)=s_E(e_{i+1})$ for $i=1,2,\ldots,m-1$, $q(w_j)=p(w_{j+1})$ for $j=1,2,\ldots,n-1$ and $r_E(e_m)=p(w_1)$. Denote \[\alpha^*=w_n^*w_{n-1}^*\cdots w_1^* e_m^* e_{m-1}^*\cdots e_1^*.\] with the obvious understanding that $v^*=v$ for all $v\in E^0$. We define $\mathfrak{D}(T)$ to be the subalgebra of $\mathbf{A}(T)$ generated by elements of the form $\alpha\alpha^*$ where $\alpha\in \omega(F^0\cup E^1)$ such that $\alpha\neq 0$ in $\mathbf{A}(T)$. Using similar arguments used in the proof of Lemma \ref{lem swapping real and ghost}, it is not hard to see that $\mathfrak{D}(T)$ is a commutative subalgebra of $\mathbf{A}(T)$. We call $\mathfrak{D}(T)$, the \emph{diagonal subalgebra} of $\mathbf{A}(T)$. 

We can view $\mathbf{A}(T)$ as a $\mathbb{Z}^2$-graded algebra. To see this, assign degrees to the generators of $\mathbb{F}_\mathsf{k}(\omega(X))$ as follows: \[deg(v):=0,~~deg(w):=\mathbf{e}_2,~~deg(e):=\mathbf{e}_1,~~deg(w^*):=-\mathbf{e}_2,~~deg(e^*):=-\mathbf{e}_1,\] for all $v\in E^0$, $w\in F^0$ and $e\in E^1$. This induces a $\mathbb{Z}^2$-grading on $\mathbb{F}_\mathsf{k}(\omega(X))$. Note that the relations $(T1)-(T5)$ are all homogeneous under this grading, and hence the relational ideal $I$ is a graded ideal. This makes $\mathbf{A}(T)$ a $\mathbb{Z}^2$-graded algebra. In Section \ref{sec the diagonal graph}, we also impose a certain $\mathbb{Z}$-grading on $\mathbf{A}(T)$.

\begin{rmks}\label{rem connecting with other algebras}
(\emph{Connections of $\mathbf{A}(T)$ with other algebras})



$(i)$ \textbf{Kumjian--Pask algebra}: If $T$ is an $LR$-textile system and $E,\overline{E}$ are row-finite without sinks, then the algebra $\mathbf{A}(T)$ is isomorphic to the Kumjian--Pask algebra $\KP(\Lambda_T)$ as $\mathbb{Z}^2$-graded algebras. Indeed, if $(Q,T)$ denotes the Kumjian--Pask $\Lambda_T$-family in $\KP(\Lambda_T)$, then $\{Q_v,T_w,T_e,T_{w^*},T_{e^*}~|~v\in E^0,w\in F^0,e\in E^1\}$ is easily seen to be a textile $T$-family inside $\KP(\Lambda_T)$. This, in view of Proposition \ref{pro universal property of the algebra}, gives a unique algebra homomorphism $\pi:\mathbf{A}(T)\longrightarrow \KP(\Lambda_T)$ such that $\pi(v)=Q_v$, $\pi(w)=T_w$ and $\pi(e)=T_e$. On the other hand, if we set \[P_v:=v,\] and \[S_\lambda:=\lambda(0,\ea)\lambda(\ea,2\ea)\cdots \lambda((m-1)\ea,m\ea)\lambda(m\ea,m\ea+\eb)\cdots\lambda(m\ea+(n-1)\eb,m\ea+n\eb)\] for each $v\in \Lambda_T^0$ and $\lambda\in \Lambda_T^{\neq 0}$ with $d(\lambda)=(m,n)$, then $\{P_v,S_\lambda,S_{\Lambda}^*~|~v\in \Lambda_T^0,\lambda\in \Lambda_T^{\neq 0}\}$ is a Kumjian--Pask $\Lambda_T$-family inside $\mathbf{A}(T)$. By \cite[Theorem 3.4]{Pino}, there is a unique algebra homomorphism $\eta:\KP(\Lambda_T)\longrightarrow \mathbf{A}(T)$ with $\eta(Q_v)=P_v$, $\eta(T_\lambda)=S_\lambda$ and $\eta(T_{\lambda^*})=S_\lambda^*$. One can easily check that $\pi$ and $\eta$ are inverse of each other.  

$(ii)$ \textbf{Quotient of a category algebra}: Define a directed graph $\mathsf{G}(T)$ as follows: \[\mathsf{G}(T)^0:=E^0,~~\mathsf{G}(T)^1:=F^0\cup E^1\cup G(F^0)\cup G(E^1),\] 
\[r(x):=
	\left\{
	\begin{array}{llll}
		q(x)  & \mbox{if } x\in F^0, \\
		r_E(x) & \mbox{if }  x\in E^1,\\
         p(w) & \mbox{if }  x=w^*\in G(F^0),\\
         s_E(e) & \mbox{if }  x=e^*\in G(E^1),\\
	\end{array}
	\right.~~~
s(x):=
	\left\{
	\begin{array}{llll}
		p(x)  & \mbox{if } x\in F^0, \\
		s_E(x) & \mbox{if }  x\in E^1,\\
         q(w) & \mbox{if }  x=w^*\in G(F^0),\\
         r_E(e) & \mbox{if }  x=e^*\in G(E^1).\\
	\end{array}
	\right. 
\] Let $F(\mathsf{G}(T))$ be the \emph{free category} (path category) of $\mathsf{G}(T)$. For each pair of vertices $u,v\in E^0$, define a binary relation $R_{u,v}$ on $\Hom(u,v)$ by \[R_{u,v}:=\{(we,e'w'),(e^*w^*,w'^*e'^*)\in \Hom(u,v)^2~|~w=s_F(f),e=q(f),w'=r_F(f),e'=p(f)~\text{for some}~f\in F^1\}.\] Suppose $\mathcal{R}$ is the least congruence on $F(\mathsf{G}(T))$ containing $R=(R_{u,v})_{(u,v)\in E^0\times E^0}$. Consider the quotient category $\mathcal{C}:=F(\mathsf{G}(T))/\mathcal{R}$ (see \cite{Maclane}). Then it is easy to realize that the textile algebra $\mathbf{A}(T)$ over a field $\mathsf{k}$ is the quotient of the category algebra $\mathsf{k}\mathcal{C}$ modulo the ideal generated by the relations $(T3)$ and $(T4)$. 

$(iii)$ \textbf{Leavitt path algebras of $F$ and $\overline{F}$}: Suppose $T$ is $LR$. If in addition $q$ is injective, then $q:F\longrightarrow E$ is a complete graph morphism. In that case, we observe that $\{q(w),q(f),q(f)^*~|~w\in F^0,f\in F^1\}$ is a Leavitt $F$-family inside $\mathbf{A}(T)$. So there is a unique $\mathsf{k}$-algebra homomorphism $\mu:L(F)\longrightarrow \mathbf{A}(T)$ such that \[\mu(w)=q(w),~~\mu(f)=q(f),~~\mu(f^*)=q(f)^*\] for all $w\in F^0$ and $f\in F^1$. Similarly, if $\overline{q}:\overline{F}\longrightarrow \overline{E}$ is injective, then we have a unique $\mathsf{k}$-algebra homomorphism $\eta:L(\overline{F})\longrightarrow \mathbf{A}(T)$ with the property that \[\eta(e)=r_E(e),~~\eta(f)=r_F(f),~~\eta(f^*)=r_F(f)^*\] for all $e\in \overline{F}^0=E^1$ and $f\in \overline{F}^1=F^1$. Combining these two and using the universal property of direct sum, we have a unique homomorphism \[\Psi:L(F)\oplus L(\overline{F})\longrightarrow \mathbf{A}(T)\] such that $\Psi((x,0))=\mu(x)$ and $\Psi((0,y))=\eta(y)$ for all $x\in L(F)$ and $y\in L(\overline{F})$. Rephrasing relations $(T5)$ in terms of $\Psi$, we have $s_F(f)\Psi((f,0))=s_{\overline{F}}(f)\Psi((0,f))$ for all $f\in F^1$.
\end{rmks}
\section{The diagonal graph}\label{sec the diagonal graph}
In this section, we show that the two-dimensional first-quadrant shift of finite type $\mathsf{X}_T^+$ associated with an $LR$-textile system can be determined completely by a one-dimensional one-sided shift of finite type. This is quite in line with the fact that the two-dimensional subshift of a $C^*$-textile dynamical system can be described entirely from the one-dimensional subshift consisting of all admissible diagonal sequences (see \cite[\S 3]{Matsumoto}).  We adopt the following definition from \cite{Matsumoto}. 

\begin{dfn}\label{def diagonal property}
($cf.$ \cite[Definition 3.1]{Matsumoto}) Let $\mathsf{X}$ be a two-dimensional first-quadrant shift of finite type. We say that $\mathsf{X}$ satisfies the \emph{diagonal property} if for any two arrays $x=(x_{(i,j)})$, $y=(y_{(i,j)})$ in $\mathsf{X}$, $x_{(i,j)}=y_{(i,j)}$ and $x_{(i+1,j+1)}=y_{(i+1,j+1)}$ imply $x_{(i+1,j)}=y_{(i+1,j)}$ and $x_{(i,j+1)}=y_{(i,j+1)}$. 
\end{dfn}

The next lemma sets the stage for the introduction of a suitable one-sided one-dimensional shift of finite type describing $\mathsf{X}_T^+$.
\begin{lem}\label{lem LR textile has the diagonal property}
Let $T$ be an $LR$-textile system. Then $\mathsf{X}_T^+$ satisfies the diagonal property.  
\end{lem}
\begin{proof}
Suppose $x,y\in \mathsf{X}_T^+$ are such that $x_{(i,j)}=y_{(i,j)}$ and $x_{(i+1,j+1)}=y_{(i+1,j+1)}$ for some $(i,j)\in \mathbb{N}^2$. Let $e:=q(x_{(i,j)})=q(y_{(i,j)})$ and $v:=s_F(x_{(i+1,j+1)})=s_F(y_{(i+1,j+1)})$. Note that, \[p(v)=s_E(p(x_{(i+1,j+1)}))=s_E(q(x_{(i+1,j)}))=q(s_F(x_{(i+1,j)}))=q(r_F(x_{(i,j)}))=r_E(e).\] By unique $r$-path lifting of $p$, there should exist a unique $f\in F^1$ such that $p(f)=e$ and $r_F(f)=v$. But we have, $p(x_{(i,j+1)})=e=p(y_{(i,j+1)})$ and $r_F(x_{(i,j+1)})=v=r_F(y_{(i,j+1)})$. Therefore, $x_{(i,j+1)}=y_{(i,j+1)}$. Similarly, using $s$-path lifting of $q$, we have $x_{(i+1,j)}=y_{(i+1,j)}$. 
\end{proof}

We now consider a certain one-sided subshift of the one-sided full shift $(F^1)^\mathbb{N}$. Define \[\mathsf{X}_{\Delta_T}^+:=\{y\in (F^1)^\mathbb{N}~|~r_E(q(y_n))=s_E(p(y_{n+1})\}.\] Thus $\mathsf{X}_{\Delta_T}^+$ is the collection of all admissible diagonal sequences in $\mathsf{X}_T^+$. When $F^1$ is finite, it is clearly a shift of finite type, as it can be described by a finite set of forbidden blocks \[\mathcal{F}:=\{fg\in \mathcal{B}_2(F^1)~|~r_E(q(f))\neq s_E(p(g))\}.\] Just Like any other one-sided shift of finite type in one-dimension, $\mathsf{X}_{\Delta_T}^+$ can be realized essentially as the infinite path space of a finite essential directed graph, in fact, if we consider the directed graph \[\Delta_T:=(E^0,F^1,r_E\circ q, s_E\circ p),\] then $\mathsf{X}_{\Delta_T}^+=\Delta_T^\infty$. We call $\Delta_T$, the \emph{diagonal graph} of $T$.  

Obviously, for any $x\in \mathsf{X}_T^+$, the diagonal sequence $(x_{(i,i)})_{i\in \mathbb{N}}\in \mathsf{X}_{\Delta_T}^+$. On the other hand, using Lemma \ref{lem LR textile has the diagonal property} and a similar line of argument like \cite[Lemma 3.2]{Matsumoto}, it is easy to observe that any sequence $y\in \mathsf{X}_{\Delta_T}^+$ can be uniquely extended to a first-quadrant two-dimensional array $x\in \mathsf{X}_T^+$ such that $y_i=x_{(i,i)}$ for all $i\in \mathbb{N}$. Thus we can say that for an $LR$-textile system $T$, $\mathsf{X}_{\Delta_T}^+$ determines $\mathsf{X}_T^+$. We now characterize, in terms of the graph $\Delta_T$ as well as the data of the textile system, when $\mathsf{X}_T^+$ is nonempty as a shift space. Probably, this is a good time to establish a relationship between the adjacency matrix of $\Delta_T$ and the adjacency matrices of $E$ and $\overline{E}$. 
\begin{lem}\label{lem conncetion between matrices}
Let $T=(F,E,p,q)$ be an $LR$-textile system such that $F,E$ are finite. Let $\overline{T}=(\overline{F},\overline{E},\overline{p},\overline{q})$ be the dual textile system. Then the following hold:

$(i)$ $A(\Delta_T)=A(E)A(\overline{E})$;

$(ii)$ $A_{\Delta_T}=A_F A_{\overline{F}}$ where $A_{\Delta_T}$, $A_F$ and $A_{\overline{F}}$ are the edge connection matrices of $\Delta_T$, $F$ and $\overline{F}$ respectively.  
\end{lem}
\begin{proof}
$(i)$ Let $u,v\in E^0$. Set \[u\Delta_Tv:=\{f\in \Delta_T^1~|~s_{\Delta_T}(f)=u,r_{\Delta_T}(f)=v\}\] and \[uE^1F^0v:=\{(e,w)\in E^1\times F^0~|~s_E(e)=u,r_E(e)=p(w),q(w)=v\}.\] Using unique $r$-path lifting of $p$, it is easy to show that the map 
\begin{align*}
    u\Delta_Tv&\longrightarrow uE^1F^0v\\
    f&\longmapsto (p(f),r_F(f))
\end{align*}
is a bijection. Therefore, \[A(\Delta_T)(u,v)=|u\Delta_Tv|=|uE^1F^0v|=\displaystyle{\sum_{x\in E^0}}A(E)(u,x)A(\overline{E})(x,v)=(A(E)A(\overline{E}))(u,v),\] and hence we have the desired relationship.

$(ii)$ Let $f,g\in F^1$. Note that for any $h\in F^1$, $A_F(f,h)A_{\overline{F}}(h,g)\neq 0$ (in which case it is $1$) if and only if $r_F(f)=s_F(h)$ and $q(h)=p(g)$. Therefore, 
\allowdisplaybreaks
{\begin{align*}
A_F A_{\overline{F}}(f,g)&=\displaystyle{\sum_{h\in F^1}}A_F(f,h)A_{\overline{F}}(h,g)\\
&=|s_F^{-1}(r_F(f))\cap q^{-1}(p(g))|\\
&=\left\{
	\begin{array}{ll}
		1  & \mbox{if } q(r_F(f))=s_E(p(g)), \\
		0 & \mbox{otherwise}
	\end{array}
	\right.\\
&=\left\{
	\begin{array}{ll}
		1  & \mbox{if } r_{\Delta_T}(f)=s_{\Delta_T}(g), \\
		0 & \mbox{otherwise}
	\end{array}
	\right.\\
&=A_{\Delta_T}(f,g).
\end{align*}
}
\end{proof} 

\begin{prop}\label{pro characterizing nonemptyness}
Let $T=(F,E,p,q)$ be an $LR$-textile system such that $F,E$ are finite graphs. Then the following are equivalent.

$(i)$ $\mathsf{X}_T^+$ is nonempty.

$(ii)$ $\mathsf{X}_{\Delta_T}^+$ is nonempty.

$(iii)$ $\Delta_T$ contains a cycle.

$(iv)$ There exist paths $\alpha\in E^*$, $\beta\in \overline{E}^*$ such that $|\alpha|=|\beta|$, $s_E(\alpha)=r_{\overline{E}}(\beta)$, $r_E(\alpha)=s_{\overline{E}}(\beta)$. 

$(v)$ There exist $n\in \mathbb{N}\setminus \{0\}$ and words $\nu\in \mathcal{B}_n(\mathsf{X}_F)$, $\omega\in \mathcal{B}_n(\mathsf{X}_{\overline{F}})$ such that $\nu_1=\omega_n$ and $\nu_n=\omega_1$. 

$(vi)$ There exists $u,v\in \Lambda_T^0$ and $n\in \mathbb{N}\setminus \{0\}$ such that $u\Lambda_T^{n\mathbf{e}_1}v,v\Lambda_T^{n\mathbf{e}_2}u\neq \emptyset$. 
\end{prop}
\begin{proof}
$(i)\Leftrightarrow (ii)$ This follows from the discussion in the paragraph preceding Lemma \ref{lem conncetion between matrices}. 

$(ii)\Leftrightarrow (iii)$ If $\Delta_T$ contains a cycle $c$, then we have an infinite path $c^\infty=ccc\ldots$ in $\Delta_T$, whence $\mathsf{X}_{\Delta_T}^+\neq \emptyset$. Thus $(iii)\Rightarrow (ii)$. On the other hand, if $x=(x_i)_{i\in \mathbb{N}}\in \mathsf{X}_{\Delta_T}^+$, then since $\Delta_T^1=F^1$ is finite, there should exist $i,j\in \mathbb{N}$ with $i< j$ such that $x_j=x_i$. This leads to a closed path $x_i x_{i+1} x_{i+2}\ldots x_{j-1}$ and subsequently, a cycle in $\Delta_T$. Hence $(ii)\Rightarrow (iii)$. 

$(iii)\Leftrightarrow (iv)$ Since $T$ is an $LR$-textile system, $A(E)$ and $A(\overline{E})$ commute. By using Lemma \ref{lem conncetion between matrices} $(i)$, we have $A(\Delta_T)^n=A(E)^nA(\overline{E})^n$ for any $n\in \mathbb{N}$. This, together with the fact that the number of cycles of length $n$ in $\Delta_T$ is $\tr(A(\Delta_T)^n)$, gives the equivalence $(iii)\Leftrightarrow (iv)$. 

$(ii)\Leftrightarrow (v)$ Note that $\mathsf{X}_{\Delta_T}^+$ can be viewed as the one-sided matrix shift $\mathsf{X}(A_{\Delta_T})^+$, which is nonempty if and only if $A_{\Delta_T}^n(f,f)\ge 1$ for some $n\in \mathbb{N}\setminus \{0\}$. By Lemma \ref{lem conncetion between matrices} $(ii)$, this amounts to say that there is some $g\in F^1$ such that $A_F^n(f,g),A_{\overline{F}}^n(g,f)\ge 1$. Finally, since $\mathsf{X}_F=X(A_F)$ and $\mathsf{X}_{\overline{F}}=X(A_{\overline{F}})$, we have the desired equivalence. 

$(iv)\Leftrightarrow (vi)$ Follows immediately since $E,\overline{E}$ are the coordinate graphs of $\Lambda_T$.  
\end{proof}

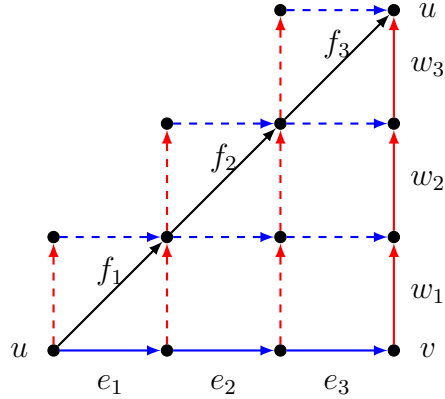
\[
\begin{tikzpicture}[scale=1.5]
\node[inner sep=1.5pt, circle,draw,fill=black] (A1) at (0,0) {};
\node[inner sep=1.5pt, circle,draw,fill=black] (A2) at (1,0) {};
\node[inner sep=1.5pt, circle,draw,fill=black] (A3) at (2,0) {};
\node[inner sep=1.5pt, circle,draw,fill=black] (A4) at (3,0) {};

\node[inner sep=1.5pt, circle,draw,fill=black] (B1) at (0,1) {};
\node[inner sep=1.5pt, circle,draw,fill=black] (B2) at (1,1) {};
\node[inner sep=1.5pt, circle,draw,fill=black] (B3) at (2,1) {};
\node[inner sep=1.5pt, circle,draw,fill=black] (B4) at (3,1) {};

\node[inner sep=1.5pt, circle,draw,fill=black] (C1) at (1,2) {};
\node[inner sep=1.5pt, circle,draw,fill=black] (C2) at (2,2) {};
\node[inner sep=1.5pt, circle,draw,fill=black] (C3) at (3,2) {};

\node[inner sep=1.5pt, circle,draw,fill=black] (D1) at (2,3) {};
\node[inner sep=1.5pt, circle,draw,fill=black] (D2) at (3,3) {};

\path[->, red, dashed, >=latex,thick] (A1) edge [] node[]{} (B1);
\path[->, red, dashed, >=latex,thick] (A2) edge [] node[]{} (B2);
\path[->, red, dashed, >=latex,thick] (A3) edge [] node[]{} (B3);
\path[->, red, >=latex,thick] (A4) edge [] node[]{} (B4);
\path[->, red, dashed, >=latex,thick] (B2) edge [] node[]{} (C1);
\path[->, red, dashed, >=latex,thick] (B3) edge [] node[]{} (C2);
\path[->, red, >=latex,thick] (B4) edge [] node[]{} (C3);
\path[->, red, dashed, >=latex,thick] (C2) edge [] node[]{} (D1);
\path[->, red, >=latex,thick] (C3) edge [] node[]{} (D2);

\path[->,blue, >=latex,thick] (A1) edge [] node[]{} (A2);
\path[->,blue, >=latex,thick] (A2) edge [] node[]{} (A3);
\path[->,blue, >=latex,thick] (A3) edge [] node[]{} (A4);
\path[->,blue, dashed, >=latex,thick] (B1) edge [] node[]{} (B2);
\path[->,blue, dashed, >=latex,thick] (B2) edge [] node[]{} (B3);
\path[->,blue, dashed, >=latex,thick] (B3) edge [] node[]{} (B4);
\path[->,blue, dashed, >=latex,thick] (C1) edge [] node[]{} (C2);
\path[->,blue, dashed, >=latex,thick] (C2) edge [] node[]{} (C3);
\path[->,blue, dashed, >=latex,thick] (D1) edge [] node[]{} (D2);

\path[->,black, >=latex,thick] (A1) edge [] node {} (B2);
\path[->,black, >=latex,thick] (B2) edge [] node {} (C2);
\path[->,black, >=latex,thick] (C2) edge [] node {} (D2);

\node at (-0.3,0) {$u$};
\node at (3.3,0) {$v$};
\node at (3.3,3) {$u$};
\node at (0.5,-0.3) {$e_1$};
\node at (1.5,-0.3) {$e_2$};
\node at (2.5,-0.3) {$e_3$};
\node at (3.3,0.5) {$w_1$};
\node at (3.3,1.5) {$w_2$};
\node at (3.3,2.5) {$w_3$};
\node at (0.5,0.7) {$f_1$};
\node at (1.5,1.7) {$f_2$};
\node at (2.5,2.7) {$f_3$};

\end{tikzpicture}
\]
The above figure may help to understand the implications $(iii)\Leftrightarrow (iv)$. Using the paths $\alpha=e_1e_2e_3\in E^*$ and $\beta=w_1w_2w_3\in \overline{E}^*$ and applying the $LR$-property, one eventually obtains the cycle $f_1f_2f_3$ in $\Delta_T$ by successively filling up the squares starting with the bottom right one. Conversely, any cycle $c=f_1f_2f_3$ based at $u$ in $\Delta_T$ uniquely determines a vertex $v\in E^0$ and two paths $\alpha\in E^*$ and $\beta\in \overline{E}^*$ such that $|\alpha|=|\beta|=|c|$, $s_E(\alpha)=u$, $r_E(\alpha)=v$ and $s_{\overline{E}}(\beta)=v$, $r_{\overline{E}}(\beta)=u$. 

We provide some examples which illustrate the application of the above result.
\begin{example}\label{ex the self-dual example}
$(i)$ Let $E$ be the following directed graph:

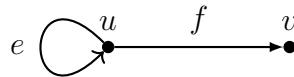
\[
\begin{tikzpicture}[scale=1.2]
\node[inner sep=1.5pt, circle,draw,fill=black,label={$u$}] (A) at (0,0) {};	
\node[inner sep=1.5pt, circle,draw,fill=black,label={$v$}] (B) at (2,0) {};
	
\path[->,black, >=latex,thick] (A) edge [left] node[above=0.05cm]{} (B);
\path[->, black,thick, every loop/.style={looseness=35}](A)edge[in=225, out=135, loop] node [](d) {} (A);

\node at (-1,0) {$e$};

\node at (1,0.3) {$f$};
\end{tikzpicture}
\] 

and $p:E\longrightarrow E$ the graph homomorphism with $p(u)=p(v)=u$ and $p(e)=p(f)=e$. Then $T=(E,E,p,id_E)$ is an $LR$-textile system. The dual system is given as $\overline{T}=(G,H,s_E,r_E)$ where $G,H$ are shown below.
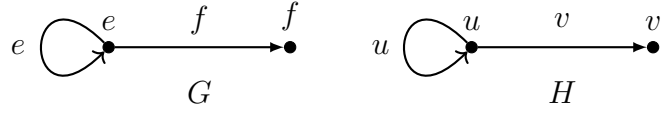
\[
\begin{tikzpicture}[scale=1.2]
\node[inner sep=1.5pt, circle,draw,fill=black,label={$e$}] (A) at (0,0) {};	
\node[inner sep=1.5pt, circle,draw,fill=black,label={$f$}] (B) at (2,0) {};
	
\path[->,black, >=latex,thick] (A) edge [left] node[above=0.05cm]{} (B);
\path[->, black,thick, every loop/.style={looseness=35}](A)edge[in=225, out=135, loop] node [](d) {} (A);

\node at (-1,0) {$e$};

\node at (1,0.3) {$f$};

\node at (1,-0.5) {$G$};

\node[inner sep=1.5pt, circle,draw,fill=black,label={$u$}] (C) at (4,0) {};	
\node[inner sep=1.5pt, circle,draw,fill=black,label={$v$}] (D) at (6,0) {};
	
\path[->,black, >=latex,thick] (C) edge [left] node[above=0.05cm]{} (D);
\path[->, black,thick, every loop/.style={looseness=35}](C)edge[in=225, out=135, loop] node [](d) {} (C);

\node at (3,0) {$u$};

\node at (5,0.3) {$v$};

\node at (5,-0.5) {$H$};
\end{tikzpicture}
\] 

Note that there is a vertex $u\in E^0$ and paths $e\in E^*$, $u\in H^*$ such that $|e|=|u|=1$ and $s_E(e)=r_H(u)=s_H(u)=r_E(e)=u$. Thus, the condition $(iv)$ of Proposition \ref{pro characterizing nonemptyness} is satisfied. Hence, $\mathsf{X}_T^+$ is nonempty. Here, the diagonal graph $\Delta_T=E$.  

\end{example}
The textile system in the above example is a \emph{specific} example of a self-dual textile system, i.e., a textile system which is isomorphic (see \cite[Definitions 4.3]{KPW}) to its dual. If $\psi$ is the canonical isomorphism between $E$ and $H$, then note that $\psi^0=id_{E^0}$. We remark that for this special type of textile systems, $\mathsf{X}_T^+$ is nonempty if and only if $E$ has a cycle. 

If $\mathsf{X}_T^+$ is nonempty, then both $\mathsf{X}_F$ and $\mathsf{X}_{\overline{F}}$ are nonempty. The following example shows that the converse may not hold. 

\begin{example}\label{ex cycles in separate graphs is not sufficient}
Consider a textile system $T=(F,E,p,q)$ where $F,E$ are the following graphs:
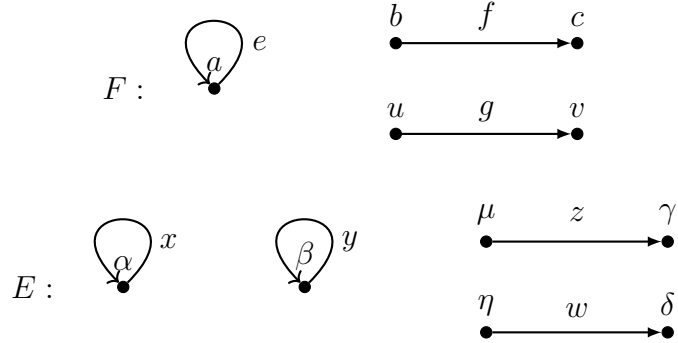
\[
\begin{tikzpicture}[scale=1.2]
\node[inner sep=1.5pt, circle,draw,fill=black,label={$a$}] (A) at (0,0) {};	
\node[inner sep=1.5pt, circle,draw,fill=black,label={$b$}] (B) at (2,0.5) {};
\node[inner sep=1.5pt, circle,draw,fill=black,label={$c$}] (C) at (4,0.5) {};	
\node[inner sep=1.5pt, circle,draw,fill=black,label={$u$}] (D) at (2,-0.5) {};
\node[inner sep=1.5pt, circle,draw,fill=black,label={$v$}] (E) at (4,-0.5) {};
	
\path[->, black,thick, every loop/.style={looseness=35}](A)edge[in=135, out=45, loop] node [](d) {} (A);
\path[->,black, >=latex,thick] (B) edge [left] node[above=0.05cm]{} (C);
\path[->,black, >=latex,thick] (D) edge [left] node[above=0.05cm]{} (E);

\node at (0.5,0.5) {$e$};

\node at (3,0.8) {$f$};

\node at (3,-0.25) {$g$};

\node at (-1,0) {$F:$};
\end{tikzpicture}
\] 

\[
\begin{tikzpicture}[scale=1.2]
\node[inner sep=1.5pt, circle,draw,fill=black,label={$\alpha$}] (A1) at (-2,0) {};
\node[inner sep=1.5pt, circle,draw,fill=black,label={$\beta$}] (A) at (0,0) {};	
\node[inner sep=1.5pt, circle,draw,fill=black,label={$\mu$}] (B) at (2,0.5) {};
\node[inner sep=1.5pt, circle,draw,fill=black,label={$\gamma$}] (C) at (4,0.5) {};	
\node[inner sep=1.5pt, circle,draw,fill=black,label={$\eta$}] (D) at (2,-0.5) {};
\node[inner sep=1.5pt, circle,draw,fill=black,label={$\delta$}] (E) at (4,-0.5) {};
	
\path[->, black,thick, every loop/.style={looseness=35}](A)edge[in=135, out=45, loop] node [](d) {} (A);
\path[->, black,thick, every loop/.style={looseness=35}](A1)edge[in=135, out=45, loop] node [](d) {} (A1);
\path[->,black, >=latex,thick] (B) edge [left] node[above=0.05cm]{} (C);
\path[->,black, >=latex,thick] (D) edge [left] node[above=0.05cm]{} (E);

\node at (0.5,0.5) {$y$};

\node at (-1.5,0.5) {$x$};

\node at (3,0.8) {$z$};

\node at (3,-0.25) {$w$};

\node at (-3,0) {$E:$};
\end{tikzpicture}
\] 
and $p$ is the graph morphism which maps the loop of $F$ to the first loop of $E$, the edges $f,g$ respectively to $z,w$; whereas $q$ sends the loop of $F$ to the second loop of $E$ and the edges $f,g$ to $w,z$ respectively. One can easily see that $T$ is an $LR$-textile system. The graphs of the dual textile system $\overline{T}$ are shown below:

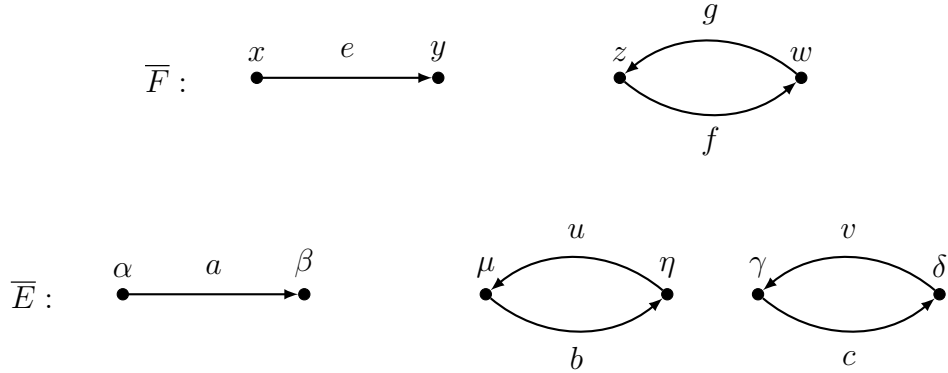
\[
\begin{tikzpicture}[scale=1.2]
\node[inner sep=1.5pt, circle,draw,fill=black,label={$x$}] (A) at (0,0) {};	
\node[inner sep=1.5pt, circle,draw,fill=black,label={$y$}] (B) at (2,0) {};
\node[inner sep=1.5pt, circle,draw,fill=black,label={$z$}] (C) at (4,0) {};	
\node[inner sep=1.5pt, circle,draw,fill=black,label={$w$}] (D) at (6,0) {};

\path[->,black, >=latex,thick] (A) edge [left] node[above=0.05cm]{} (B);
\path[->, >=latex,thick] (C) edge [bend right=40] node[below=0.05cm]{} (D);
\path[->, >=latex,thick] (D) edge [bend right=40] node[below=0.05cm]{} (C);

\node at (1,0.3) {$e$};
\node at (5,0.7) {$g$};
\node at (5,-0.7) {$f$};

\node at (-1,0) {$\overline{F}:$};

\end{tikzpicture}
\]

\[
\begin{tikzpicture}[scale=1.2]
\node[inner sep=1.5pt, circle,draw,fill=black,label={$\alpha$}] (A) at (0,0) {};	
\node[inner sep=1.5pt, circle,draw,fill=black,label={$\beta$}] (B) at (2,0) {};
\node[inner sep=1.5pt, circle,draw,fill=black,label={$\mu$}] (C) at (4,0) {};	
\node[inner sep=1.5pt, circle,draw,fill=black,label={$\eta$}] (D) at (6,0) {};
\node[inner sep=1.5pt, circle,draw,fill=black,label={$\gamma$}] (E) at (7,0) {};	
\node[inner sep=1.5pt, circle,draw,fill=black,label={$\delta$}] (F) at (9,0) {};

\path[->,black, >=latex,thick] (A) edge [left] node[above=0.05cm]{} (B);
\path[->, >=latex,thick] (C) edge [bend right=40] node[below=0.05cm]{} (D);
\path[->, >=latex,thick] (D) edge [bend right=40] node[below=0.05cm]{} (C);
\path[->, >=latex,thick] (E) edge [bend right=40] node[below=0.05cm]{} (F);
\path[->, >=latex,thick] (F) edge [bend right=40] node[below=0.05cm]{} (E);

\node at (1,0.3) {$a$};
\node at (5,0.7) {$u$};
\node at (5,-0.7) {$b$};
\node at (8,0.7) {$v$};
\node at (8,-0.7) {$c$};

\node at (-1,0) {$\overline{E}:$};

\end{tikzpicture}
\]
Note that all the four graphs $F,E,\overline{F}, \overline{E}$ contain cycles. The shift spaces $\mathsf{X}_F$, $\mathsf{X}_{\overline{F}}$ are thus nonempty. But, $\mathsf{X}_T^+=\emptyset$ by Proposition \ref{pro characterizing nonemptyness}, as the diagonal graph $\Delta_T$ (see below) is acyclic. 

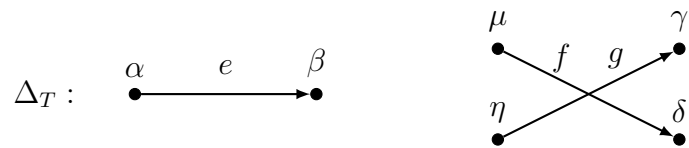
\[
\begin{tikzpicture}[scale=1.2]
\node[inner sep=1.5pt, circle,draw,fill=black,label={$\alpha$}] (A) at (0,0) {};	
\node[inner sep=1.5pt, circle,draw,fill=black,label={$\beta$}] (B) at (2,0) {};
\node[inner sep=1.5pt, circle,draw,fill=black,label={$\mu$}] (C) at (4,0.5) {};	
\node[inner sep=1.5pt, circle,draw,fill=black,label={$\gamma$}] (D) at (6,0.5) {};
\node[inner sep=1.5pt, circle,draw,fill=black,label={$\eta$}] (E) at (4,-0.5) {};	
\node[inner sep=1.5pt, circle,draw,fill=black,label={$\delta$}] (F) at (6,-0.5) {};

\path[->,black, >=latex,thick] (A) edge [left] node[above=0.05cm]{} (B);
\path[->, >=latex,thick] (C) edge [left] node[below=0.05cm]{} (F);
\path[->, >=latex,thick] (E) edge [left] node[below=0.05cm]{} (D);

\node at (1,0.3) {$e$};
\node at (5.3,0.4) {$g$};
\node at (4.7,0.4) {$f$};

\node at (-1,0) {$\Delta_T:$};

\end{tikzpicture}
\]
\end{example}

As we have established that the two-dimensional first quadrant shift space of an $LR$-textile system $T$ is uniquely determined by the one-dimensional one-sided edge shift of the diagonal graph $\Delta_T$, it is natural to investigate the connection between $T$ and $\Delta_T$ in the level of algebras. We now establish such a connection. Note that there is a unique $\mathbb{Z}$-grading on $\mathbb{F}_\mathsf{k}(\omega(X))$ such that \[deg(v):=0,~~deg(w)=deg(w^*):=0,~~deg(e)=1,~~deg(e^*)=-1\] for all $v\in E^0$, $w\in F^0$ and $e\in E^1$. Since the ideal $I$ generated by the relations $(T1)-(T5)$ is again homogeneous with respect to this grading, this gives a $\mathbb{Z}$-grading on $\mathbf{A}(T)$. 

\begin{thm}\label{th the connection between A(T) and L(Delta_T)}
Let $T$ be an $LR$-textile system such that $F,E$ are finite graphs without sinks, $p,q$ are surjective and $p$ has $s$-path lifting. There is an injective $\mathbb{Z}$-graded algebra homomorphism \[\Phi:L(\Delta_T)\longrightarrow \mathbf{A}(T)\] which when restricted to the diagonal subalgebra $\mathcal{D}(\Delta_T)$ of $L(\Delta_T)$ gives an isomorphism between $\mathcal{D}(\Delta_T)$ and $\mathfrak{D}(T)$. 
\end{thm}

\begin{proof}
First, we find a suitable Leavitt $\Delta_T$-family inside the algebra $\mathbf{A}(T)$. Define 
\begin{align*}
    a_v:=v,~~b_f:=s_F(f)q(f)=p(f)r_F(f),~~c_f:=q(f)^*s_F(f)^*=r_F(f)^*p(f)^*
\end{align*}
for all $v\in \Delta_T^0=E^0$, $f\in \Delta_T^1=F^1$. From $(T1)$, it follows that the elements of $E^0$ are orthogonal idempotents in $\mathbf{A}(T)$. Hence $a_ua_v=\delta_{u,v}a_u$. Next observe that for each $f\in \Delta_T^1$, \[a_{s_{\Delta_T(f)}}b_f=s_E(p(f))p(f)r_F(f)=p(f)r_F(f)=b_f,\] and \[b_f a_{r_{\Delta_T(f)}}=s_F(f)q(f)r_E(q(f))=s_F(f)q(f)=b_f\] by using $(T2)$. Similarly, we have $a_{r_{\Delta_T(f)}}c_f=c_f=c_f a_{s_{\Delta_T(f)}}$ for each $f\in \Delta_T^1$. Thus, the path algebra relations are verified. We now verify the Cuntz--Krieger relations for the family $\{a_v,b_f,c_f\}$. For $(CK1)$ we need to show that $c_f b_g=\delta_{f,g}a_{r_{\Delta_T(f)}}$ for $f,g\in \Delta_T^1$. We deal in two cases.

\emph{Case-I:} Suppose $f\neq g$. If $s_F(f)\neq s_F(g)$, then $c_f b_g=q(f)^*s_F(f)^*s_F(g)q(g)=0$, by using $(T3)$, and we are done. So assume that $s_F(f)=s_F(g)$. As $f\neq g$, the unique $s$-path lifting of $q$ now forces that $q(f)\neq q(g)$. This together with $(T3)$ yields \[c_f b_g=q(f)^*s_F(f)^*s_F(g)q(g)=q(f)^*q(s_F(f))q(g)=q(f)^*s_E(q(g))q(g)=q(f)^*q(g)=0.\]

\emph{Case-II:} Suppose $f=g$. Then \[c_f b_g=q(f)^*s_F(f)^*s_F(f)q(f)=q(f)^*q(s_F(f))q(f)=q(f)^*s_E(q(f))q(f)=q(f)^*q(f)=r_E(q(f))=a_{r_{\Delta_T(f)}}\] by using $(T3)$.

For $(CK2)$, choose $v\in \Delta_T^0$ such that $s_{\Delta_T}^{-1}(v)\neq \emptyset$. Now 
\allowdisplaybreaks
{\begin{align*}
\displaystyle{\sum_{f\in s_{\Delta_T}^{-1}(v)}} b_f c_f &= \displaystyle{\sum_{f\in s_{\Delta_T}^{-1}(v)}} s_F(f)q(f)q(f)^*s_F(f)^*\\
&= \displaystyle{\sum_{w\in s_F(s_{\Delta_T}^{-1}(v))}}w\left(\displaystyle{\sum_{f\in s_F^{-1}(w)}}q(f)q(f)^*\right)w^*\\
&= \displaystyle{\sum_{w\in s_F(s_{\Delta_T}^{-1}(v))}} w \left(\displaystyle{\sum_{e\in s_E^{-1}(q(w))}} ee^*\right) w^*\\
&= \displaystyle{\sum_{w\in s_F(s_{\Delta_T}^{-1}(v))}} w q(w) w^*\\
&= \displaystyle{\sum_{w\in s_F(s_{\Delta_T}^{-1}(v))}} ww^*\\
&= \displaystyle{\sum_{w\in p^{-1}(v)}}ww^*=v=a_v.
\end{align*}
}
The third equality follows from the unique $s$-path lifting of $q$, whereas the third from the last is due to the $s$-path lifting of $p$. Indeed, if $w\in p^{-1}(v)$ then $p(w)=s_E(p(f))$ for some $f\in s_{\Delta_T}^{-1}(v)$. Then by $s$-path lifting of $p$, there exists $g\in F_1$ such that $p(g)=p(f)$ and $s_F(g)=w$. It follows that $w\in s_F(s_{\Delta_T}^{-1}(v))$. Hence, $(CK2)$ is verified, and we have shown that $\{a_v,b_f,c_f~|~v\in \Delta_T^0,f\in \Delta_T^1\}$ is a Leavitt $\Delta_T$-family in $\mathbf{A}(T)$. By the universal property of $L(\Delta_T)$, there is a unique homomorphism \[\Phi:L(\Delta_T)\longrightarrow \mathbf{A}(T),\] such that $\Phi(v)=a_v$, $\Phi(f)=b_f$ and $\Phi(f^*)=c_f$ for all $v\in \Delta_T^0$ and $f\in \Delta_T^1$. It is easy to see that $\Phi$ preserves the degrees of the generators of $L(\Delta_T)$, and hence it is a $\mathbb{Z}$-graded homomorphism. The hypothesis on the textile system guarantees that we can apply Proposition \ref{pro universal property of the algebra} to conclude that $\Phi(v)\neq 0$ for all $v\in \Delta_T^0$. Therefore, by the graded uniqueness theorem, it follows that $\Phi$ is injective. Now we show that $\Phi$ maps the diagonal subalgebra $\mathcal{D}(\Delta_T)$ onto $\mathfrak{D}(T)$. It is straightforward to notice that $\Phi(\mathcal{D}(\Delta_T))\subseteq \mathfrak{D}(T)$. We just need to show the other inclusion. Let $\beta\in \omega(E^1\cup F^0)$ be a nonzero element in $\mathbf{A}(T)$. Recall that we can express $\beta$ as 
\begin{equation}\label{eq form of beta}
\beta=e_1e_2\cdots e_m w_1 w_2\cdots w_n=w_1' w_2'\cdots w_n' e_1' e_2'\cdots e_m'   
\end{equation}
for some $m,n\in \mathbb{N}$ with $r_E(e_i)=s_E(e_{i+1})$, $r_E(e_m)=p(w_1)$, $q(w_j)=p(w_{j+1})$, $q(w_j')=p(w_{j+1}')$, $q(w_n')=s_E(e_1')$, $r_E(e_i')=s_E(e_{i+1}')$ for all $i$ and $j$. It suffices to show that $\beta\beta^*\in \Phi(\mathcal{D}(\Delta_T))$. The following cases may appear.

\emph{Case-I:} $m=n$. If $m=n=0$ then $\beta\in E^0$ and consequently, $\beta\beta^*=\Phi(\beta)$. So assume $m=n\ge 1$. Since $r_E(e_m)=p(w_1)$ and $p$ has unique $r$-path lifting, there exists a unique $f\in F^1$ with $r_F(f)=w_1$ and $p(f)=e_m$. Then $\beta=e_1e_2\cdots e_{m-1}s_F(f)q(f)w_2\cdots w_m$. Since $r_E(e_{m-1})=p(s_F(f))$ and $r_E(q(f))=p(w_2)$, we can again apply unique $r$-path lifting on the pairs $(e_{m-1},s_F(f))$ and $(q(f),w_2)$. Continuing in this way, we can eventually reduce $\beta$ in the form \[\beta=\epsilon_1 \omega_1 \epsilon_2 \omega_2\cdots \epsilon_m \omega_m,\] where $r_E(\epsilon_i)=p(\omega_i)$ for $i=1,2,\ldots, m$ and $q(\omega_i)=s_E(\epsilon_{i+1})$ for $i=1,2,\ldots,m-1$. Now, for each $i=1,2,\ldots,m$, applying unique $r$-path lifting of $p$, we have a unique $f_i\in F^1$ with $p(f_i)=\epsilon_i$ and $r_F(f_i)=\omega_i$. It follows that $\alpha=f_1f_2\ldots f_m\in \Delta_T^*$ and \[\Phi(\alpha\alpha^*)=\Phi(f_1)\cdots \Phi(f_m)\Phi(f_m^*)\cdots \Phi(f_1^*)=\beta\beta^*.\]

\emph{Case-II:} $m-n\ge 1$. We use induction on $m-n$. If $m-n=1$ then \[\beta=e_1 e_2\cdots e_m w_1 w_2 \cdots w_n q(w_n)=e_1 e_2\cdots e_m w_1 w_2 \cdots w_n \left(\displaystyle{\sum_{w\in p^{-1}(q(w_n))}} ww^*\right)=\displaystyle{\sum_{w\in p^{-1}(q(w_n))}} \beta_w w^*,\] using $(T4)$ where $\beta_w:=e_1 e_2 \cdots e_m w_1 w_2 \cdots w_n w$. Now, \[\beta \beta^*=\displaystyle{\sum_{w\in p^{-1}(q(w_n))}} \beta_w w^* w_n^*\cdots w_1^* e_m^* \cdots e_1^*=\displaystyle{\sum_{w\in p^{-1}(q(w_n))}} \beta_w \beta_w^*\in \Phi(\mathcal{D}(\Delta_T)),\] since each $\beta_w \beta_w^*\in \Phi(\mathcal{D}(\Delta_T))$ by \emph{Case-I}. Hence, we are done for the base case. Now assume that we have the desired conclusion for $m-n=k$ where $k\ge 1$ is a fixed integer. Now choose $\beta=e_1e_2\cdots e_m w_1 w_2\cdots w_n$ such that $m-n=k+1$. Following the same approach as taken in the base case, we can write $\beta=\displaystyle{\sum_{w\in p^{-1}(q(w_n))}} \beta_w w^*$ and this time $\beta_w \beta_w^*\in \Phi(\mathcal{D}(\Delta_T))$ by the induction hypothesis. As a consequence, $\beta\beta^*=\displaystyle{\sum_{w\in p^{-1}(q(w_n))}} \beta_w \beta_w^*\in \Phi(\mathcal{D}(\Delta_T))$. 

\emph{Case-III:} $n-m\ge 1$. This case is similar to \emph{Case-II} by considering the later representation of $\beta$ in $(\ref{eq form of beta})$ and using induction on $n-m$. 

Considering all the above cases, we have $\Phi(\mathcal{D}(\Delta_T))=\mathfrak{D}(T)$. 
\end{proof}

\section{A groupoid model for the textile algebra}\label{sec groupoid of TS}
We associate, with each $LR$-textile system $T$, a certain groupoid which can be realized as the two-dimensional analogue of the graph groupoid of a directed graph as well as the textile system counterpart of the infinite path groupoid of a $2$-graph. Basically, we consider the \emph{Deaconu--Renault} groupoid associated to the canonical action of $\mathbb{N}^2$ on the first quadrant two-dimensional shift space $\mathsf{X}_T^+$. 

Following the notations described in Section \ref{sec preliminaries}, we set some more notations consistent with textile shifts. Let $T=(F,E,p,q)$ be any textile system. For each $v\in F^0$, $e\in E^1$ and $w\in E^0$, define \[\Zz_j(v):=\{x\in \mathsf{X}_T^+~|~s_F(x_{(0,j)})=v\},~~\Zz_i(e):=\{x\in \mathsf{X}_T^+~|~p(x_{(i,0)})=e\}\] and \[\Zz_{(i,j)}(w):=\{x\in \Xs_T^+~|~s_E(p(x_{(i,j)}))=w\}\] for all $i,j\in \mathbb{N}$. Notice that, \[\Zz_j(v)=\displaystyle{\bigsqcup_{\omega\in \mathcal{B}_{(1,j+1)}^v(\Xs_T^+)}} \Zz(\omega),\] where $\mathcal{B}_{(1,j+1)}^v(\Xs_T^+)$ is the set of all $1\times (j+1)$-allowed blocks $\omega$ in $\Xs_T^+$ such that $s_F(\omega_j)=v$. Thus, $\Zz_j(v)$ is open in $\Xs_T^+$. Similar decompositions show that $\Zz_i(e)$, $\Zz_{(i,j)}(w)$ are also open in $\Xs_T^+$ for all $e\in E^1$, $w\in E^0$ and $i,j\in \mathbb{N}$. We denote $\Zz_0(v)$, $\Zz_0(e)$ and $\Zz_{(i,j)}(w)$ simply as $\Zz(v)$, $\Zz(e)$ and $\Zz(w)$ respectively. 

The following proposition sets the stage for the formation of our desired groupoid. Recall that the map $\sigma:\mathbf{n}\longmapsto \sigma_\mathbf{n}$ gives a continuous action of $\mathbb{N}^2$ on $\mathsf{X}_T^+$. 
\begin{prop}\label{pro the LR-property gives a DR system}
Suppose $T=(F,E,p,q)$ is an $LR$-textile system where $F,E$ are finite graphs. Then $(\mathsf{X}_T^+,\sigma)$ forms a rank-$2$ Deaconu--Renault system.     
\end{prop}
\begin{proof}
Since the alphabet $F^1$ is finite, $\Xs_T^+$ is a compact Hausdorff space. We show that $\sigma_\mathbf{n}$ is a local homeomorphism for each $\mathbf{n}\in \mathbb{N}^2$. Since composition of local homeomorphisms is a local homeomorphism, it suffices to prove only for $\mathbf{n}=\ea,\eb$. Let $x\in \Xs_T^+$ and $f=x_{(0,0)}$. Then $\Zz(f)$ is an open neighbourhood of $x$. We observe that $\sigma_{\ea}(\Zz(f))=\Zz(r_F(f))$. The inclusion $\sigma_{\ea}(\Zz(f)) \subseteq \Zz(r_F(f))$ is obvious. Let $y\in \Zz(r_F(f))$. Now, \[p(s_F(y_{(0,1)}))=s_E(p(y_{(0,1)}))=s_E(q(y_{(0,0)}))=q(r_F(f))=r_E(q(f)).\] Since, $p$ has unique $r$-path lifting, there is a unique $f_1\in F^1$ such that $p(f_1)=q(f)$ and $r_F(f_1)=s_F(y_{(0,1)})$. Applying the unique $r$-path lifting of $p$, we can inductively construct a sequence $(f_i)_{i\ge 1}$ such that $r_F(f_i)=s_F(y_{(0,i)})$ and $p(f_i)=q(f_{i-1})$ with $f_0=f$. Clearly, this gives an array $\overline{y}\in \Xs_T^+$ with $\overline{y}_\mathbf{n}=y_{\mathbf{n}-\ea}$ for all $\mathbf{n}\ge \ea$ and $\overline{y}_{(0,i)}=f_i$ for all $i\ge 0$. Then, $\overline{y}\in \Zz(f)$ and $\sigma_{\ea}(\overline{y})=y$, whence, the other inclusion follows. This shows that $\sigma_{\ea}(\Zz(f))$ is an open set in $\Xs_T^+$. A similar type of argument, using the $r$-path lifting of $p$ as required, shows that $\sigma_{\ea}(\Zz(\omega))$ is open for all $\omega\in \mathcal{B}(\Xs_T^+)$. For instance, if $\omega\in \mathcal{B}_{(1,n+1)}(\Xs_T^+)$ for some $n\ge 0$, then \[\sigma_{\ea}(\Zz(\omega))=\displaystyle{\bigcap_{i=0}^{n}}~\Zz_i(r_F(\omega_i))\] and since each set on the right-hand side is open, $\sigma_{\ea}(\Zz(\omega))$ is open too. Therefore, $\sigma_{\ea}$ is an open map. Finally, we show that $\sigma_{\ea}|_{\Zz(f)}$ is injective. Suppose there exist $x,y\in \Zz(f)$ such that $\sigma_{\ea}(x)=\sigma_{\ea}(y)$. Then, $x_{(0,0)}=f=y_{(0,0)}$ and $x_{(i,j)}=y_{(i,j)}$ for all $i\ge 1$, $j\ge 0$. An inductive application of the diagonal property, which $\Xs_T^+$ possess by Lemma \ref{lem LR textile has the diagonal property}, now shows that $x_{(0,j)}=y_{(0,j)}$ for all $j\ge 1$ and consequently, $x=y$. Therefore, $\sigma_{\ea}|_{\Zz(f)}:\Zz(f)\longrightarrow \Zz(r_F(f))$ is a continuous, open bijection and hence a homeomorphism. Following the exact same argument using the unique $s$-path lifting of $q$, we can show that \[\sigma_{\eb}|_{\Zz(f)}:\Zz(f)\longrightarrow \Zz(q(f))\] is a homeomorphism. Hence, $\sigma_\mathbf{n}$ is a local homeomorphism for all $\mathbf{n}\in \mathbb{N}^2$ and subsequently, $(X,\sigma)$ is a rank-$2$ Deaconu--Renault system. 
\end{proof}
\begin{dfn}\label{def the textile groupoid}
Let $T$ be an $LR$-textile system. The \emph{textile groupoid} $\mathcal{G}_T$ of $T$ is defined to be the Deaconu--Renault groupoid associated with $(\Xs_T^+,\sigma)$.    
\end{dfn}
In view of \cite[Lemma 3.1]{Sims}, $\mathcal{G}_T$ is a locally compact Hausdorff \'{e}tale groupoid with a basis consists of sets of the form \[\Zz(U,\mathbf{m},\mathbf{n},V):=\{(x,\mathbf{m}-\mathbf{n},y)~|~x\in U,y\in V~\text{and}~\sigma_\M (x)=\sigma_\N (y)\},\] where $U,V$ are open sets in $\Xs_T^+$ and $\M,\N\in \mathbb{N}^2$ are such that $\sigma_\M|_U$, $\sigma_\N|_V$ are homeomorphisms onto a common open set in $\Xs_T^+$. It is naturally $\mathbb{Z}^2$-graded via the continuous cocycle $c:\mathcal{G}_T\longrightarrow \mathbb{Z}^2$; $c((x,\M-\N,y)):=\M-\N$. As the unit space $\Xs_T^+$ is locally compact, Hausdorff and zero-dimensional, it is totally disconnected. Therefore, $\mathcal{G}_T$ is an ample groupoid. We can, of course, exhibit a basis consisting of compact open bisections. Indeed, if $\M=(m_1,m_2),\N=(n_1,n_2)\in \mathbb{N}^2$, $\alpha\in \mathcal{B}_{(m_1+1,m_2+1)}(\Xs_T^+)$, $\beta\in \mathcal{B}_{(n_1+1,n_2+1)}(\Xs_T^+)$ and $\sigma_\N(\Zz(\alpha))=\sigma_\M(\Zz(\beta))$, then the open set \[\Zz(\alpha,\beta):=\Zz(\Zz(\alpha),\M,\N,\Zz(\beta))\] is compact in $\mathcal{G}_T$ as $\Zz(\alpha)$, $\Zz(\beta)$ are compact in $\Xs_T^+$ (see the proof of \cite[Lemma 3.1]{Sims}). Also, it is a bisection which is not difficult to establish using the diagonal property of $\Xs_T^+$. The sets of the above form are the members of the desired basis for the ample topology of $\mathcal{G}_T$. 

Given a $2$-graph $\Lambda$ with finite edge set, the shift spaces $\Lambda^\infty$ and $\Xs_{T_\Lambda}^+$ are shown to be conjugate in \cite[Lemma 4.11]{BGGLPP}. We now use the homeomorphism, giving the conjugacy to show that the textile groupoid is isomorphic to the infinite path groupoid of the associated $2$-graph.
\begin{prop}\label{pro textile groupoid is isomorphic to path groupoid}
Let $T$ be an $LR$-textile system such that $F,E$ are finite graphs, $E$ is sink-free and $p$ is surjective. Let $\Lambda_T$ be the associated $2$-graph. Then there is a cocycle preserving isomorphism between the groupoids $\mathcal{G}_T$ and $\mathcal{G}_{\Lambda_T}$.
\end{prop}
\begin{proof}
The conjugacy established in \cite[Lemma 4.11]{BGGLPP}, with $\Lambda$ replaced by $\Lambda_T$ gives a homeomorphism on the level of unit spaces of the respective groupoids. However, for the readers' convenience, we here reproduce the map and then extend this to the whole groupoid. Consider the map $\phi:\Xs_T^+\longrightarrow \Lambda_T^\infty$ defined by 
\[\phi(x)((\M,\N)):=
	\left\{
	\begin{array}{llll}
		x_{[\M,\N-(1,1)]}  & \mbox{if } \N-\M\ge (1,1), \\
		p(x_\M) & \mbox{if }  \N=\M+\ea,\\
        s_F(x_\M) & \mbox{if }  \N=\M+\eb,\\
         s_E(p(x_\M)) & \mbox{if } \N=\M,\\
	\end{array}
	\right.
\] for all $x\in \Xs_T^+$ and $(\M,\N)\in \Omega_2$. On the other hand, consider $\psi:\Lambda_T^\infty\longrightarrow \Xs_T^+$ defined by \[\psi(x)_\M:=x(\M,\M+(1,1)),\] for all $x\in \Lambda_T^\infty$ and $\M\in \mathbb{N}^2$. Then it is easy to observe that $\phi$ and $\psi$ are mutually inverse maps taking cylinder sets to cylinder sets. Thus, $\phi$ is a homeomorphism between $\mathcal{G}_T^{(0)}$ and $\mathcal{G}_{\Lambda_T}^{(0)}$. Now, define the map 
\begin{align*}
    \Phi:\mathcal{G}_{T}&\longrightarrow \mathcal{G}_{\Lambda_T}\\
    (x,\M-\N,y)&\longmapsto (\phi(x),\M-\N,\phi(y))
\end{align*} 
for all $(x,\M-\N,y)\in \mathcal{G}_T$. Since, $\phi$ is shift commuting, it follows that $\Phi$ is well-defined. Using the inverse $\psi$, we can show that $\Phi$ is a bijection. To show that it is an open map, take any basic open set $\Zz(U,\M,\N,V)$ in $\mathcal{G}_T$. Since $\phi$ is a homeomorphism, $\phi(U),\phi(V)$ are open sets in $\Lambda_T^\infty$. Now, $\sigma_{\Lambda_T}^\M|_{\phi(U)}=\phi\circ \sigma_\M|_U\circ \phi^{-1}$ and $\sigma_{\Lambda_T}^\N|_{\phi(V)}=\phi\circ \sigma_\N|_V\circ \phi^{-1}$, from which it follows that $\sigma_{\Lambda_T}^\M|_{\phi(U)}$ and $\sigma_{\Lambda_T}^\N|_{\phi(V)}$ are injective. Again since, $\sigma_\M(U)=\sigma_\N(V)$ and $\phi$ is shift commuting, so $\sigma_{\Lambda_T}^\M(\phi(U))=\sigma_{\Lambda_T}^\N(\phi(V))$. Now, \cite[Lemma 2.4]{Rout} implies that $\Phi(\Zz(U,\M,\N,V))=Z(\phi(U),\M,\N,\phi(V))$ is an open set in $\mathcal{G}_{\Lambda_T}$. Reversing the arguments, it can be shown that $\Phi$ is continuous and hence, a homeomorphism. By definition, $\Phi$ preserves the composition and commutes with the respective cocycles. Hence, the result follows.
\end{proof}
As an immediate consequence of the above proposition, we obtain the following result.
\begin{cor}\label{cor textile algebra as Steinberg algebra}
Let $T$ be an $LR$-textile system such that $F,E$ are finite graphs, $E$ is sink-free and $p$ is surjective. Let $\mathsf{k}$ be any field. Then the textile $\mathsf{k}$-algebra $\mathbf{A}(T)$ is isomorphic to the Steinberg algebra $A_\mathsf{k}(\mathcal{G}_T)$ as $\mathbb{Z}^2$-graded algebras.
\end{cor}
\begin{proof}
We have the following chain of $\mathbb{Z}^2$-graded $\mathsf{k}$-algebra isomorphisms \[\mathbf{A}(T)\cong \KP_\mathsf{k}(\Lambda_T)\cong A_\mathsf{k}(\mathcal{G}_{\Lambda_T})\cong A_\mathsf{k}(\mathcal{G}_T),\] where the first isomorphism is discussed in Remark \ref{rem connecting with other algebras} $(i)$, the second one is due to \cite[Proposition 5.4]{Clark}, and the last one follows from Proposition \ref{pro textile groupoid is isomorphic to path groupoid}. 
\end{proof}
\section{Three-dimensional textile systems: a blueprint}\label{sec 3D textiles}
In this section, we propose to give a model of a textile system in three dimensions. The motivation for this construction comes from the relationship between $2$-graphs and $LR$-textile systems. We carefully observe this connection and then extend this naturally to fit nicely in three dimensions. 

Let $T=(F,E,p,q)$ be an $LR$-textile system. The taxonomy in the following diagram helps to recall the roles of the textile components in the corresponding $2$-graph. 
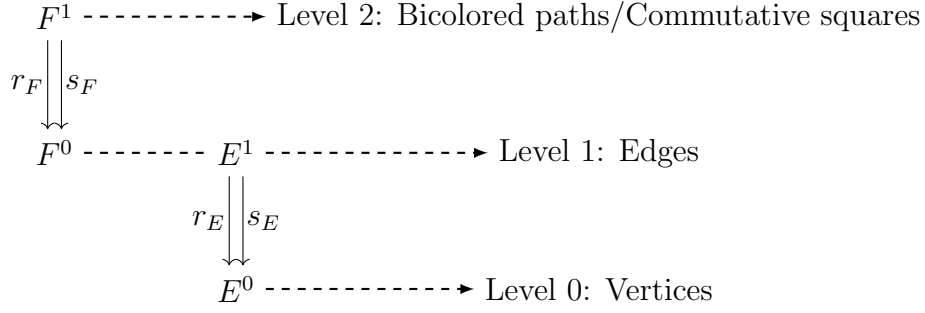
\[
\begin{tikzpicture}[scale=1.2]
\node[] (A) at (0,1.5) {$F^1$};
\node[] (B) at (0,0) {$F^0$};	
\node[] (C) at (2,0) {$E^1$};
\node[] (D) at (2,-1.5) {$E^0$};

\node[] (E) at (6,-1.5) {\text{Level 0: Vertices}};
\node[] (F) at (6,0) {Level 1: Edges};
\node[] (G) at (6,1.5) {Level 2: Bicolored paths/Commutative squares};

\draw[transform canvas={xshift=0.5ex}, ->] (A) -- (B);
\draw[transform canvas={xshift=-0.5ex}, ->] (A) -- (B);

\draw[transform canvas={xshift=0.5ex}, ->] (C) -- (D);
\draw[transform canvas={xshift=-0.5ex}, ->] (C) -- (D);

\draw[dashed, thick] (B) -- (C);

\path[->, dashed, >=latex,thick] (A) edge [left] node {} (G);
\path[->, dashed, >=latex,thick] (C) edge [left] node {} (F);
\path[->, dashed, >=latex,thick] (D) edge [left] node {} (E);

\node at (-0.3,0.75) {$r_F$};
\node at (0.3,0.75) {$s_F$};

\node at (1.7,-0.75) {$r_E$};
\node at (2.3,-0.75) {$s_E$};
\end{tikzpicture}
\] 
Evidently, we need to add one more level of tricolored paths if we wish to have the above type of taxonomy in three dimensions. Also, there are three types of bicolored paths in the skeleton of a $3$-graph. These naturally suggest that we should accommodate four directed graphs $G,F_1,F_2,E$ in a plausible three-dimensional textile system, so that the components can be divided into four levels as described in the following diagram:

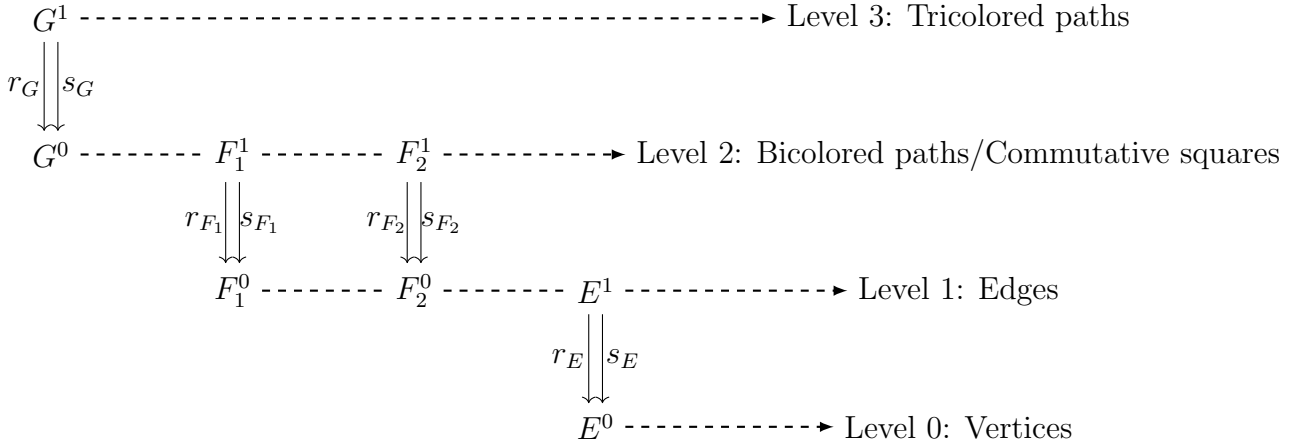
\[
\begin{tikzpicture}[scale=1.2]
\node[] (A) at (0,1.5) {$F_2^1$};
\node[] (B) at (0,0) {$F_2^0$};	
\node[] (C) at (2,0) {$E^1$};
\node[] (D) at (2,-1.5) {$E^0$};
\node[] (E) at (-4,3) {$G^1$};
\node[] (F) at (-4,1.5) {$G^0$};	
\node[] (G) at (-2,1.5) {$F_1^1$};
\node[] (H) at (-2,0) {$F_1^0$};

\node[] (X) at (6,-1.5) {Level 0: Vertices};
\node[] (Y) at (6,0) {Level 1: Edges};
\node[] (Z) at (6,1.5) {Level 2: Bicolored paths/Commutative squares};
\node[] (W) at (6,3) {Level 3: Tricolored paths};

\draw[transform canvas={xshift=0.5ex}, ->] (A) -- (B);
\draw[transform canvas={xshift=-0.5ex}, ->] (A) -- (B);

\draw[transform canvas={xshift=0.5ex}, ->] (C) -- (D);
\draw[transform canvas={xshift=-0.5ex}, ->] (C) -- (D);

\draw[transform canvas={xshift=0.5ex}, ->] (E) -- (F);
\draw[transform canvas={xshift=-0.5ex}, ->] (E) -- (F);

\draw[transform canvas={xshift=0.5ex}, ->] (G) -- (H);
\draw[transform canvas={xshift=-0.5ex}, ->] (G) -- (H);

\draw[dashed, thick] (H) -- (B);
\draw[dashed, thick] (B) -- (C);
\draw[dashed, thick] (F) -- (G);
\draw[dashed, thick] (G) -- (A);

\path[->, dashed, >=latex,thick] (E) edge [left] node {} (W);
\path[->, dashed, >=latex,thick] (A) edge [left] node {} (Z);
\path[->, dashed, >=latex,thick] (C) edge [left] node {} (Y);
\path[->, dashed, >=latex,thick] (D) edge [left] node {} (X);

\node at (-4.3,2.25) {$r_G$};
\node at (-3.7,2.25) {$s_G$};

\node at (-2.3,0.75) {$r_{F_1}$};
\node at (-1.7,0.75) {$s_{F_1}$};

\node at (-0.3,0.75) {$r_{F_2}$};
\node at (0.3,0.75) {$s_{F_2}$};

\node at (1.7,-0.75) {$r_E$};
\node at (2.3,-0.75) {$s_E$};
\end{tikzpicture}
\] 
We now present a formal definition.
\begin{dfn}\label{def 3D textile}
A three-dimensional textile system $\mathcal{T}$ is given by the following diagram:
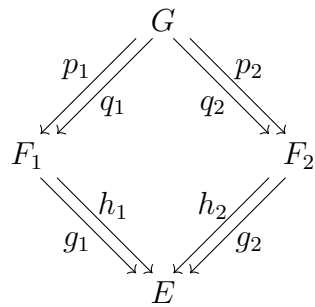
\[
\begin{tikzpicture}[scale=1.2]
\node[inner sep=2.5pt, circle] (B) at (-1.5,0) {$F_1$};
\node[inner sep=2.5pt, circle] (C) at (1.5,0) {$F_2$};	
\node[inner sep=2.5pt, circle] (D) at (0,-1.5) {$E$};
\node[inner sep=2.5pt, circle] (A) at (0,1.5) {$G$}; 

\draw[transform canvas={xshift=0.6ex}, ->] (A) -- (B);
\draw[transform canvas={xshift=-0.6ex}, ->] (A) -- (B);

\draw[transform canvas={xshift=0.6ex}, ->] (A) -- (C);
\draw[transform canvas={xshift=-0.6ex}, ->] (A) -- (C);

\draw[transform canvas={xshift=0.6ex}, ->] (C) -- (D);
\draw[transform canvas={xshift=-0.6ex}, ->] (C) -- (D);

\draw[transform canvas={xshift=0.6ex}, ->] (B) -- (D);
\draw[transform canvas={xshift=-0.6ex}, ->] (B) -- (D);

\node at (-0.55,0.55) {$q_1$};
\node at (-0.95,0.95) {$p_1$};

\node at (0.55,0.55) {$q_2$};
\node at (0.95,0.95) {$p_2$};

\node at (-0.55,-0.55) {$h_1$};
\node at (-0.95,-0.95) {$g_1$};

\node at (0.55,-0.55) {$h_2$};
\node at (0.95,-0.95) {$g_2$};
\end{tikzpicture}
\]
where $G,F_1,F_2,E$ are directed graphs and the arrows represent graph homomorphisms such that

$(\mathcal{C})$ the following commutativity hold:
\begin{align*}
g_1\circ p_1&=g_2\circ p_2~\text{(Outer-commutativity)};\\
h_1\circ q_1&=h_2\circ q_2~\text{(Inner-commutativity)};\\
g_1\circ q_1&=h_2\circ p_2~\text{(Cross-commutativity 1)};\\
h_1\circ p_1&=g_2\circ q_2~\text{(Cross-commutativity 2)};
\end{align*}
and

$(\mathcal{I})$ the following maps are injective:

\begin{align*}
    s_G\times r_G\times p_1^1\times q_1^1\times p_2^1\times q_2^1: G^1 &\longrightarrow G^0\times G^0\times F_1^1\times F_1^1\times F_2^1\times F_2^1\\
    p_1^0\times q_1^0\times p_2^0\times q_2^0: G^0 &\longrightarrow F_1^0\times F_1^0\times F_2^0\times F_2^0\\
    s_{F_1}\times r_{F_1}\times g_1^1\times h_1^1: F_1^1 &\longrightarrow F_1^0\times F_1^0\times E^1\times E^1\\
    s_{F_2}\times r_{F_2}\times g_2^1\times h_2^1: F_2^1 &\longrightarrow F_2^0\times F_2^0\times E^1\times E^1.
\end{align*}
\end{dfn}

The injectivity of the first map in $(\mathcal{I})$ says that we can represent any edge $x\in G^1$ uniquely by a cube of the form

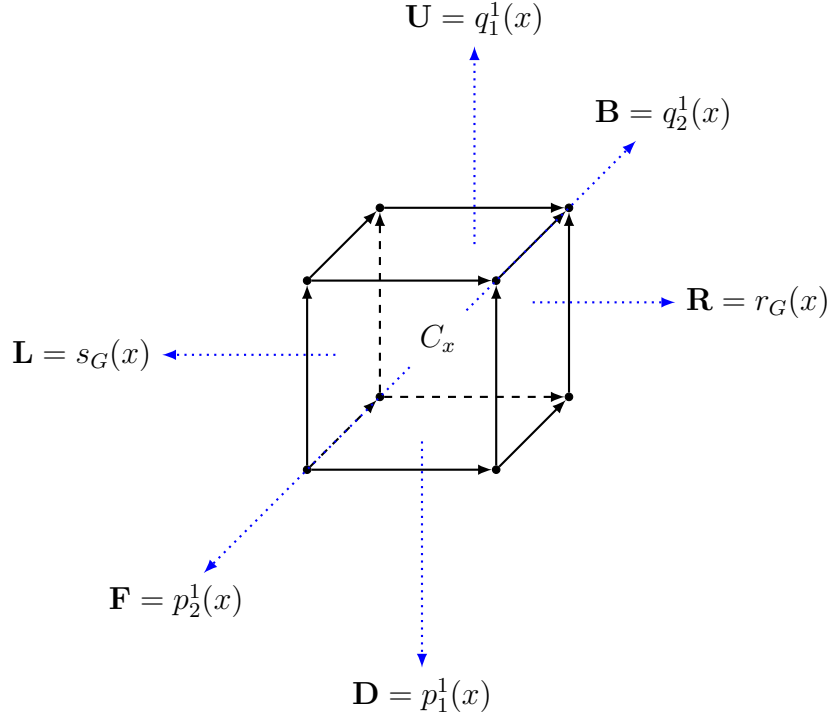
\[
\begin{tikzpicture}[scale=2.5]
\node[inner sep=1pt, circle,draw,fill=black] (A1) at (0,0,0) {};
\node[inner sep=1pt, circle,draw,fill=black] (A2) at (1,0,0) {};
\node[inner sep=1pt, circle,draw,fill=black] (A3) at (1,0,1) {};
\node[inner sep=1pt, circle,draw,fill=black] (A4) at (0,0,1) {};

\node[inner sep=1pt, circle,draw,fill=black] (B1) at (0,1,0) {};
\node[inner sep=1pt, circle,draw,fill=black] (B2) at (1,1,0) {};
\node[inner sep=1pt, circle,draw,fill=black] (B3) at (1,1,1) {};
\node[inner sep=1pt, circle,draw,fill=black] (B4) at (0,1,1) {};

\path[->, black, >=latex,thick] (B4) edge [] node[]{} (B1);
\path[->, black, dashed, >=latex,thick] (A4) edge [] node[]{} (A1);
\path[->, black, >=latex,thick] (B3) edge [] node[]{} (B2);
\path[->, black, >=latex,thick] (A3) edge [] node[]{} (A2);

\path[->,black, dashed, >=latex,thick] (A1) edge [] node[]{} (B1);
\path[->,black, >=latex,thick] (A4) edge [] node[]{} (B4);
\path[->,black, >=latex,thick] (A2) edge [] node[]{} (B2);
\path[->,black, >=latex,thick] (A3) edge [] node[]{} (B3);

\path[->, black, >=latex,thick] (B1) edge [] node {} (B2);
\path[->, black, >=latex,thick] (B4) edge [] node {} (B3);
\path[->,black, dashed, >=latex,thick] (A1) edge [] node {} (A2);
\path[->, black, >=latex,thick] (A4) edge [] node {} (A3);

\node[] (L) at (-1,0.8,1.5) {$\mathbf{L}=s_G(x)$};
\node[] (R) at (2,0.5,0) {$\mathbf{R}=r_G(x)$};

\node[] (D) at (0.8,-1,1.5) {$\mathbf{D}=p_1^1(x)$};
\node[] (U) at (0.5,2,0) {$\mathbf{U}=q_1^1(x)$};

\node[] (F) at (-0.5,-0.5,1.5) {$\mathbf{F}=p_2^1(x)$};
\node[] (B) at (1.5,1.5,0) {$\mathbf{B}=q_2^1(x)$};

\node[] (A) at (0.4,0.8,1.5) {};
\node[] (K) at (0.75,0.5,0) {};

\node[] (C) at (0.8,0.4,1.5) {};
\node[] (G) at (0.5,0.75,0) {};

\node[] (E) at (0.6,0.6,1) {};
\node[] (H) at (0.4,0.4,0) {};

\node at (0.5,0.5,0.5) {$C_x$};

\path[->,blue, dotted, >=latex,thick] (A) edge [] node[]{} (L);
\path[->,blue, dotted, >=latex,thick] (K) edge [] node[]{} (R);
\path[->,blue, dotted, >=latex,thick] (C) edge [] node[]{} (D);
\path[->,blue, dotted, >=latex,thick] (G) edge [] node[]{} (U);
\path[->,blue, dotted, >=latex,thick] (E) edge [] node[]{} (F);
\path[->,blue, dotted, >=latex,thick] (H) edge [] node[]{} (B);
\end{tikzpicture}
\]
whereas $u\in G^0$, $f\in F_1^1$, $f'\in F_2^1$ are uniquely represented by the tiles 

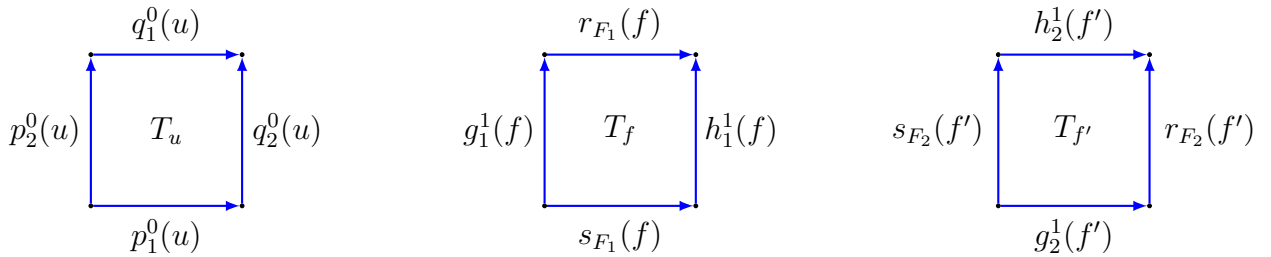
\[
\begin{tikzpicture}[scale=2.0]
\node[inner sep=0.5pt, circle,draw,fill=black] (A1) at (0,0) {};
\node[inner sep=0.5pt, circle,draw,fill=black] (A2) at (1,0) {};
\node[inner sep=0.5pt, circle,draw,fill=black] (A3) at (0,1) {};
\node[inner sep=0.5pt, circle,draw,fill=black] (A4) at (1,1) {};

\path[->,blue, >=latex,thick] (A1) edge [] node[]{} (A2);
\path[->,blue, >=latex,thick] (A1) edge [] node[]{} (A3);
\path[->,blue, >=latex,thick] (A2) edge [] node[]{} (A4);
\path[->,blue, >=latex,thick] (A3) edge [] node[]{} (A4);

\node at (0.5,-0.2) {$p_1^0(u)$};
\node at (0.5,1.2) {$q_1^0(u)$};
\node at (-0.3,0.5) {$p_2^0(u)$};
\node at (1.3,0.5) {$q_2^0(u)$};

\node at (0.5,0.5) {$T_u$};

\node[inner sep=0.5pt, circle,draw,fill=black] (B1) at (3,0) {};
\node[inner sep=0.5pt, circle,draw,fill=black] (B2) at (4,0) {};
\node[inner sep=0.5pt, circle,draw,fill=black] (B3) at (3,1) {};
\node[inner sep=0.5pt, circle,draw,fill=black] (B4) at (4,1) {};

\path[->,blue, >=latex,thick] (B1) edge [] node[]{} (B2);
\path[->,blue, >=latex,thick] (B1) edge [] node[]{} (B3);
\path[->,blue, >=latex,thick] (B2) edge [] node[]{} (B4);
\path[->,blue, >=latex,thick] (B3) edge [] node[]{} (B4);

\node at (3.5,-0.2) {$s_{F_1}(f)$};
\node at (3.5,1.2) {$r_{F_1}(f)$};
\node at (2.7,0.5) {$g_1^1(f)$};
\node at (4.3,0.5) {$h_1^1(f)$};

\node at (3.5,0.5) {$T_f$};

\node[inner sep=0.5pt, circle,draw,fill=black] (C1) at (6,0) {};
\node[inner sep=0.5pt, circle,draw,fill=black] (C2) at (7,0) {};
\node[inner sep=0.5pt, circle,draw,fill=black] (C3) at (6,1) {};
\node[inner sep=0.5pt, circle,draw,fill=black] (C4) at (7,1) {};

\path[->,blue, >=latex,thick] (C1) edge [] node[]{} (C2);
\path[->,blue, >=latex,thick] (C1) edge [] node[]{} (C3);
\path[->,blue, >=latex,thick] (C2) edge [] node[]{} (C4);
\path[->,blue, >=latex,thick] (C3) edge [] node[]{} (C4);

\node at (6.5,-0.2) {$g_2^1(f')$};
\node at (6.5,1.2) {$h_2^1(f')$};
\node at (5.6,0.5) {$s_{F_2}(f')$};
\node at (7.4,0.5) {$r_{F_2}(f')$};

\node at (6.5,0.5) {$T_{f'}$};

\end{tikzpicture}
\]
in view of the injectivity of the remaining maps in $(\mathcal{I})$. 

\begin{rmks}\label{rem vieweing 2D textile as a 3D textile}
$(i)$ One can identify three two-dimensional textile systems inside a three-dimensional textile $\mathcal{T}$. The injectivity of the last two maps in $(\mathcal{I})$ immediately tells that $T_1:=(F_1,E,g_1,h_1)$ and $T_2:=(F_2,E,g_2,h_2)$ are textile systems. The third one, although not apparent, is not hard to find. Consider the directed graphs $F:=(F_1^0,G^0,q_1^0,p_1^0)$ and $H:=(E^0,F_2^0,h_2^0,g_2^0)$. By $(\mathcal{C})$, the maps $p:=(g_1^0,p_2^0)$ and $q:=(h_1^0,q_2^0)$ are then graph homomorphisms from $F$ to $H$. Note that $T_3=(F,H,p,q)$ is a textile system since $p_1^0\times q_1^0\times p_2^0\times q_2^0$ is injective. We call $T_1,T_2,T_3$, the $2D$-projections of $\mathcal{T}$.

$(ii)$ Any directed graph $E$ can be seen as a two-dimensional textile system (see Remark \ref{rem LPA as A(T)}). In a similar manner, one can realize a two-dimensional textile system $T=(F,E,p,q)$ as a three-dimensional textile: consider the diagram 

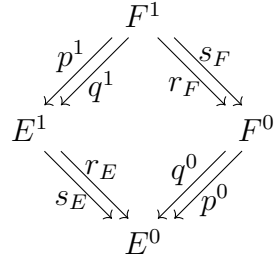
\[
\begin{tikzpicture}[scale=1]
\node[inner sep=2.5pt, circle] (B) at (-1.5,0) {$E^1$};
\node[inner sep=2.5pt, circle] (C) at (1.5,0) {$F^0$};	
\node[inner sep=2.5pt, circle] (D) at (0,-1.5) {$E^0$};
\node[inner sep=2.5pt, circle] (A) at (0,1.5) {$F^1$}; 

\draw[transform canvas={xshift=0.6ex}, ->] (A) -- (B);
\draw[transform canvas={xshift=-0.6ex}, ->] (A) -- (B);

\draw[transform canvas={xshift=0.6ex}, ->] (A) -- (C);
\draw[transform canvas={xshift=-0.6ex}, ->] (A) -- (C);

\draw[transform canvas={xshift=0.6ex}, ->] (C) -- (D);
\draw[transform canvas={xshift=-0.6ex}, ->] (C) -- (D);

\draw[transform canvas={xshift=0.6ex}, ->] (B) -- (D);
\draw[transform canvas={xshift=-0.6ex}, ->] (B) -- (D);

\node at (-0.55,0.55) {$q^1$};
\node at (-0.95,0.95) {$p^1$};

\node at (0.55,0.55) {$r_F$};
\node at (0.95,0.95) {$s_F$};

\node at (-0.55,-0.55) {$r_E$};
\node at (-0.95,-0.95) {$s_E$};

\node at (0.55,-0.55) {$q^0$};
\node at (0.95,-0.95) {$p^0$};
\end{tikzpicture}
\]
where $F^1,E^1,F^0,E^0$ are viewed as $0$-graphs. The four types of commutativity in $(\mathcal{C})$ of Definition \ref{def 3D textile} follow from the fact that $p,q$ are graph homomorphisms. Note that the maps with domains $(F^1)^1$, $(E^1)^1$ and, $(F^0)^1$ are vacuously injective whereas the injectivity of the map with domain $(F^1)^0=F^1$ is precisely the injectivity of $s_F\times r_F\times p^1\times q^1$, which is guaranteed as $T$ is a two-dimensional textile system. This realization also justifies the definition of a three-dimensional textile system.
\end{rmks}

Similar to two-dimensional textile systems, any three-dimensional textile $\mathcal{T}$ gives rise to a three-dimensional shift space 

\[\mathsf{X}_{\mathcal{T}}:=\{x=(x_\mathbf{n})_{\mathbf{n}\in \mathbb{Z}^3}~|~r_G(x_\mathbf{n})=s_G(x_{\mathbf{n}+\mathbf{e}_3}), q_1^1(x_\mathbf{n})=p_1^1(x_{\mathbf{n}+\mathbf{e_1}}), q_2^1(x_\mathbf{n})=p_2^1(x_{\mathbf{n}+\mathbf{e}_2})\}.\] When $G^1$ is finite, $\Xs_\mathcal{T}$ is a shift of finite type over the alphabet $\mathcal{A}=G^1$. Indeed, it is the matrix shift $\mathsf{X}(A_1,A_2,A_3)$ where $A_1,A_2,A_3$ are transition matrices indexed by $G^1$ and defined as follows:
\[A_1(x,y):=
	\left\{
	\begin{array}{ll}
		1  & \mbox{if } q_1^1(x)=p_1^1(y) \\
        0 & \mbox{otherwise, }
	\end{array}
	\right.~
A_2(x,y):=
	\left\{
	\begin{array}{ll}
		1  & \mbox{if } q_2^1(x)=p_2^1(y) \\
        0 & \mbox{otherwise, }
	\end{array}
	\right.~
A_3(x,y):=
	\left\{
	\begin{array}{ll}
		1  & \mbox{if } r_G(x)=s_G(y) \\
        0 & \mbox{otherwise. }
	\end{array}
	\right.
\]
Also, it is not hard to realize any three-dimensional shift of finite type (up to a higher-block presentation), as the shift space associated with a three-dimensional textile. The following can be seen as analogues of \cite[Theorem 2.3.2]{Lind-Marcus} and \cite[Proposition 2.3]{JM} in three-dimensions. 
\begin{prop}\label{pro realizing 3D SFT as 3-textile SFT}
Let $\mathsf{X}$ be any three-dimensional shift of finite type over the alphabet $\mathcal{A}$. Then there exists a three-dimensional textile system $\mathcal{T}$ such that $\mathsf{X}_\mathcal{T}=\mathsf{X}^{[(2,2,2)]}$. 
\end{prop}
\begin{proof}
For $m,n,p\ge 1$, suppose $\mathcal{B}_{(m,n,p)}(\mathsf{X})$ denotes the set of all allowed $m\times n\times p$-blocks in $\mathsf{X}$. Construct directed graphs $G,F_1,F_2,E$ as follows: \[G^0:=\mathcal{B}_{(2,2,1)}(\mathsf{X}),~G^1:=\mathcal{B}_{(2,2,2)}(\mathsf{X}),~r_G(\gamma):=\mathbf{U}(\gamma),~s_G(\gamma):=\mathbf{D}(\gamma),\] \[F_1^0:=\mathcal{B}_{(1,2,1)}(\mathsf{X}),~F_1^1:=\mathcal{B}_{(1,2,2)}(\mathsf{X}),~r_{F_1}\left(\begin{matrix}
c & d\\
a & b
\end{matrix}\right):=\begin{matrix}
c & d
\end{matrix},~s_{F_1}\left(\begin{matrix}
c & d\\
a & b
\end{matrix}\right):=\begin{matrix}
a & b
\end{matrix},\] 
\[F_2^0:=\mathcal{B}_{(2,1,1)}(\mathsf{X}),~F_2^1:=\mathcal{B}_{(2,1,2)}(\mathsf{X}),~r_{F_2}\left(\begin{matrix}
c & d\\
a & b
\end{matrix}\right):=\begin{matrix}
c & d
\end{matrix},~s_{F_2}\left(\begin{matrix}
c & d\\
a & b
\end{matrix}\right):=\begin{matrix}
a & b
\end{matrix},\] 
\[E^0:=\mathcal{A},~E^1:=\mathcal{B}_{(1,1,2)}(\mathsf{X}),~r_E\left(\begin{matrix}
b\\
a
\end{matrix}\right):=b,~s_E\left(\begin{matrix}
b\\
a
\end{matrix}\right):=a.\] Now define graph morphisms $p_1,q_1:G\longrightarrow F_1$, $p_2,q_2:G\longrightarrow F_2$, $g_1,h_1:F_1\longrightarrow E$, $g_2,h_2:F_2\longrightarrow E$ as follows:
\[p_1:=\left(\begin{matrix}
c & d\\
a & b
\end{matrix}\longmapsto \begin{matrix}
c\\
a
\end{matrix},~\mathbf{B}(-)\right),~~ q_1:=\left(\begin{matrix}
c & d\\
a & b
\end{matrix}\longmapsto \begin{matrix}
d\\
b
\end{matrix},~\mathbf{F}(-)\right),\] 
\[p_2:=\left(\begin{matrix}
c & d\\
a & b
\end{matrix}\longmapsto \begin{matrix}
a & b
\end{matrix},~\mathbf{L}(-)\right),~~ q_2:=\left(\begin{matrix}
c & d\\
a & b
\end{matrix}\longmapsto \begin{matrix}
c & d
\end{matrix},~\mathbf{R}(-)\right),\]
\[g_1:=\left(\begin{matrix}
a & b
\end{matrix}\longmapsto \begin{matrix}
a
\end{matrix},~\begin{matrix}
c & d\\
a & b
\end{matrix}\longmapsto \begin{matrix}
c\\
a
\end{matrix}\right),~~ h_1:=\left(\begin{matrix}
a & b
\end{matrix}\longmapsto \begin{matrix}
b
\end{matrix},~\begin{matrix}
c & d\\
a & b
\end{matrix}\longmapsto \begin{matrix}
d\\
b
\end{matrix}\right),\]
\[g_2:=\left(\begin{matrix}
a & b
\end{matrix}\longmapsto \begin{matrix}
a
\end{matrix},~\begin{matrix}
c & d\\
a & b
\end{matrix}\longmapsto \begin{matrix}
c\\
a
\end{matrix}\right),~~ h_2:=\left(\begin{matrix}
a & b
\end{matrix}\longmapsto \begin{matrix}
b
\end{matrix},~\begin{matrix}
c & d\\
a & b
\end{matrix}\longmapsto \begin{matrix}
d\\
b
\end{matrix}\right).\] A routine verification yields that the diagram 
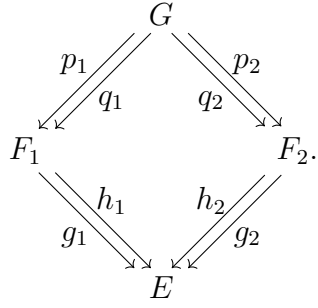
\[
\begin{tikzpicture}[scale=1.2]
\node[inner sep=2.5pt, circle] (B) at (-1.5,0) {$F_1$};
\node[inner sep=2.5pt, circle] (C) at (1.5,0) {$F_2.$};	
\node[inner sep=2.5pt, circle] (D) at (0,-1.5) {$E$};
\node[inner sep=2.5pt, circle] (A) at (0,1.5) {$G$}; 

\draw[transform canvas={xshift=0.6ex}, ->] (A) -- (B);
\draw[transform canvas={xshift=-0.6ex}, ->] (A) -- (B);

\draw[transform canvas={xshift=0.6ex}, ->] (A) -- (C);
\draw[transform canvas={xshift=-0.6ex}, ->] (A) -- (C);

\draw[transform canvas={xshift=0.6ex}, ->] (C) -- (D);
\draw[transform canvas={xshift=-0.6ex}, ->] (C) -- (D);

\draw[transform canvas={xshift=0.6ex}, ->] (B) -- (D);
\draw[transform canvas={xshift=-0.6ex}, ->] (B) -- (D);

\node at (-0.55,0.55) {$q_1$};
\node at (-0.95,0.95) {$p_1$};

\node at (0.55,0.55) {$q_2$};
\node at (0.95,0.95) {$p_2$};

\node at (-0.55,-0.55) {$h_1$};
\node at (-0.95,-0.95) {$g_1$};

\node at (0.55,-0.55) {$h_2$};
\node at (0.95,-0.95) {$g_2$};
\end{tikzpicture}
\] defines a three-dimensional textile system $\mathcal{T}$ and $\mathsf{X}_\mathcal{T}=\mathsf{X}^{[(2,2,2)]}$. 
\end{proof}

Now we present a categorical interpretation of three-dimensional textile systems. This helps us to realize the construction of textile systems from a unifying point of view and may be helpful for further generalization. Following the notation in \cite{Maclane}, we denote by $\downdownarrows$, the category with only two objects, namely, $0$ and $1$ and two parallel non-identity arrows directed from $1$ to $0$ as shown below: 

\[
\begin{tikzpicture}[scale=1]
\node[] (A) at (0,0) {$1$};
\node[] (B) at (2,0) {$0$.};

\draw[transform canvas={yshift=0.5ex}, ->] (A) -- (B);
\draw[transform canvas={yshift=-0.5ex}, ->] (A) -- (B);

\end{tikzpicture}
\] 

It is well known that the category of directed graphs \textbf{Dgr} is nothing but the functor category \textbf{Set}$^\downdownarrows$, i.e., any directed graph can be viewed as a functor from $\downdownarrows$ to \textbf{Set}, and any graph homomorphism between two directed graphs is a natural transformation between the corresponding functors. A similar type of description (for which we do not find any reference in the literature) can be given for two-dimensional textile systems. It is straightforward to observe that any textile system $T=(F,E,p,q)$ can be realized as a functor \[T:\downdownarrows\longrightarrow \textbf{Dgr},\] where $F=T(1)$, $E=T(0)$ and $p,q$ are images of the parallel arrows. Also, a morphism between two textiles $T_1=(F_1,E_1,p_1,q_1)$ and $T_2=(F_2,E_2,p_2,q_2)$ which is defined in \cite{KPW} as a pair of graph morphisms $\phi=(\phi_F,\phi_E)$, $\phi_F:F_1\longrightarrow F_2$, $\phi_E: E_1\longrightarrow E_2$ satisfying \[p_2\circ \phi_F=\phi_E \circ p_1,~\text{and}~q_2 \circ \phi_F=\phi_E \circ q_1,\] is precisely a natural transformation between the functors $T_1$ and $T_2$. It follows that the category of all two-dimensional textile systems, which we denote by \textbf{2-TSys}, is a full subcategory of the functor category $\textbf{Dgr}^\downdownarrows$. 

The next two results make us realize that our description of a three-dimensional textile system is in line with the above sort of categorical representations of directed graphs and two-dimensional textiles. We use the following notation: if $G$ is a directed graph, then $\Delta(G)$ denotes the subgraph of $G\times G$ with $\Delta(G)^0:=\{(u,u)~|~u\in G^0\}$ and $\Delta(G)^1:=\{(e,e)~|~e\in G^1\}$. 
\begin{prop}\label{pro categorical realization of 3D textiles}
Suppose $\mathcal{T}$ is a three-dimensional textile presented by the diagram 

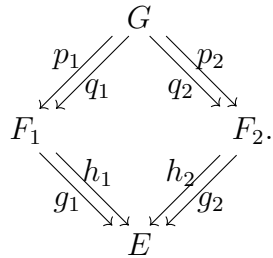
\[
\begin{tikzpicture}[scale=1]
\node[inner sep=2.5pt, circle] (B) at (-1.5,0) {$F_1$};
\node[inner sep=2.5pt, circle] (C) at (1.5,0) {$F_2.$};	
\node[inner sep=2.5pt, circle] (D) at (0,-1.5) {$E$};
\node[inner sep=2.5pt, circle] (A) at (0,1.5) {$G$}; 

\draw[transform canvas={xshift=0.6ex}, ->] (A) -- (B);
\draw[transform canvas={xshift=-0.6ex}, ->] (A) -- (B);

\draw[transform canvas={xshift=0.6ex}, ->] (A) -- (C);
\draw[transform canvas={xshift=-0.6ex}, ->] (A) -- (C);

\draw[transform canvas={xshift=0.6ex}, ->] (C) -- (D);
\draw[transform canvas={xshift=-0.6ex}, ->] (C) -- (D);

\draw[transform canvas={xshift=0.6ex}, ->] (B) -- (D);
\draw[transform canvas={xshift=-0.6ex}, ->] (B) -- (D);

\node at (-0.55,0.55) {$q_1$};
\node at (-0.95,0.95) {$p_1$};

\node at (0.55,0.55) {$q_2$};
\node at (0.95,0.95) {$p_2$};

\node at (-0.55,-0.55) {$h_1$};
\node at (-0.95,-0.95) {$g_1$};

\node at (0.55,-0.55) {$h_2$};
\node at (0.95,-0.95) {$g_2$};
\end{tikzpicture}
\]
Then $T_1:=(\Delta(G),F_1\times F_2, p_1\times p_2,q_1\times q_2)$, $T_0:=(F_2\times F_1, E\times E, g_2\times g_1, h_2\times h_1)$ are two-dimensional textile systems, and $\phi:=(p_2\times q_1,g_1\times h_2)$, $\psi:=(q_2\times p_1, h_1\times g_2)$ are textile morphisms from $T_1$ to $T_0$. Consequently, $\mathcal{T}$ can be realized as a functor from $\downdownarrows$ to \textup{\textbf{2-TSys}}. 
\end{prop}
\begin{proof}
Let $x,y\in \Delta(G)^1$. Then, $x=(e,e)$ and $y=(f,f)$ for some $e,f\in G^1$. Now \[\left(s_{\Delta(G)}(x),r_{\Delta(G)}(x),(p_1\times p_2)(x),(q_1\times q_2)(x)\right)=\left(s_{\Delta(G)}(y),r_{\Delta(G)}(y),(p_1\times p_2)(y),(q_1\times q_2)(y)\right)\] implies that \[(s_{G_1}(e),r_{G_1}(e),p_1^1(e),q_1^1(e),p_2^1(e),q_2^1(e))=(s_{G_1}(f),r_{G_1}(f),p_1^1(f),q_1^1(f),p_2^1(f),q_2^1(f)),\] whence $e=f$ by the injectivity of the map $s_G\times r_G\times p_1^1\times q_1^1\times p_2^1\times q_2^1$. Hence, $x=y$ and it follows that $T_1$ is a two-dimensional textile system. Now we show that $T_0$ is a textile system. Suppose $(x,y),(x',y')\in (F_2\times F_1)^1$ such that $\left(s_{F_2\times F_1}(x,y),r_{F_2\times F_1}(x,y),(g_2\times g_1)(x,y), (h_2\times h_1)(x,y)\right)=(s_{F_2\times F_1}(x',y'),r_{F_2\times F_1}(x',y'),$ $(g_2\times g_1)(x',y'), (h_2\times h_1)(x',y'))$. This implies \[(s_{F_2}(x),r_{F_2}(x),g_2(x),h_2(x))=(s_{F_2}(x'),r_{F_2}(x'),g_2(x'),h_2(x'))\] and \[(s_{F_1}(y),r_{F_1}(y),g_1(y),h_1(y))=(s_{F_1}(y'),r_{F_1}(y'),g_1(y'),h_1(y')).\] Now the injectivity of $s_{F_1}\times r_{F_1}\times g_1^1\times h_1^1$ and $s_{F_2}\times r_{F_2}\times g_2^1\times h_2^1$ gives $y=y'$ and $x=x'$, from which we have $(x,y)=(x',y')$. Thus, $T_0$ is a textile system. To show that $\phi=(p_2\times q_1,g_1\times h_2)$ is a textile morphism from $T_1$ to $T_0$, we just need to verify that both the diagrams 

\[
\begin{tikzpicture}[scale=2.0]
\node[] (A1) at (0,0) {$\Delta(G)$};
\node[] (A2) at (2,0) {$F_1\times F_2$};
\node[] (A3) at (4,0) {$\Delta(G)$};
\node[] (A4) at (6,0) {$F_1\times F_2$};

\node[] (B1) at (0,-1.5) {$F_2\times F_1$};
\node[] (B2) at (2,-1.5) {$E\times E$};
\node[] (B3) at (4,-1.5) {$F_2\times F_1$};
\node[] (B4) at (6,-1.5) {$E\times E$};

\path[->, >=latex,thick] (A1) edge [] node[]{} (A2);
\path[->, >=latex,thick] (B1) edge [] node[]{} (B2);
\path[->, >=latex,thick] (A3) edge [] node[]{} (A4);
\path[->, >=latex,thick] (B3) edge [] node[]{} (B4);

\path[->, >=latex,thick] (A1) edge [] node[]{} (B1);
\path[->, >=latex,thick] (A2) edge [] node[]{} (B2);
\path[->, >=latex,thick] (A3) edge [] node[]{} (B3);
\path[->, >=latex,thick] (A4) edge [] node[]{} (B4);

\node at (1,0.2) {$p_1\times p_2$};
\node at (5,0.2) {$q_1\times q_2$};
\node at (1,-1.7) {$g_2\times g_1$};
\node at (5,-1.7) {$h_2\times h_1$};

\node at (-0.4,-0.75) {$p_2\times q_1$};
\node at (2.4,-0.75) {$g_1\times h_2$};
\node at (3.6,-0.75) {$p_2\times q_1$};
\node at (6.4,-0.75) {$g_1\times h_2$};
\end{tikzpicture}
\]
commute, i.e., $(g_1\circ p_1 \times h_2\circ p_2)=(g_2\circ p_2\times g_1\circ q_1)$ and $(g_1\circ q_1\times h_2\circ q_2)=(h_2\circ p_2\times h_1\circ q_1)$. But these exactly follow by using the inner-commutativity, the outer-commutativity and the cross-commutativity 1 of Definition \ref{def 3D textile}. A similar argument using the inner- and outer-commutativity and the cross-commutativity 2 shows that $\psi=(q_2\times p_1, h_1\times g_2)$ is a textile morphism. It is now clear that we can view the three-dimensional textile system as a functor $\mathcal{T}:\downdownarrows\longrightarrow \textbf{2-TSys}$ by defining $\mathcal{T}(1):=T_1$, $\mathcal{T}(0):=T_0$ and $\phi,\psi$ as the images of the parallel arrows.  
\end{proof}
Although, the above functorial description nicely incorporates all the defining conditions of a three-dimensional textile system, from this, it is somehow hard to figure out when a functor from $\downdownarrows$ to \textbf{2-TSys} gives a three-dimensional textile system. Now, we provide a complete answer to this. As it appears, we can use the following theorem to give an alternate definition of a three-dimensional textile system, exactly similar to the two-dimensional case.
\begin{thm}\label{th alternate description of 3D textile}
The diagram 
\[
\begin{tikzpicture}[scale=1]
\node[inner sep=2.5pt, circle] (B) at (-1.5,0) {$F_1$};
\node[inner sep=2.5pt, circle] (C) at (1.5,0) {$F_2.$};	
\node[inner sep=2.5pt, circle] (D) at (0,-1.5) {$E$};
\node[inner sep=2.5pt, circle] (A) at (0,1.5) {$G$}; 

\draw[transform canvas={xshift=0.6ex}, ->] (A) -- (B);
\draw[transform canvas={xshift=-0.6ex}, ->] (A) -- (B);

\draw[transform canvas={xshift=0.6ex}, ->] (A) -- (C);
\draw[transform canvas={xshift=-0.6ex}, ->] (A) -- (C);

\draw[transform canvas={xshift=0.6ex}, ->] (C) -- (D);
\draw[transform canvas={xshift=-0.6ex}, ->] (C) -- (D);

\draw[transform canvas={xshift=0.6ex}, ->] (B) -- (D);
\draw[transform canvas={xshift=-0.6ex}, ->] (B) -- (D);

\node at (-0.55,0.55) {$q_1$};
\node at (-0.95,0.95) {$p_1$};

\node at (0.55,0.55) {$q_2$};
\node at (0.95,0.95) {$p_2$};

\node at (-0.55,-0.55) {$h_1$};
\node at (-0.95,-0.95) {$g_1$};

\node at (0.55,-0.55) {$h_2$};
\node at (0.95,-0.95) {$g_2$};
\end{tikzpicture}
\] 
represents a three-dimensional textile system if and only if \[T_t:=(G,F_1,p_1,q_1),~~T_b:=(F_2,E,g_2,h_2)\] are two-dimensional textile systems and $(p_2,g_1)$, $(q_2,h_1)$ are textile morphisms from $T_t$ to $T_b$ such that the maps 
\begin{align*}
s_{F_1}\times r_{F_1}\times g_1^1\times h_1^1:F_1^1 &\longrightarrow F_1^0\times F_1^0\times E^1\times E^1,\\
p_1^0\times q_1^0\times p_2^0\times q_2^0:G^0 &\longrightarrow F_1^0\times F_1^0\times F_2^0\times F_2^0
\end{align*}
are injective.
\end{thm}
\begin{proof}
Suppose the diagram represents a three-dimensional textile system. That $T_b$ is a two-dimensional textile follows from the definition (see Remarks \ref{rem vieweing 2D textile as a 3D textile} $(i)$). We show that $T_t$ is also a textile system in two-dimension. Assume that $(s_G\times r_G\times p_1^1\times q_1^1)(x)=(s_G\times r_G\times p_1^1\times q_1^1)(y)$ for $x,y\in G^1$. Now, \[s_{F_2}(p_2^1(x))=p_2^0(s_G(x))=p_2^0(s_G(y))=s_{F_2}(p_2^1(y)),~ r_{F_2}(p_2^1(x))=p_2^0(r_G(x))=p_2^0(r_G(y))=r_{F_2}(p_2^1(y)),\] \[g_2^1(p_2^1(x))=g_1^1(p_1^1(x))=g_1^1(p_1^1(y))=g_2^1(p_2^1(y)),~ h_2^1(p_2^1(x))=g_1^1(q_1^1(x))=g_1^1(q_1^1(y))=h_2^1(p_2^1(y)).\] Since, $s_{F_2}\times r_{F_2}\times g_2^1\times h_2^1$ is injective, it follows that $p_2^1(x)=p_2^1(y)$. Similarly, we can show $q_2^1(x)=q_2^1(y)$. The injectivity of the first map in $(\mathcal{I})$ of Definition \ref{def 3D textile} now implies $x=y$. Thus, $T_t$ is a textile system. A moment's thought shows that $(p_2,g_1)$ and $(q_2,h_1)$ are textile morphisms from $T_t$ to $T_b$ if and only if the four commutativity in $(\mathcal{C})$ of Definition \ref{def 3D textile} hold. The necessity follows. For the sufficiency, note that everything is in hand except the commutativity of the first map in $(\mathcal{I})$. But, this is also straightforward to show using the fact that $T_t$ is a two-dimensional textile system.
\end{proof}
\section{The interplay between three-dimensional textiles and 3-graphs}\label{sec 3-graphs and 3D textiles}
As described in the beginning of this section, our definition of a three-dimensional textile system is inspired by the analogy between $2$-graphs and $LR$-textile systems. We now seek to investigate a similar type of analogy between $3$-graphs and suitable three-dimensional textile systems. We start by showing that any $3$-graph naturally gives rise to a specific type of three-dimensional textile. 

For the rest of this section, we fix $\Lambda$ to be a $3$-graph. We now define four directed graphs as follows: \[G(\Lambda):=(\Lambda^{\ea+\eb},\Lambda^{\mathbf{e}_1+\mathbf{e}_2+\mathbf{e}_3},\ev_{\mathbf{e}_3,\mathbf{e}_1+\mathbf{e}_2+\mathbf{e}_3},\ev_{0,\mathbf{e}_1+\mathbf{e}_2});\] \[F_1(\Lambda):=(\Lambda^{\eb},\Lambda^{\eb+\ec},\ev_{\ec,\eb+\ec},\ev_{0,\eb}),~~F_2(\Lambda):=(\Lambda^{\ea},\Lambda^{\ec+\ea},\ev_{\ec,\ea+\ec},\ev_{0,\ea});\] \[E(\Lambda):=(\Lambda^0,\Lambda^{\ec},s_\Lambda, r_\Lambda).\]

\begin{prop}\label{pro 3-graph to 3D textile}
Let $\Lambda$ be any $3$-graph. Consider the system $\mathcal{T}_\Lambda$ presented by the diagram:

\[
\begin{tikzpicture}[scale=1.3]
\node[inner sep=0.5pt, circle] (B) at (-1.5,0) {$F_1(\Lambda)$};
\node[inner sep=0.5pt, circle] (C) at (1.5,0) {$F_2(\Lambda)$};	
\node[inner sep=0.5pt, circle] (D) at (0,-1.5) {$E(\Lambda)$};
\node[inner sep=0.5pt, circle] (A) at (0,1.5) {$G(\Lambda)$}; 

\draw[transform canvas={xshift=0.6ex}, ->] (A) -- (B);
\draw[transform canvas={xshift=-0.6ex}, ->] (A) -- (B);

\draw[transform canvas={xshift=0.6ex}, ->] (A) -- (C);
\draw[transform canvas={xshift=-0.6ex}, ->] (A) -- (C);

\draw[transform canvas={xshift=0.6ex}, ->] (C) -- (D);
\draw[transform canvas={xshift=-0.6ex}, ->] (C) -- (D);

\draw[transform canvas={xshift=0.6ex}, ->] (B) -- (D);
\draw[transform canvas={xshift=-0.6ex}, ->] (B) -- (D);

\node at (-0.55,0.55) {$q_1$};
\node at (-0.95,0.95) {$p_1$};

\node at (0.55,0.55) {$q_2$};
\node at (0.95,0.95) {$p_2$};

\node at (-0.55,-0.55) {$h_1$};
\node at (-0.95,-0.95) {$g_1$};

\node at (0.55,-0.55) {$h_2$};
\node at (0.95,-0.95) {$g_2$};
\end{tikzpicture}
\]
where \[p_1:=(\ev_{0,\eb},\ev_{0,\eb+\ec}), q_1:=(\ev_{\ea,\ea+\eb},\ev_{\ea,\ea+\eb+\ec}),p_2:=(\ev_{0,\ea},\ev_{0,\ec+\ea}),q_2:=(\ev_{\eb,\ea+\eb},\ev_{\eb,\ea+\eb+\ec})\] and \[g_1:=(r_\Lambda,\ev_{0,\ec}),h_1:=(s_\Lambda,\ev_{\eb,\eb+\ec}),g_2:=(r_\Lambda,\ev_{0,\ec}),h_2:=(s_\Lambda,\ev_{\ea,\ea+\ec}).\] Then $\mathcal{T}_\Lambda$ is a three-dimensional textile system. 
\end{prop}
\begin{proof}
Let $\lambda\in G(\Lambda)^1$. Then \[r_{F_1(\Lambda)}(p_1(\lambda))=(\lambda(0,\eb+\ec))(\ec,\eb+\ec)=\lambda(\ec,\eb+\ec)=(\lambda(\ec,\ea+\eb+\ec))(0,\eb)=p_1(r_{G(\Lambda)}(\lambda)),\] and \[s_{F_1(\Lambda)}(p_1(\lambda))=(\lambda(0,\eb+\ec))(0,\eb)=\lambda(0,\eb)=(\lambda(0,\ea+\eb))(0,\eb)=p_1(s_{G(\Lambda)}(\lambda)).\] This shows that $p_1$ is a graph homomorphism. The other maps can similarly be proved to be graph morphisms, so we skip the details. Now we verify the four commutativity. We have \[g_1^1(p_1^1(\lambda))=(\lambda(0,\eb+\ec))(0,\ec)=\lambda(0,\ec)=(\lambda(0,\ec+\ea))(0,\ec)=g_2^1(p_2^1(\lambda)),\] and \[g_1^0(p_1^0(\mu))=r_\Lambda(\mu(0,\eb))=r_\Lambda(\mu(0,\ea))=g_2^0(p_2^0(\mu))\] for all $\lambda\in G(\Lambda)^1$ and $\mu\in G(\Lambda)^0$. Hence, the outer-commutativity follows. The inner-commutativity is similar. Also \[g_1^1(q_1^1(\lambda))=(\lambda(\ea,\ea+\eb+\ec))(0,\ec)=\lambda(\ea,\ea+\ec)=(\lambda(0,\ec+\ea))(\ea,\ea+\ec)=h_2^1(p_2^1(\lambda)),\] and \[g_1^0(q_1^0(\mu))=r_\Lambda(\mu(\ea,\ea+\eb))=s_\Lambda(\mu(0,\ea))=h_2^0(p_2^0(\mu))\] for all $\lambda\in G(\Lambda)^1$ and $\mu\in G(\Lambda)^0$. Thus, we have the cross-commutativity 1. The other cross-commutativity follows similarly. Finally, the injectivity of the four maps in $(\mathcal{I})$ of Definition \ref{def 3D textile} follow from the unique factorization of paths in $\Lambda$. For instance, if we have $\lambda,\mu\in G(\Lambda)^1$ with \[(s_{G(\Lambda)}(\lambda),r_{G(\Lambda)}(\lambda),p_1^1(\lambda),q_1^1(\lambda),p_2^1(\lambda),q_2^1(\lambda))=(s_{G(\Lambda)}(\mu),r_{G(\Lambda)}(\mu),p_1^1(\mu),q_1^1(\mu),p_2^1(\mu),q_2^1(\mu)),\] then, $\lambda(0,\ea+\eb)=\mu(0,\ea+\eb)$ and \[\lambda(\ea+\eb,\ea+\eb+\ec)=(\lambda(\ea,\ea+\eb+\ec))(\eb,\eb+\ec)=(\mu(\ea,\ea+\eb+\ec))(\eb,\eb+\ec)=\mu(\ea+\eb,\ea+\eb+\ec)\] which, in view of the unique factorization property, implies $\lambda=\mu$. 
\end{proof}

The morphisms involved in the textile system associated with a $3$-graph satisfy some nice properties. We now present these with the help of pullback squares (see Subsection \ref{ssec textile systems and 2-graphs}).  

\begin{prop}\label{pro path liftings for 3-graph textile}
Let $\Lambda$ be any $3$-graph and $\mathcal{T}_\Lambda$ the associated three-dimensional textile system. Then 

$(i)$ $g_1,g_2$ have unique $r$-path lifting;

$(ii)$ $h_1,h_2$ have unique $s$-path lifting;

$(iii)$ the diagrams 

\[
\begin{tikzpicture}[scale=1.3]
\node[] (A1) at (0,0) {$G(\Lambda)^0$};
\node[] (A2) at (2,0) {$F_1(\Lambda)^0$};
\node[] (A3) at (4,0) {$G(\Lambda)^0$};
\node[] (A4) at (6,0) {$F_1(\Lambda)^0$};

\node[] (B1) at (0,-2) {$F_2(\Lambda)^0$};
\node[] (B2) at (2,-2) {$E(\Lambda)^0$};
\node[] (B3) at (4,-2) {$F_2(\Lambda)^0$};
\node[] (B4) at (6,-2) {$E(\Lambda)^0$};

\path[->, >=latex,thick] (A1) edge [] node[]{} (A2);
\path[->, >=latex,thick] (B1) edge [] node[]{} (B2);
\path[->, >=latex,thick] (A3) edge [] node[]{} (A4);
\path[->, >=latex,thick] (B3) edge [] node[]{} (B4);

\path[->, >=latex,thick] (A1) edge [] node[]{} (B1);
\path[->, >=latex,thick] (A2) edge [] node[]{} (B2);
\path[->, >=latex,thick] (A3) edge [] node[]{} (B3);
\path[->, >=latex,thick] (A4) edge [] node[]{} (B4);

\node at (1,0.3) {$q_1^0$};
\node at (5,0.3) {$p_1^0$};
\node at (1,-2.3) {$h_2^0$};
\node at (5,-2.3) {$g_2^0$};

\node at (-0.3,-1) {$p_2^0$};
\node at (2.3,-1) {$g_1^0$};
\node at (3.7,-1) {$q_2^0$};
\node at (6.3,-1) {$h_1^0$};
\end{tikzpicture}
\]
are pullback squares;

$(iv)$ $q_1$ has $s$-path lifting.

\end{prop}
\begin{proof} We only prove $(iii)$ and $(iv)$. Similar line of argument, which we omit, can be employed to establish $(i)$ and $(ii)$. 

$(iii)$ The diagrams in $(iii)$ are commutative by $(\mathcal{C})$ of Definition \ref{def 3D textile}. Suppose $Q$ is a set and we have maps $\alpha:Q\longrightarrow F_1(\Lambda)^0$, $\beta:Q\longrightarrow F_2(\Lambda)^0$ such that $g_1^0 \circ \alpha=h_2^0 \circ \beta$. Choose any $x\in Q$. Then, $g_1^0(\alpha(x))=h_2^0(\beta(x))$ which implies $r_\Lambda(\alpha(x))=s_\Lambda(\beta(x))$. Let $\lambda_x=\beta(x)\alpha(x)$. Then $\lambda_x\in \Lambda^{\ea+\eb}=G(\Lambda)^0$ and \[q_1^0(\lambda_x)=\lambda_x(\ea,\ea+\eb)=\alpha(x)\] and \[p_2^0(\lambda_x)=\lambda_x(0,\ea)=\beta(x).\] By unique factorization, it follows that $\lambda_x$ is the only element in $G(\Lambda)^0$ with these properties. Define a map 
\begin{align*}
\lambda:Q &\longrightarrow G(\Lambda)^0\\
x &\longmapsto \lambda_x.
\end{align*}
Then it is easy to see that $\lambda$ is the unique map with the property that $q_1^0\circ \lambda=\alpha$ and $p_2^0\circ \lambda=\beta$. Hence, the first diagram is a pullback square. The case for the second one is similar.

$(iv)$ Suppose $\lambda\in G(\Lambda)^0$ and $\mu\in F_1(\Lambda)^1$ are such that $q_1^0(\lambda)=s_{F_1(\Lambda)}(\mu)$. This means $\lambda(\ea,\ea+\eb)=\mu(0,\eb)$. Let \[\gamma:=\lambda(0,\ea)\lambda(\ea,\ea+\eb)\mu(\eb,\eb+\ec)=\lambda(0,\ea)\mu(0,\eb)\mu(\eb,\eb+\ec)\in \Lambda^{\ea+\eb+\ec}=G(\Lambda)^1.\] Then, $s_{G(\Lambda)}(\gamma)=\lambda$ and $q_1^1(\gamma)=\mu$. Note that \[\gamma(\ea+\eb,\ea+\eb+\ec)=(\gamma(\ea,\ea+\eb+\ec))(\eb,\eb+\ec)=\mu(\eb,\eb+\ec).\] If $\delta\in G(\Lambda)^1$ is such that $s_{G(\Lambda)}(\delta)=\lambda$ and $q_1^1(\delta)=\mu$ then \[\gamma(0,\ea+\eb)=\lambda=\delta(0,\ea+\eb)\] and \[\gamma(\ea+\eb,\ea+\eb+\ec)=\mu(\eb,\eb+\ec)=\delta(\ea+\eb,\ea=\eb+\ec).\] By the unique factorization, this implies $\gamma=\delta$. Thus, $q_1$ has unique $s$-path lifting.
\end{proof}

The properties $(i)$--$(iv)$ of the conclusion of Proposition \ref{pro path liftings for 3-graph textile} can yield several other pullback squares in a three-dimensional textile system. We establish this in the following lemma; the proof highlights the advantage of expressing path lifting properties using pullback squares.
\begin{lem}\label{lem other pullbacks}
Let $\mathcal{T}$ be a three-dimensional textile presented by the following diagram: 

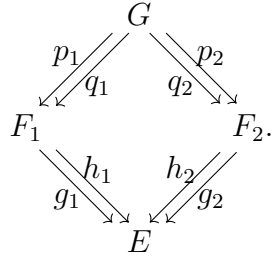
\[
\begin{tikzpicture}[scale=1]
\node[inner sep=2.5pt, circle] (B) at (-1.5,0) {$F_1$};
\node[inner sep=2.5pt, circle] (C) at (1.5,0) {$F_2.$};	
\node[inner sep=2.5pt, circle] (D) at (0,-1.5) {$E$};
\node[inner sep=2.5pt, circle] (A) at (0,1.5) {$G$}; 

\draw[transform canvas={xshift=0.6ex}, ->] (A) -- (B);
\draw[transform canvas={xshift=-0.6ex}, ->] (A) -- (B);

\draw[transform canvas={xshift=0.6ex}, ->] (A) -- (C);
\draw[transform canvas={xshift=-0.6ex}, ->] (A) -- (C);

\draw[transform canvas={xshift=0.6ex}, ->] (C) -- (D);
\draw[transform canvas={xshift=-0.6ex}, ->] (C) -- (D);

\draw[transform canvas={xshift=0.6ex}, ->] (B) -- (D);
\draw[transform canvas={xshift=-0.6ex}, ->] (B) -- (D);

\node at (-0.55,0.55) {$q_1$};
\node at (-0.95,0.95) {$p_1$};

\node at (0.55,0.55) {$q_2$};
\node at (0.95,0.95) {$p_2$};

\node at (-0.55,-0.55) {$h_1$};
\node at (-0.95,-0.95) {$g_1$};

\node at (0.55,-0.55) {$h_2$};
\node at (0.95,-0.95) {$g_2$};
\end{tikzpicture}
\] Suppose the following hold: 

$(i)$ $g_1,g_2$ have unique $r$-path lifting;

$(ii)$ $h_1,h_2$ have unique $s$-path lifting;

$(iii)$ the diagrams 

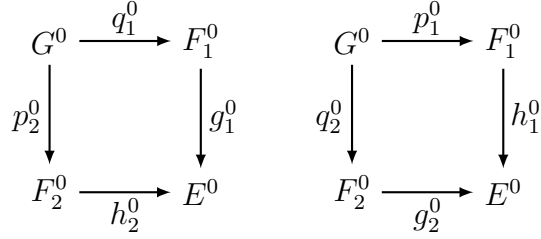
\[
\begin{tikzpicture}[scale=1]
\node[] (A1) at (0,0) {$G^0$};
\node[] (A2) at (2,0) {$F_1^0$};
\node[] (A3) at (4,0) {$G^0$};
\node[] (A4) at (6,0) {$F_1^0$};

\node[] (B1) at (0,-2) {$F_2^0$};
\node[] (B2) at (2,-2) {$E^0$};
\node[] (B3) at (4,-2) {$F_2^0$};
\node[] (B4) at (6,-2) {$E^0$};

\path[->, >=latex,thick] (A1) edge [] node[]{} (A2);
\path[->, >=latex,thick] (B1) edge [] node[]{} (B2);
\path[->, >=latex,thick] (A3) edge [] node[]{} (A4);
\path[->, >=latex,thick] (B3) edge [] node[]{} (B4);

\path[->, >=latex,thick] (A1) edge [] node[]{} (B1);
\path[->, >=latex,thick] (A2) edge [] node[]{} (B2);
\path[->, >=latex,thick] (A3) edge [] node[]{} (B3);
\path[->, >=latex,thick] (A4) edge [] node[]{} (B4);

\node at (1,0.3) {$q_1^0$};
\node at (5,0.3) {$p_1^0$};
\node at (1,-2.3) {$h_2^0$};
\node at (5,-2.3) {$g_2^0$};

\node at (-0.3,-1) {$p_2^0$};
\node at (2.3,-1) {$g_1^0$};
\node at (3.7,-1) {$q_2^0$};
\node at (6.3,-1) {$h_1^0$};
\end{tikzpicture}
\]
are pullback squares;

$(iv)$ $q_1$ has $s$-path lifting. 

Then $q_1$ has unique $s$-path lifting and 

$(a)$ the diagrams 

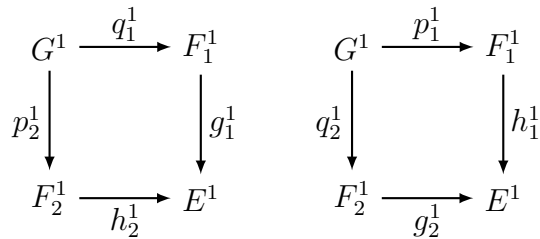
\[
\begin{tikzpicture}[scale=1]
\node[] (A1) at (0,0) {$G^1$};
\node[] (A2) at (2,0) {$F_1^1$};
\node[] (A3) at (4,0) {$G^1$};
\node[] (A4) at (6,0) {$F_1^1$};

\node[] (B1) at (0,-2) {$F_2^1$};
\node[] (B2) at (2,-2) {$E^1$};
\node[] (B3) at (4,-2) {$F_2^1$};
\node[] (B4) at (6,-2) {$E^1$};

\path[->, >=latex,thick] (A1) edge [] node[]{} (A2);
\path[->, >=latex,thick] (B1) edge [] node[]{} (B2);
\path[->, >=latex,thick] (A3) edge [] node[]{} (A4);
\path[->, >=latex,thick] (B3) edge [] node[]{} (B4);

\path[->, >=latex,thick] (A1) edge [] node[]{} (B1);
\path[->, >=latex,thick] (A2) edge [] node[]{} (B2);
\path[->, >=latex,thick] (A3) edge [] node[]{} (B3);
\path[->, >=latex,thick] (A4) edge [] node[]{} (B4);

\node at (1,0.3) {$q_1^1$};
\node at (5,0.3) {$p_1^1$};
\node at (1,-2.3) {$h_2^1$};
\node at (5,-2.3) {$g_2^1$};

\node at (-0.3,-1) {$p_2^1$};
\node at (2.3,-1) {$g_1^1$};
\node at (3.7,-1) {$q_2^1$};
\node at (6.3,-1) {$h_1^1$};
\end{tikzpicture}
\]
are pullback squares;

$(b)$ $q_2$ has unique $s$-path lifting;

$(c)$ $p_1,p_2$ have unique $r$-path lifting.
\end{lem}
\begin{proof}
To show that the $s$-path lifting of $q_1$ is unique, consider $x,y\in G^1$ such that $s_G(x)=s_G(y)$ and $q_1^1(x)=q_1^1(y)$. Then \[s_{F_2}(p_2^1(x))=p_2^0(s_G(x))=p_2^0(s_G(y))=s_{F_2}(p_2^1(y)),\] and \[h_2^1(p_2^1(x))=g_1^1(q_1^1(x))=g_1^1(q_1^1(y))=h_2^1(p_2^1(y)).\] This implies $p_2^1(x)=p_2^1(y)$ by unique $s$-path lifting of $h_2$. A similar argument yields $q_2^1(x)=q_2^1(y)$. Again, \[s_{F_1}(p_1^1(x))=p_1^0(s_G(x))=p_1^0(s_G(y))=s_{F_1}(p_1^1(y)),\] and \[h_1^1(p_1^1(x))=g_2^1(q_2^1(x))=g_2^1(q_2^1(y))=h_1^1(p_1^1(y)),\] which gives $p_1^1(x)=p_1^1(y)$ by unique $s$-path lifting property of $h_1$. Similarly, using the first pullback square of $(iii)$, it can be shown that $r_G(x)=r_G(y)$. Since $s_G\times r_G\times p_1^1\times q_1^1\times p_2^1\times q_2^1$ is injective, we have $x=y$ and thus $q_1$ has unique $s$-path lifting. We now show that the first square in $(a)$ is a pullback square. For this, consider the following diagram:

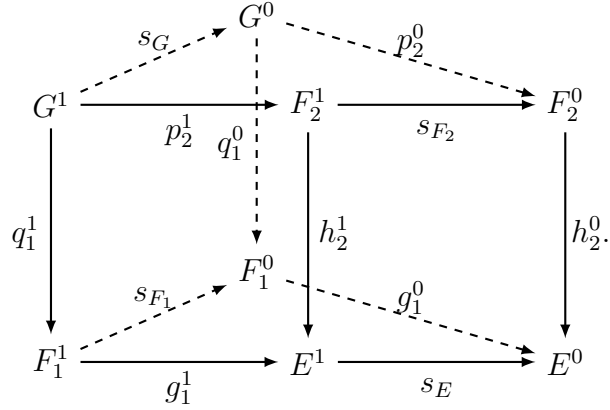
\[
\begin{tikzpicture}[scale=1.7]
\node[] (A1) at (0,0) {$G^1$};
\node[] (A2) at (2,0) {$F_2^1$};
\node[] (A3) at (0,-2) {$F_1^1$};
\node[] (A4) at (2,-2) {$E^1$};

\node[] (B1) at (1.6,0.7) {$G^0$};
\node[] (B2) at (4,0) {$F_2^0$};
\node[] (B3) at (1.6,-1.3) {$F_1^0$};
\node[] (B4) at (4,-2) {$E^0$};

\path[->, black, >=latex,thick] (A1) edge [] node[]{} (A2);
\path[->, black, >=latex,thick] (A3) edge [] node[]{} (A4);
\path[->, black, >=latex,thick] (A1) edge [] node[]{} (A3);
\path[->, black, >=latex,thick] (A2) edge [] node[]{} (A4);

\path[->,black, dashed, >=latex,thick] (B1) edge [] node[]{} (B3);
\path[->,black, >=latex,thick] (A4) edge [] node[]{} (B4);
\path[->,black, >=latex,thick] (A2) edge [] node[]{} (B2);
\path[->,black, >=latex,thick] (B2) edge [] node[]{} (B4);

\path[->,black, dashed, >=latex,thick] (A1) edge [] node {} (B1);
\path[->,black, dashed, >=latex,thick] (A3) edge [] node {} (B3);
\path[->,black, dashed, >=latex,thick] (B1) edge [] node {} (B2);
\path[->,black, dashed, >=latex,thick] (B3) edge [] node {} (B4);

\node at (1,-0.2) {$p_2^1$};
\node at (1,-2.2) {$g_1^1$};
\node at (-0.2,-1) {$q_1^1$};
\node at (2.2,-1) {$h_2^1$};

\node at (3,-0.2) {$s_{F_2}$};
\node at (3,-2.2) {$s_E$};
\node at (4.2,-1) {$h_2^0$.};

\node at (0.8,0.5) {$s_G$};
\node at (0.8,-1.5) {$s_{F_1}$};
\node at (2.8,0.5) {$p_2^0$};
\node at (2.8,-1.5) {$g_1^0$};

\node at (1.4,-0.3) {$q_1^0$};

\end{tikzpicture}
\]
We observe that the two squares in the back are pullback squares. The right one is essentially the first square in $(iii)$, whereas the left one is a pullback square, since $q_1$ has unique $s$-path lifting. By pasting law for pullbacks, the rectangle in the back (formed by pasting the left and right squares, side by side) is a pullback rectangle. Now, $p_2$ and $g_1$ are graph morphisms, which implies $s_{F_2}\circ p_2^1=p_2^0\circ s_G$ and $s_E\circ g_1^1=g_1^0\circ s_{F_1}$. It follows that the rectangle in the front side, obtained by attaching the left and right squares along the side $h_2^1$ and then ignoring this side, is a pullback rectangle. Again, the right-hand square is a pullback due to the unique $s$-path lifting of $h_2$. Applying the pasting law again, we can conclude that the left-hand square is a pullback and we are done. 

Before tackling the second square in $(a)$, we prove $(b)$. Consider the diagram:

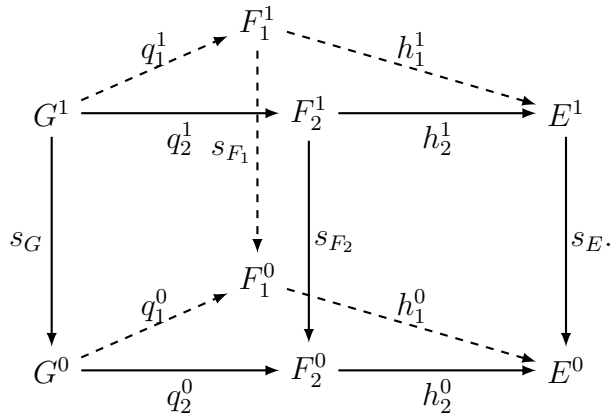
\[
\begin{tikzpicture}[scale=1.7]
\node[] (A1) at (0,0) {$G^1$};
\node[] (A2) at (2,0) {$F_2^1$};
\node[] (A3) at (0,-2) {$G^0$};
\node[] (A4) at (2,-2) {$F_2^0$};

\node[] (B1) at (1.6,0.7) {$F_1^1$};
\node[] (B2) at (4,0) {$E^1$};
\node[] (B3) at (1.6,-1.3) {$F_1^0$};
\node[] (B4) at (4,-2) {$E^0$};

\path[->, black, >=latex,thick] (A1) edge [] node[]{} (A2);
\path[->, black, >=latex,thick] (A3) edge [] node[]{} (A4);
\path[->, black, >=latex,thick] (A1) edge [] node[]{} (A3);
\path[->, black, >=latex,thick] (A2) edge [] node[]{} (A4);

\path[->,black, dashed, >=latex,thick] (B1) edge [] node[]{} (B3);
\path[->,black, >=latex,thick] (A4) edge [] node[]{} (B4);
\path[->,black, >=latex,thick] (A2) edge [] node[]{} (B2);
\path[->,black, >=latex,thick] (B2) edge [] node[]{} (B4);

\path[->,black, dashed, >=latex,thick] (A1) edge [] node {} (B1);
\path[->,black, dashed, >=latex,thick] (A3) edge [] node {} (B3);
\path[->,black, dashed, >=latex,thick] (B1) edge [] node {} (B2);
\path[->,black, dashed, >=latex,thick] (B3) edge [] node {} (B4);

\node at (1,-0.2) {$q_2^1$};
\node at (1,-2.2) {$q_2^0$};
\node at (-0.2,-1) {$s_G$};
\node at (2.2,-1) {$s_{F_2}$};

\node at (3,-0.2) {$h_2^1$};
\node at (3,-2.2) {$h_2^0$};
\node at (4.2,-1) {$s_E$.};

\node at (0.8,0.5) {$q_1^1$};
\node at (0.8,-1.5) {$q_1^0$};
\node at (2.8,0.5) {$h_1^1$};
\node at (2.8,-1.5) {$h_1^0$};

\node at (1.4,-0.3) {$s_{F_1}$};

\end{tikzpicture}
\] Since $q_1,h_1$ have unique $s$-path lifting, the two squares in the back are pullback squares. Pasting these two along the side $s_{F_1}$, we have a pullback rectangle in the back. Note that, $h_1^1\circ q_1^1=h_2^1\circ q_2^1$ and $h_1^0\circ q_1^0=h_2^0\circ q_2^0$ by the inner-commutativity. Hence, the rectangle in the front (made with solid edges) is a pullback rectangle. Finally, since the right-hand square in the front is a pullback by the unique $s$-path lifting of $h_2$, it follows that the left-hand square is also a pullback whence, $q_2$ has unique $s$-path lifting. Now, with the unique $s$-path lifting of $q_2$ in hand, a similar line of argument, applied on the diagram:
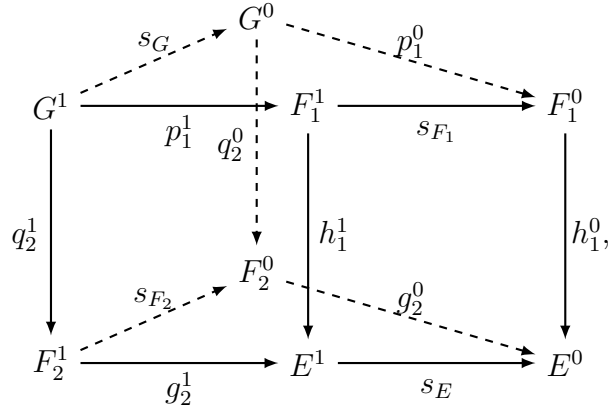
\[
\begin{tikzpicture}[scale=1.7]
\node[] (A1) at (0,0) {$G^1$};
\node[] (A2) at (2,0) {$F_1^1$};
\node[] (A3) at (0,-2) {$F_2^1$};
\node[] (A4) at (2,-2) {$E^1$};

\node[] (B1) at (1.6,0.7) {$G^0$};
\node[] (B2) at (4,0) {$F_1^0$};
\node[] (B3) at (1.6,-1.3) {$F_2^0$};
\node[] (B4) at (4,-2) {$E^0$};

\path[->, black, >=latex,thick] (A1) edge [] node[]{} (A2);
\path[->, black, >=latex,thick] (A3) edge [] node[]{} (A4);
\path[->, black, >=latex,thick] (A1) edge [] node[]{} (A3);
\path[->, black, >=latex,thick] (A2) edge [] node[]{} (A4);

\path[->,black, dashed, >=latex,thick] (B1) edge [] node[]{} (B3);
\path[->,black, >=latex,thick] (A4) edge [] node[]{} (B4);
\path[->,black, >=latex,thick] (A2) edge [] node[]{} (B2);
\path[->,black, >=latex,thick] (B2) edge [] node[]{} (B4);

\path[->,black, dashed, >=latex,thick] (A1) edge [] node {} (B1);
\path[->,black, dashed, >=latex,thick] (A3) edge [] node {} (B3);
\path[->,black, dashed, >=latex,thick] (B1) edge [] node {} (B2);
\path[->,black, dashed, >=latex,thick] (B3) edge [] node {} (B4);

\node at (1,-0.2) {$p_1^1$};
\node at (1,-2.2) {$g_2^1$};
\node at (-0.2,-1) {$q_2^1$};
\node at (2.2,-1) {$h_1^1$};

\node at (3,-0.2) {$s_{F_1}$};
\node at (3,-2.2) {$s_E$};
\node at (4.2,-1) {$h_1^0$,};

\node at (0.8,0.5) {$s_G$};
\node at (0.8,-1.5) {$s_{F_2}$};
\node at (2.8,0.5) {$p_1^0$};
\node at (2.8,-1.5) {$g_2^0$};

\node at (1.4,-0.3) {$q_2^0$};

\end{tikzpicture}
\] shows that the second square in $(a)$ is a pullback. Again, using this pullback square, it can be shown that $p_1$ has unique $r$-path lifting, for which one needs to apply similar order of reasoning on the diagram:

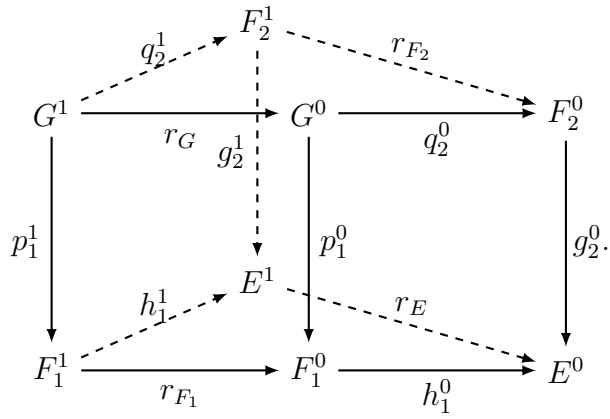
\[
\begin{tikzpicture}[scale=1.7]
\node[] (A1) at (0,0) {$G^1$};
\node[] (A2) at (2,0) {$G^0$};
\node[] (A3) at (0,-2) {$F_1^1$};
\node[] (A4) at (2,-2) {$F_1^0$};

\node[] (B1) at (1.6,0.7) {$F_2^1$};
\node[] (B2) at (4,0) {$F_2^0$};
\node[] (B3) at (1.6,-1.3) {$E^1$};
\node[] (B4) at (4,-2) {$E^0$};

\path[->, black, >=latex,thick] (A1) edge [] node[]{} (A2);
\path[->, black, >=latex,thick] (A3) edge [] node[]{} (A4);
\path[->, black, >=latex,thick] (A1) edge [] node[]{} (A3);
\path[->, black, >=latex,thick] (A2) edge [] node[]{} (A4);

\path[->,black, dashed, >=latex,thick] (B1) edge [] node[]{} (B3);
\path[->,black, >=latex,thick] (A4) edge [] node[]{} (B4);
\path[->,black, >=latex,thick] (A2) edge [] node[]{} (B2);
\path[->,black, >=latex,thick] (B2) edge [] node[]{} (B4);

\path[->,black, dashed, >=latex,thick] (A1) edge [] node {} (B1);
\path[->,black, dashed, >=latex,thick] (A3) edge [] node {} (B3);
\path[->,black, dashed, >=latex,thick] (B1) edge [] node {} (B2);
\path[->,black, dashed, >=latex,thick] (B3) edge [] node {} (B4);

\node at (1,-0.2) {$r_G$};
\node at (1,-2.2) {$r_{F_1}$};
\node at (-0.2,-1) {$p_1^1$};
\node at (2.2,-1) {$p_1^0$};

\node at (3,-0.2) {$q_2^0$};
\node at (3,-2.2) {$h_1^0$};
\node at (4.2,-1) {$g_2^0$.};

\node at (0.8,0.5) {$q_2^1$};
\node at (0.8,-1.5) {$h_1^1$};
\node at (2.8,0.5) {$r_{F_2}$};
\node at (2.8,-1.5) {$r_E$};

\node at (1.4,-0.3) {$g_2^1$};

\end{tikzpicture}
\] Finally, interchanging the subscripts $1$ and $2$ at relevant nodes and arrows in the above diagram, and using necessary pullbacks in the resulting diagram, we can show that $p_2$ has unique $r$-path lifting. 
\end{proof}

\begin{rmk}\label{rem quivalence of the pullbacks}
In the above proof, we use the unique $s$-path lifting of $q_1$ to establish five other pullback squares. Interestingly, we can do this in any order we like. Under the conditions $(i)$--$(iii)$ of Lemma \ref{lem other pullbacks}, any one of the six properties: $(1)$ $q_1$ has unique $s$-path lifting, $(2)$ $q_2$ has unique $s$-path lifting, $(3)$ $p_1$ has unique $r$-path lifting, $(4)$ $p_2$ has unique $r$-path lifting, $(5)$ the first diagram in $(a)$ is a pullback square, $(6)$ the second diagram in $(a)$ is a pullback square; implies the remaining five. To accomplish this, one just need to play with the four diagrams in the proof of Lemma \ref{lem other pullbacks} and use these in a suitable order. Our choice of starting with the $s$-path lifting of $q_1$ is made following the definition of $\mathcal{T}_\Lambda$ and will turn out to be canonical in the proof of Theorem \ref{th 3D textile to 3-graph} below.  
\end{rmk}
Proposition \ref{pro 3-graph to 3D textile} says that any $3$-graph gives a three-dimensional textile. We now turn our attention to the other way round. Suppose $\mathcal{T}$ is a three-dimensional textile presented by the diagram: 

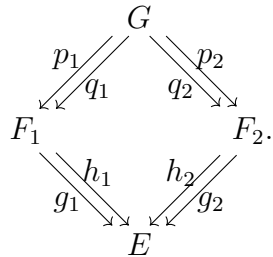
\[
\begin{tikzpicture}[scale=1]
\node[inner sep=2.5pt, circle] (B) at (-1.5,0) {$F_1$};
\node[inner sep=2.5pt, circle] (C) at (1.5,0) {$F_2$.};	
\node[inner sep=2.5pt, circle] (D) at (0,-1.5) {$E$};
\node[inner sep=2.5pt, circle] (A) at (0,1.5) {$G$}; 

\draw[transform canvas={xshift=0.6ex}, ->] (A) -- (B);
\draw[transform canvas={xshift=-0.6ex}, ->] (A) -- (B);

\draw[transform canvas={xshift=0.6ex}, ->] (A) -- (C);
\draw[transform canvas={xshift=-0.6ex}, ->] (A) -- (C);

\draw[transform canvas={xshift=0.6ex}, ->] (C) -- (D);
\draw[transform canvas={xshift=-0.6ex}, ->] (C) -- (D);

\draw[transform canvas={xshift=0.6ex}, ->] (B) -- (D);
\draw[transform canvas={xshift=-0.6ex}, ->] (B) -- (D);

\node at (-0.55,0.55) {$q_1$};
\node at (-0.95,0.95) {$p_1$};

\node at (0.55,0.55) {$q_2$};
\node at (0.95,0.95) {$p_2$};

\node at (-0.55,-0.55) {$h_1$};
\node at (-0.95,-0.95) {$g_1$};

\node at (0.55,-0.55) {$h_2$};
\node at (0.95,-0.95) {$g_2$};
\end{tikzpicture}
\]

Define a directed graph $G_{\mathcal{T}}$ with $G_{\mathcal{T}}^0:=E^0$, $G_{\mathcal{T}}^1:=F_1^0\cup F_2^0\cup E^1$. The range and source maps are defined as follows:
\[r(x):=
	\left\{
	\begin{array}{lll}
		h_1(x)  & \mbox{if } x\in F_1^0, \\
		h_2(x) & \mbox{if }  x\in F_2^0,\\
        r_E(x) & \mbox{if }  x\in E^1,
	\end{array}
	\right.~~~
s(x):=
	\left\{
	\begin{array}{lll}
		g_1(x)  & \mbox{if } x\in F_1^0, \\
		g_2(x) & \mbox{if }  x\in F_2^0,\\
        s_E(x) & \mbox{if }  x\in E^1,
	\end{array}
	\right. 
\]
for all $x\in G_{\mathcal{T}}^1$. We consider $G_{\mathcal{T}}$ as a $3$-colored graph by setting the colors of the edges using the map
\[
d(x):=
	\left\{
	\begin{array}{lll}
		\eb  & \mbox{if } x\in F_1^0, \\
		\ea & \mbox{if }  x\in F_2^0,\\
        \ec & \mbox{if }  x\in E^1,
	\end{array}
	\right. 
\] and then extending this additively to color any finite path. Next, we will define a certain congruence on the free category $G_{\mathcal{T}}^*$. Let us start by defining binary relations on the three sets of bi-colored paths:

$\bullet$ \emph{$R^{12}$ on $\ea$-$\eb$ bi-colored paths}: Let $v,v'\in F_2^0$, $w,w'\in F_1^0$ such that $vw,w'v'\in G_{\mathcal{T}}^2$. Then $vw R^{12} w'v'$ if there exists $x\in G^0$ such that $p_2^0(x)=v$, $q_2^0(x)=v'$, $p_1^0(x)=w'$ and $q_1^0(x)=w$.

$\bullet$ \emph{$R^{23}$ on $\eb$-$\ec$ bi-colored paths}: Let $w,w'\in F_1^0$, $e,e'\in E^1$ such that $we,e'w'\in G_{\mathcal{T}}^2$. Then $we R^{23} e'w'$ if there exists $x\in F_1^1$ such that $s_{F_1}(x)=w$, $r_{F_1}(x)=w'$, $g_1(x)=e'$ and $h_1(x)=e$.

$\bullet$ \emph{$R^{13}$ on $\ea$-$\ec$ bi-colored paths}: Let $v,v'\in F_2^0$, $e,e'\in E^1$ such that $ve,e'v'\in G_{\mathcal{T}}^2$. Then $ve R^{13} e'v'$ if there exists $x\in F_2^1$ such that $s_{F_2}(x)=v$, $r_{F_2}(x)=v'$, $g_2(x)=e'$ and $h_2(x)=e$.

Suppose $\mathcal{R}$ is the smallest congruence on $G_{\mathcal{T}}^*$ containing $R^{12}\cup R^{23}\cup R^{13}\cup \Delta_{G_{\mathcal{T}}^1}$, where $\Delta_{G_{\mathcal{T}}^1}$ is the identity relation on $G_{\mathcal{T}}^1$. Since each of $R^{12}$, $R^{23}$, $R^{13}$ and $\Delta_{G_{\mathcal{T}}^1}$ preserves range, source and color, $\mathcal{R}$ is $(r,s,d)$-preserving. By definition it satisfies $(KG1)$ (see Subsection \ref{ssec k-graphs}). Since $\mathcal{R}$ is a congruence, it also satisfies $(KG0)$. Consider the quotient category $\Lambda_{\mathcal{T}}:=G_{\mathcal{T}}^*/\mathcal{R}$. 

The following theorem provides conditions under which $\Lambda_{\mathcal{T}}$ is a $3$-graph.

\begin{thm}\label{th 3D textile to 3-graph}
Suppose $\mathcal{T}$ is a three-dimensional textile system presented by the diagram 
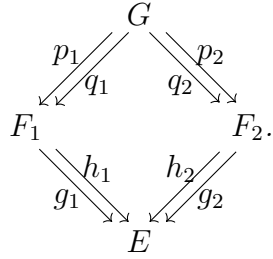
\[
\begin{tikzpicture}[scale=1]
\node[inner sep=2.5pt, circle] (B) at (-1.5,0) {$F_1$};
\node[inner sep=2.5pt, circle] (C) at (1.5,0) {$F_2$.};	
\node[inner sep=2.5pt, circle] (D) at (0,-1.5) {$E$};
\node[inner sep=2.5pt, circle] (A) at (0,1.5) {$G$}; 

\draw[transform canvas={xshift=0.6ex}, ->] (A) -- (B);
\draw[transform canvas={xshift=-0.6ex}, ->] (A) -- (B);

\draw[transform canvas={xshift=0.6ex}, ->] (A) -- (C);
\draw[transform canvas={xshift=-0.6ex}, ->] (A) -- (C);

\draw[transform canvas={xshift=0.6ex}, ->] (C) -- (D);
\draw[transform canvas={xshift=-0.6ex}, ->] (C) -- (D);

\draw[transform canvas={xshift=0.6ex}, ->] (B) -- (D);
\draw[transform canvas={xshift=-0.6ex}, ->] (B) -- (D);

\node at (-0.55,0.55) {$q_1$};
\node at (-0.95,0.95) {$p_1$};

\node at (0.55,0.55) {$q_2$};
\node at (0.95,0.95) {$p_2$};

\node at (-0.55,-0.55) {$h_1$};
\node at (-0.95,-0.95) {$g_1$};

\node at (0.55,-0.55) {$h_2$};
\node at (0.95,-0.95) {$g_2$};
\end{tikzpicture}
\]

Suppose $q_1$ has $s$-path lifting. Then $\Lambda_{\mathcal{T}}$ (defined above) is a $3$-graph if and only if the following conditions hold:

$(i)$ $g_1,g_2$ have unique $r$-path lifting;

$(ii)$ $h_1,h_2$ have unique $s$-path lifting;

$(iii)$ the diagrams 

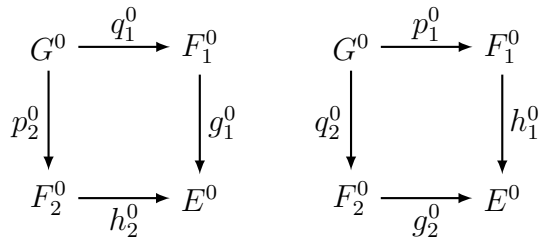
\[
\begin{tikzpicture}[scale=1]
\node[] (A1) at (0,0) {$G^0$};
\node[] (A2) at (2,0) {$F_1^0$};
\node[] (A3) at (4,0) {$G^0$};
\node[] (A4) at (6,0) {$F_1^0$};

\node[] (B1) at (0,-2) {$F_2^0$};
\node[] (B2) at (2,-2) {$E^0$};
\node[] (B3) at (4,-2) {$F_2^0$};
\node[] (B4) at (6,-2) {$E^0$};

\path[->, >=latex,thick] (A1) edge [] node[]{} (A2);
\path[->, >=latex,thick] (B1) edge [] node[]{} (B2);
\path[->, >=latex,thick] (A3) edge [] node[]{} (A4);
\path[->, >=latex,thick] (B3) edge [] node[]{} (B4);

\path[->, >=latex,thick] (A1) edge [] node[]{} (B1);
\path[->, >=latex,thick] (A2) edge [] node[]{} (B2);
\path[->, >=latex,thick] (A3) edge [] node[]{} (B3);
\path[->, >=latex,thick] (A4) edge [] node[]{} (B4);

\node at (1,0.3) {$q_1^0$};
\node at (5,0.3) {$p_1^0$};
\node at (1,-2.3) {$h_2^0$};
\node at (5,-2.3) {$g_2^0$};

\node at (-0.3,-1) {$p_2^0$};
\node at (2.3,-1) {$g_1^0$};
\node at (3.7,-1) {$q_2^0$};
\node at (6.3,-1) {$h_1^0$};
\end{tikzpicture}
\]
are pullback squares.
\end{thm}
\begin{proof}
Assume that the conditions $(i)$--$(iii)$ hold. Note that, the congruence $\mathcal{R}$ already satisfies $(KG0)$ and $(KG1)$. If we can show that the congruence $\mathcal{R}$ satisfies $(KG2)$ and $(KG3)$ as well, then $\Lambda_{\mathcal{T}}$ will be a $3$-graph (see Subsection \ref{ssec k-graphs}). Let $\mu=\mu_1\mu_2\in G_{\mathcal{T}}^2$ be a bi-colored path. If $d(\mu_1)=\ea$ and $d(\mu_2)=\eb$, then $\mu_1\in F_2^0$ and $\mu_2\in F_1^0$. Since $h_2(\mu_1)=r(\mu_1)=s(\mu_2)=g_1(\mu_2)$ and the first square in $(iii)$ is a pullback square, there exists a unique $x\in G^0$ such that $q_1^0(x)=\mu_2$ and $p_2^0(x)=\mu_1$. Then $\nu:=p_1^0(x)q_2^0(x)\in G_{\mathcal{T}}^2$, $d(p_1^0(x))=\eb$, $d(q_2^0(x))=\ea$ and $\mu\mathcal{R}\nu$. The uniqueness of $\nu$ follows from the uniqueness of $x$. If $d(\mu_1)=\eb$ and $d(\mu_1)=\ea$, then a similar argument using the second pullback square of $(iii)$ gives a unique $\eta=\eta_1\eta_2\in G_{\mathcal{T}}^2$ such that $d(\eta_1)=\ea$, $d(\eta_2)=\eb$ and $\mu\mathcal{R}\eta$. Using the unique path lifting properties of $g_1,g_2,h_1,h_2$ and following similar reasoning, we can uniquely swap the colors of other bi-colored paths and hence $\mathcal{R}$ satisfies $(KG2)$. We now show that $\mathcal{R}$ satisfies $(KG3)$. For this, we consider a tri-colored path $abc\in G_{\mathcal{T}}^3$ with $d(a)=\ea$, $d(b)=\eb$, $d(c)=\ec$, and show that the two paths $c''b''a''$ and $c_2b_2a_2$ with $d(c'')=d(c_2)=\ec$, $d(b'')=d(b_2)=\eb$, $d(a'')=d(a_2)=\ea$, obtained by following the routes
\begin{align*}
(1)~ &abc\longrightarrow ac'b'\longrightarrow c''a'b'\longrightarrow c''b''a'',\\
(2)~ &abc\longrightarrow b_1a_1c\longrightarrow b_1 c_1 a_2\longrightarrow c_2 b_2 a_2;
\end{align*}
are equal. Here $bc R^{23} c'b'$, $ac' R^{13} c''a'$, $a' b' R^{12}b'' a''$ and $ab R^{12} b_1 a_1$, $a_1 c R^{13} c_1 a_2$, $b_1 c_1 R^{23} c_2 b_2$. From the first pullback square of $(iii)$, there exists a unique $\alpha\in G^0$ such that $p_2^0(\alpha)=a$ and $q_1^0(\alpha)=b$. Again, by unique $s$-path lifting of $h_1$, there exists a unique $\beta\in F_1^1$ such that $s_{F_1}(\beta)=b$ and $h_1^1(\beta)=c$. We now apply the $s$-path lifting of $q_1$ to guarantee the existence of an edge $\lambda\in G^1$ such that $s_G(\lambda)=\alpha$ and $q_1^1(\lambda)=\beta$. By Lemma \ref{lem other pullbacks}, this $\lambda$ is unique with this property. We now use this $\lambda$ to express the colored edges $a'',b'',c''$, $a_2,b_2,c_2$ found at the end of the routes $(1)$ and $(2)$. Let us first follow route $(1)$. From the definition of $R^{23}$, it follows that $b'=r_{F_1}(\beta)$ and $c'=g_1^1(\beta)$. Note that, \[s_{F_2}(p_2^1(\lambda))=p_2^0(s_G(\lambda))=p_2^0(\alpha)=a,\] and \[h_2^1(p_2^1(\lambda))=g_1^1(q_1^1(\lambda))=g_1^1(\beta)=c'.\] Since $h_2$ has unique $s$-path lifting, $p_2^1(\lambda)$ is the only path with this property. Using the definition of $R^{13}$, we now have $a'=r_{F_2}(p_2^1(\lambda))$ and \textcolor{blue}{$c''=g_2^1(p_2^1(\lambda))$}. Again, \[q_1^0(r_G(\lambda))=r_{F_1}(q_1^1(\lambda))=r_{F_1}(\beta)=b'\] and \[p_2^0(r_G(\lambda))=r_{F_2}(p_2^1(\lambda))=a'.\] The definition of $R^{12}$ then implies \textcolor{blue}{$b''=p_1^0(r_G(\lambda))$} and \textcolor{blue}{$a''=q_2^0(r_G(\lambda))$}. We now consider the other route. Since $p_2^0(\alpha)=a$, $q_1^0(\alpha)=b$, so $b_1=p_1^0(\alpha)$ and $a_1=q_2^0(\alpha)$. We observe that $c_1=g_2^1(q_2^1(\lambda))$ and \textcolor{blue}{$a_2=r_{F_2}(q_2^1(\lambda))$}. This follows using the definition of $R^{13}$ and noting that \[s_{F_2}(q_2^1(\lambda))=q_2^0(s_G(\lambda))=q_2^0(\alpha)=a_1,\] and \[h_2^1(q_2^1(\lambda))=h_1^1(q_1^1(\lambda))=h_1^1(\beta)=c.\] Again, since \[s_{F_1}(p_1^1(\lambda))=p_1^0(s_G(\lambda))=p_1^0(\alpha)=b_1,\] and \[h_1^1(p_1^1(\lambda))=g_2^1(q_2^1(\lambda))=c_1,\] it follows that \textcolor{blue}{$c_2=g_1^1(p_1^1(\lambda))$} and \textcolor{blue}{$b_2=r_{F_1}(p_1^1(\lambda))$}. Finally, we have, $a''=q_2^0(r_G(\lambda))=r_{F_2}(q_2^1(\lambda))=a_2$, $b''=p_1^0(r_G(\lambda))=r_{F_1}(p_1^1(\lambda))=b_2$ and $c''=g_2^1(p_2^1(\lambda))=g_1^1(p_1^1(\lambda))=c_2$ (For the convenience of tracking, we have highlighted $a'',b'',c''$ and $a_2,b_2,c_2$ above in blue). Thus, $c''b''a''=c_2b_2a_2$ and we are done. By Lemma \ref{lem other pullbacks}, we have five other pullback squares described in $(2)$--$(6)$ of Remark \ref{rem quivalence of the pullbacks}. Using each of these and following the same approach as above, we can show the associativity for tri-colored paths with other five types of color combinations. Therefore, $\mathcal{R}$ satisfies $(KG3)$ and consequently, $\Lambda_{\mathcal{T}}=G_{\mathcal{T}}^*/\mathcal{R}$ is a $3$-graph. 

Conversely, suppose $\Lambda_{\mathcal{T}}$ is a $3$-graph. Let $w\in F_1^0$ and $e\in E^1$ be such that $g_1(w)=r_E(e)$. Then $ew\in \Lambda_{\mathcal{T}}^{\eb+\ec}$. By the unique factorization property, there exists unique $e'\in \Lambda_{\mathcal{T}}^{\ec}=E^1$ and $w'\in \Lambda_{\mathcal{T}}^{\eb}=F_1^0$ such that $ew=w'e'$. This implies $ew R^{23} w'e'$ and so there exists $x\in F_1^1$ such that $g_1^1(x)=e$ and $r_{F_1}(x)=w$. The uniqueness of $x$ follows from the unique factorization. Hence, $g_1$ has $r$-path lifting. Using a similar argument, we can show that $g_2$ has unique $r$-path lifting, $h_1,h_2$ have unique $s$-path lifting and the squares in $(iii)$ are indeed pullback squares. 
\end{proof}

In a three-dimensional textile system, one can easily find two two-dimensional textile systems with a common base graph (see Remark \ref{rem vieweing 2D textile as a 3D textile}). We next provide a certain condition involving two such textile systems, which enables us to build a three-dimensional textile in a natural way. The following construction can be seen as a natural extension of the idea of building an $LR$-textile system using the coordinate graphs of a $2$-graph \cite[Proposition 3.7]{Deaconu} or equivalently, a pair of commuting non-negative integral matrices of same order (for more details, see \cite[\S 9.1]{Matsumoto}).

Recall that in the category \textbf{Dgr}, the \emph{pullback} of the diagram 
\[
\begin{tikzpicture}[scale=1]
\node[] (A) at (0,0) {$F$};	
\node[] (B) at (1.5,0) {$E$};
\node[] (C) at (3,0) {$G$};
	
\path[->,black, >=latex,thick] (A) edge [left] node[above=0.05cm]{} (B);
\path[->,black, >=latex,thick] (C) edge [left] node[above=0.05cm]{} (B);

\node at (0.75,0.3) {$f$};
\node at (2.25,0.3) {$g$};
\end{tikzpicture}
\] 
is given by the directed graph $F{{}_{f}\times_{g}} G$ where \[(F{{}_{f}\times_{g}} G)^0:=\{(u,v)\in F^0\times G^0~|~f^0(u)=g^0(v)\},~~(F{{}_{f}\times_{g}} G)^1:=\{(x,y)\in F^1\times G^1~|~f^1(x)=g^1(y)\},\] \[r((x,y)):=(r_F(x),r_G(y)),~~s((x,y)):=(s_F(x),s_G(y)).\]

Suppose $T_1=(K,E,g_1,h_1)$ and $T_2=(L,E,g_2,h_2)$ are two-dimensional textile systems. Consider the pullback graphs $K{{}_{h_1}\times_{g_2}}L$, $L{{}_{h_2}\times_{g_1}}K$. Suppose we have an isomorphism of directed graphs \[\theta: K{{}_{h_1}\times_{g_2}}L\longrightarrow L{{}_{h_2}\times_{g_1}}K.\] We can use this isomorphism to set up a three-dimensional textile system. Consider the following diagram:
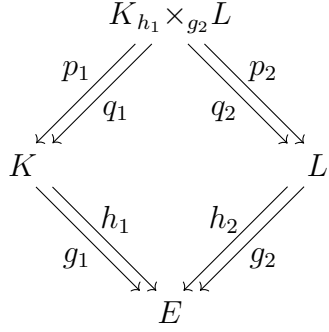
\[
\begin{tikzpicture}[scale=1.3]
\node[] (B) at (-1.5,0) {$K$};
\node[] (C) at (1.5,0) {$L$};	
\node[] (D) at (0,-1.5) {$E$};
\node[] (A) at (0,1.5) {$K{{}_{h_1}\times_{g_2}}L$}; 

\draw[transform canvas={xshift=0.6ex}, ->] (A) -- (B);
\draw[transform canvas={xshift=-0.6ex}, ->] (A) -- (B);

\draw[transform canvas={xshift=0.6ex}, ->] (A) -- (C);
\draw[transform canvas={xshift=-0.6ex}, ->] (A) -- (C);

\draw[transform canvas={xshift=0.6ex}, ->] (C) -- (D);
\draw[transform canvas={xshift=-0.6ex}, ->] (C) -- (D);

\draw[transform canvas={xshift=0.6ex}, ->] (B) -- (D);
\draw[transform canvas={xshift=-0.6ex}, ->] (B) -- (D);

\node at (-0.55,0.55) {$q_1$};
\node at (-0.95,0.95) {$p_1$};

\node at (0.55,0.55) {$q_2$};
\node at (0.95,0.95) {$p_2$};

\node at (-0.55,-0.55) {$h_1$};
\node at (-0.95,-0.95) {$g_1$};

\node at (0.55,-0.55) {$h_2$};
\node at (0.95,-0.95) {$g_2$};
\end{tikzpicture}
\]
where $p_1:=(\pi_1,\pi_1)$, $q_1:=(\pi_2\circ \theta^0, \pi_2\circ \theta^1)$, $p_2:=(\pi_1\circ \theta^0,\pi_1\circ \theta^1)$, $q_2:=(\pi_2,\pi_2)$; $\pi_1,\pi_2$ being projections on the first and the second component, respectively. It is not hard to observe that the above diagram defines a three-dimensional textile if and only if $g_1(x)=g_2(y')$ and $h_2(y)=h_1(x')$ whenever $\theta((x,y))=(y',x')$. We denote the resulting three-dimensional textile system by $\mathcal{T}(T_1,T_2,\theta)$. It is interesting to characterize when this textile system gives a $3$-graph.
\begin{prop}\label{pro 2-textiles giving a 3-graph}
With the notations described as above, $\mathcal{T}(T_1,T_2,\theta)$ satisfies $(i)$, $(ii)$ of Theorem \ref{th 3D textile to 3-graph} if and only if $T_1$, $T_2$ are $LR$-textile systems. In this case, $\Lambda_{\mathcal{T}(T_1,T_2,\theta)}$ is a $3$-graph if and only if $q_1$ has unique $s$-path lifting.     
\end{prop}
\begin{proof}
The first part is straightforward, following the definition of an $LR$-textile system. For the second part, we first observe that $\mathcal{T}(T_1,T_2,\theta)$ satisfies the condition $(iii)$ of Theorem \ref{th 3D textile to 3-graph}. To show that the first diagram there, is a pullback square, suppose we have two maps $\alpha:Q\longrightarrow K^0$, $\beta:Q\longrightarrow L^0$ such that $g_1^0\circ \alpha=h_2^0\circ \beta$. Then $(\beta(x),\alpha(x))\in (L{{}_{h_2}\times_{g_1}}K)^0$. Define a map 
\begin{align*}
    \mathfrak{h}:Q &\longrightarrow (K{{}_{h_1}\times_{g_2}}L)^0\\
    x &\longmapsto \theta^{-1}(\beta(x),\alpha(x)).
\end{align*}
Then $q_1^0(\mathfrak{h}(x))=\alpha(x)$ and $p_2^0(\mathfrak{h}(x))=\beta(x)$ for all $x\in Q$. Since $\theta$ is an isomorphism, it follows that $\mathfrak{h}$ is the unique map with these properties, and we are done. Similarly, we can show that the second square in $(iii)$ of Theorem \ref{th 3D textile to 3-graph} is a pullback. The `if' part of the second statement now follows from Theorem \ref{th 3D textile to 3-graph}. For the converse, assume that $\Lambda_{\mathcal{T}(T_1,T_2,\theta)}$ is a $3$-graph. Let $(u,v)\in (K{{}_{h_1}\times_{g_2}}L)^0$, $e\in K^1$ be such that $q_1^0((u,v))=s_{K}(e)$. Let $\theta^0((u,v))=(v',u')$. Then $u'=s_K(e)$. Note that $uv,v'u'\in \Lambda_{\mathcal{T}(T_1,T_2,\theta)}^{\ea+\eb}$, and the definition of the relation $R^{12}$ on the set of $\ea-\eb$ bi-colored paths, says that $uv=v'u'$ in $\Lambda_{\mathcal{T}(T_1,T_2,\theta)}$. Consider the tri-colored path \[\lambda:=v'u'h_1^1(e)\in \Lambda_{\mathcal{T}(T_1,T_2,\theta)}.\] By the factorization rules of $\Lambda_{\mathcal{T}(T_1,T_2,\theta)}$, $\lambda(0,\eb+\ec)$ (see the definitions of $R^{23}$ and $R^{13}$) corresponds to a unique $x\in K^1$ and $\lambda(\eb,\ea+\eb+\ec)$ corresponds to a unique $y\in L^1$. Now, \[h_1^1(x)=(\lambda(0,\eb+\ec))(\eb,\eb+\ec)=\lambda(\eb,\eb+\ec)=(\lambda(\eb,\ea+\eb+\ec))(0,\ec)=g_2^1(y),\] and so $(x,y)\in (K{{}_{h_1}\times_{g_2}}L)^1$. Note that,
\begin{align*}
s((x,y))&=(s_K(x),s_L(y))\\
&=((\lambda(0,\eb+\ec))(0,\eb),(\lambda(\eb,\ea+\eb+\ec))(0,\ea))\\
&=(\lambda(0,\eb),\lambda(\eb,\ea+\eb))=(u,v).
\end{align*}
Since $\theta$ is a graph morphism, $s(\theta^1((x,y)))=\theta^0(s((x,y)))$. Thus, \[s_K(\pi_2(\theta^1((x,y))))=\pi_2(\theta^0(s((x,y))))=u'=s_K(e)\] and \[h_1^1(\pi_2(\theta^1((x,y))))=h_2^1(y)=(\lambda(\eb,\ea+\eb+\ec))(\ea,\ea+\ec)=\lambda(\ea+\eb,\ea=\eb+\ec)=h_1^1(e).\] Now, since $h_1$ has $s$-path lifting, it follows that $q_1^1((x,y))=\pi_2(\theta^1((x,y)))=e$. Therefore, $q_1$ has $s$-path lifting. That the lifting is unique now follows from Lemma \ref{lem other pullbacks}.
\end{proof}

\section{Homology of a three-dimensional textile system}\label{sec homology of 3D textile}
In \cite{KPW}, the authors introduced the homology and cohomology theory for the usual two-dimensional textile system and relate these to the analogous notions for a $2$-graph. In this concluding section of the paper, we accomplish the following: $(i)$ we obtain a graded analogue of their result connecting homology of textiles and $2$-graphs, and $(ii)$ we define the homology of a three-dimensional textile system and relate this to the homology of the associated $3$-graphs.

Let $T=(F,E,p,q)$ be a two-dimensional textile system. Define new graphs $\hat{F}=(F^0\times \mathbb{Z}^2,F^1\times \mathbb{Z}^2,\hat{r}_F,\hat{s}_F)$ and $\hat{E}=(E^0\times \mathbb{Z}^2.E^1\times \mathbb{Z}^2,\hat{r}_E,\hat{s}_E)$ where \[\hat{r}_F((f,\N)):=(r_F(f),\N+\ea),\hat{s}_F((f,\N)):=(s_F(f),\N),\hat{r}_E((e,\N)):=(r_E(e),\N+\ea),\hat{s}_E((e,\N)):=(s_E(e),\N)\] for all $(f,\N)\in F^1\times \mathbb{Z}^2$ and $(e,\N)\in E^1\times \mathbb{Z}^2$. Also, define homomorphisms $\hat{p},\hat{q}:\hat{F}\longrightarrow \hat{E}$ as follows \[\hat{p}((x,\N)):=(p(x),\N),~~\hat{q}((x,\N)):=(q(x),\N+\eb)\] for all $x\in F^0\sqcup F^1$ and $\N\in \mathbb{Z}^2$. 

\begin{prop}\label{pro the covering textile}
Let $T=(F,E,p,q)$ be any textile system. Then $\hat{T}:=(\hat{F},\hat{E},\hat{p},\hat{q})$ is also a textile system. Moreover $T$ is $LR$ if and only if $\hat{T}$ is $LR$. 
\end{prop}
\begin{proof}
That $\hat{p}$ and $\hat{q}$ are graph homomorphisms can easily be verified. For example, if $(f,\N)\in \hat{F}^1=F^1\times \mathbb{Z}^2$, then \[\hat{r}_E((\hat{q}^1((f,\N)))=\hat{r}_E((q^1(f),\N+\eb))=(r_E(q^1(f)),\N+\eb+\ea)=(q^0(r_F(f)),\N+\ea+\eb)=\hat{q}^0(\hat{r}_F((f,\N))).\] Thus, $\hat{r}_E\circ \hat{q}^1=\hat{q}^0\circ \hat{r}_F$. The injectivity of $(\hat{s}_F,\hat{r}_F,\hat{p}^1,\hat{q}^1)$ is evident from the injectivity of $(s_F,r_F,p^1,q^1)$. Hence, $\hat{T}$ is a textile system. Now, suppose $T$ is left-resolving. Choose any $(w,\N)\in F^0\times \mathbb{Z}^2$ and $(e,\M)\in E^1\times \mathbb{Z}^2$ such that $\hat{p}((w,\N))=\hat{r}_E((e,\M))$. This implies $(p(w),\N)=(r_E(e),\M+\ea)$. Since, $p$ has unique $r$-path lifting, there exists a unique $f\in F^1$ such that $p(f)=e$ and $r_F(f)=w$. Then $\hat{p}((f,\M))=(e,\M)$ and $\hat{r}_F((f,\M))=(w,\M+\ea)=(w,\N)$. Obviously, $(f,\M)$ is unique edge in $\hat{F}$ with this property. Therefore, $\hat{p}$ has unique $r$-path lifting. Similarly, it can be shown that $\hat{q}$ has unique $s$-path lifting, which says that $\hat{T}$ is $LR$. Conversely suppose $\hat{T}$ is $LR$. To show that $p$ has unique $r$-path lifting, choose any $v\in F^0$ and $e\in E^1$ such that $p(v)=r_E(e)$. Then $\hat{p}((v,\mathbf{e}_1))=\hat{r}_E((e,0))$. By unique $r$-path lifting of $\hat{p}$, there exists a unique $(f,\N)\in \hat{F}^1$ such that $\hat{p}((f,\N))=(e,0)$ and $\hat{r}_F((f,\N))=(v,\mathbf{e}_1)$. Then $\N=0$ and $f\in F^1$ is the unique edge with $p(f)=e$ and $r_F(f)=v$, showing that $p$ has unique $r$-path lifting. That $q$ has unique $s$-path lifting follows similarly.   
\end{proof}
In view of the above proposition, we now define the \emph{graded homology} of a textile system.
\begin{dfn}\label{def graded homology of a TS}
Suppose $T=(F,E,p,q)$ is any textile system. 

$(i)$ The textile $\hat{T}$ of Proposition \ref{pro the covering textile} is called the \emph{covering textile} of $T$. 

$(ii)$ The \emph{graded $n$-th homology} of $T$ is defined as \[H_n^{\gr}(T):=H_n(\hat{T}).\]    
\end{dfn}
\begin{rmk}\label{rem covering textile and covering graphs}
One can view the top graph $\hat{F}$ and the bottom graph $\hat{E}$ of a covering textile as the usual skew-product or covering graphs (see \cite[Definition 2.1]{KPact}) of $F$ and $E$ stacked along the $y$-axis, or in other words, $\hat{F}$ (resp., $\hat{E}$) is the disjoint union of $|\mathbb{Z}|$ many isomorphic copies of the covering graph $F\times_1 \mathbb{Z}$ (resp., $E\times_1 \mathbb{Z}$). To be precise, for any fixed $j\in \mathbb{Z}$, consider the directed graph $F_j$ with \[F_j^0:=F^0\times (\mathbb{Z}\times \{j\}), F_j^1:=F^1\times (\mathbb{Z}\times \{j\}), r((f,(i,j)):=(r(f),(i+1,j)), s((f,(i,j)):=(s(f),(i,j)).\] Similarly, we have $E_j$. Then is is easy to see that $F_j\cong F\times_1 \mathbb{Z}$ and $E_j\cong E\times_1 \mathbb{Z}$. Moreover, $\hat{F}=\displaystyle{\bigsqcup_{j\in \mathbb{Z}}}~F_j$ and $\hat{E}=\displaystyle{\bigsqcup_{j\in \mathbb{Z}}}~E_j$. The homomorphism $\hat{p}$ can be though exactly as the homomorphism $p:F\longrightarrow E$ doing its work in each level $j$, whereas $\hat{q}$ is the homomorphism $q:F\longrightarrow E$ transferring everything one level up vertically. 
\end{rmk}
The graded homology groups come with a natural action of $\mathbb{Z}^2$. For each $\mathbf{q}\in \mathbb{Z}^2$, if we define $\eta_\mathbf{q}^F:\hat{F}\longrightarrow \hat{F}$ and $\eta_\mathbf{q}^E:\hat{E}\longrightarrow \hat{E}$ simply by $\eta_\mathbf{q}^F((x,\N)):=(x,\N+\mathbf{q})$ and $\eta_\mathbf{q}^E((y,\M)):=(y,\M+\mathbf{q})$ for all $(x,\N)\in \hat{F}^0\sqcup \hat{F}^1$ and $(y,\M)\in \hat{E}^0\sqcup \hat{E}^1$, then it is not hard to observe that the pair $\eta_\mathbf{q}:=(\eta_\mathbf{q}^F,\eta_\mathbf{q}^E)$ is a textile automorphism of $\hat{T}$. Thus, we have a canonical action of $\mathbb{Z}^2$ on $\hat{T}$. This in turn makes each abelian group $C_n(\hat{T})$ a $\mathbb{Z}^2$-module and each homology differential $\partial_n^{\hat{T}}$ a $\mathbb{Z}^2$-module homomorphism. Therefore, for each $n\in \mathbb{N}$, we can view the graded homology $H_n^{\gr}(T)$ as a $\mathbb{Z}^2$-module. This gives rise to a dynamics inside each graded homology group, which was missing in the usual (non-graded) homology groups. 

When $T$ is an $LR$-textile system, $\hat{T}$ is also $LR$ by Proposition \ref{pro the covering textile} and hence, we can talk about the associated $2$-graph $\Lambda_{\hat{T}}$. The next proposition exhibits an interesting connection between $\Lambda_{\hat{T}}$ and $\Lambda_T$. 

\begin{prop}\label{pro covering textile and skew-product graph}
Suppose $T$ is an $LR$-textile system. Then there is a $2$-graph isomorphism between $\Lambda_{\hat{T}}$ and $\overline{\Lambda_T}$, where $\overline{\Lambda_T}$ is the skew-product of $\Lambda_T$ with $\mathbb{Z}^2$. 
\end{prop}
\begin{proof}
Since we are dealing with $2$-graphs, it suffices to show that $\Lambda_{\hat{T}}$ and $\overline{\Lambda_T}$ have the same $1$-skeleton and the same factorization rules for bicolored paths. Let $\mathsf{G}$ and $\mathsf{H}$ be the $1$-skeletons of $\Lambda_{\hat{T}}$ and $\overline{\Lambda_T}$ respectively. Then $\mathsf{G}^0=\hat{E}^0=E^0\times \mathbb{Z}^2=\Lambda_T^0\times \mathbb{Z}^2=\mathsf{H}^0$ and \[\mathsf{G}^1=\hat{F}^0\sqcup \hat{E}^1=(F^0\times \mathbb{Z}^2)\sqcup (E^1\times \mathbb{Z}^2)=(F^0\sqcup E^1)\times \mathbb{Z}^2=\overline{\Lambda_T}^\mathbf{1}\times \mathbb{Z}^2=\mathsf{H}^1.\] To show that the adjacency is preserved, let us choose any $x\in \mathsf{G}^1=\mathsf{H}^1$. Then either $x=(e,\N)$ for some $(e,\N)\in E^1\times \mathbb{Z}^2$ or $x=(w,\M)$ for some $(w,\M)\in F^0\times \mathbb{Z}^2$. If $x=(e,\N)$, then \[s_\mathsf{G}(x)=\hat{r}_E((e,\N))=(r_E(e),\N+\ea)=(s_{\Lambda_T}(e),\N+\ea)=s_\mathsf{H}(x)\] since $d(e)=\ea$ in $\Lambda_T$. Also, \[r_\mathsf{G}(x)=\hat{s}_E((e,\N))=(s_E(e),\N)=(r_{\Lambda_T}(e),\N)=r_\mathsf{H}(x).\] On the other hand, if $x=(w,\M)$, then \[s_\mathsf{G}(x)=\hat{q}^0((w,\M))=(q^0(w),\M+\eb)=(s_{\Lambda_T}(w),\M+\eb)=s_\mathsf{H}(x)\] since $d(w)=\eb$ in $\Lambda_T$. Again, \[r_\mathsf{G}(x)=\hat{p}^0((w,\M))=(p^0(w),\M)=(r_{\Lambda_T}(w),\M)=r_\mathsf{H}(x).\] Also note that the colors of the edges of $\mathsf{G}$ and $\mathsf{H}$ agree. Therefore, the $1$-skeletons of $\Lambda_{\hat{T}}$ and $\overline{\Lambda_T}$ are exactly the same. Now, we need to show that their factorization rules are also the same. For this, we take $(w,\M),(w',\M')\in F^0\times \mathbb{Z}^2$ and $(e,\N),(e',\N')\in E^1\times \mathbb{Z}^2$. If $(w,\M)(e,\N)=(e',\N')(w',\M')$ in $\Lambda_{\hat{T}}$, then there is an $(f,\mathbf{t})\in \hat{F}^1=F^1\times \mathbb{Z}^2$ such that $s_{\hat{F}}((f,\mathbf{t}))=(w,\M)$, $r_{\hat{F}}((f,\mathbf{t}))=(w',\M')$, $\hat{q}((f,\mathbf{t}))=(e,\N)$ and $\hat{p}((f,\mathbf{t}))=(e',\N')$. This implies that $s_F(f)=w$, $r_F(f)=w'$, $q^1(f)=e$, $p^1(f)=e'$, $\mathbf{t}=\M=\N'$, $\M'=\mathbf{t}+\ea$ and $\N=\mathbf{t}+\eb$. In view of the first four equalities, we have $we=e'w'$ in $\Lambda_T$. Combining this with the equality $\M=\N'$, we have \[ (w,\M)(e,\N)=(we,\M)=(e'w',\N')=(e',\N')(w',\M')\] in $\overline{\Lambda_T}$. Conversely, if $(w,\M)(e,\N)=(e',\N')(w',\M')$ in $\overline{\Lambda_T}$, then the above argument together with the uniqueness of factorization in $\overline{\Lambda_T}$ implies $(w,\M)(e,\N)=(e',\N')(w',\M')$ in $\Lambda_{\hat{T}}$. This completes the proof.  
\end{proof}
We remark that the above proposition can alternatively be proved by realizing $\hat{T}$ as $T_{\overline{\Lambda_T}}$. Like covering textile, the skew-product of a $2$-graph $\Lambda$ is also equipped with a natural $\mathbb{Z}^2$-action, in fact $\mathbb{Z}^2$ acts freely on $\overline{\Lambda}$ by $^\mathbf{q}(\lambda,\N)=(\lambda,\N+\mathbf{q})$ (see \cite[Remark 5.6]{Kumjian-Pask}). With respect to this action, each member of the chain complex $(C_*(\overline{\Lambda}),\partial_*^{\overline{\Lambda}})$ becomes a $\mathbb{Z}^2$-module and each boundary map is in fact a $\mathbb{Z}^2$-module map. Thus, the $n$-th homology group $H_n(\overline{\Lambda})$ is also a $\mathbb{Z}^2$-module. This simple fact together with \cite[Corollary 5.13]{KPW} and Proposition \ref{pro covering textile and skew-product graph} now gives the following theorem showcasing the connection between textile systems and $2$-graphs in light of graded homology.
\begin{thm}\label{th connecting graded homology with homology of skew-product}
Let $T$ be an $LR$-textile system. Then there is a $\mathbb{Z}^2$-module isomorphism between the graded homology $H_n^{\gr}(T)$ and the usual homology $H_n(\overline{\Lambda_T})$. 
\end{thm}

We now define homology for a three-dimensional textile system. In what follows, $\mathcal{T}$ denotes a three-dimensional textile presented by the following diagram
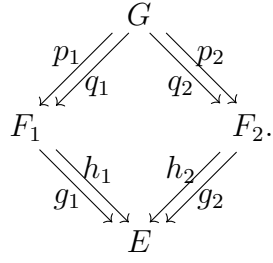
\[
\begin{tikzpicture}[scale=1]
\node[inner sep=2.5pt, circle] (B) at (-1.5,0) {$F_1$};
\node[inner sep=2.5pt, circle] (C) at (1.5,0) {$F_2$.};	
\node[inner sep=2.5pt, circle] (D) at (0,-1.5) {$E$};
\node[inner sep=2.5pt, circle] (A) at (0,1.5) {$G$}; 

\draw[transform canvas={xshift=0.6ex}, ->] (A) -- (B);
\draw[transform canvas={xshift=-0.6ex}, ->] (A) -- (B);

\draw[transform canvas={xshift=0.6ex}, ->] (A) -- (C);
\draw[transform canvas={xshift=-0.6ex}, ->] (A) -- (C);

\draw[transform canvas={xshift=0.6ex}, ->] (C) -- (D);
\draw[transform canvas={xshift=-0.6ex}, ->] (C) -- (D);

\draw[transform canvas={xshift=0.6ex}, ->] (B) -- (D);
\draw[transform canvas={xshift=-0.6ex}, ->] (B) -- (D);

\node at (-0.55,0.55) {$q_1$};
\node at (-0.95,0.95) {$p_1$};

\node at (0.55,0.55) {$q_2$};
\node at (0.95,0.95) {$p_2$};

\node at (-0.55,-0.55) {$h_1$};
\node at (-0.95,-0.95) {$g_1$};

\node at (0.55,-0.55) {$h_2$};
\node at (0.95,-0.95) {$g_2$};
\end{tikzpicture}
\]
We first define certain free abelian groups as follows
\begin{align*}
C_0(\mathcal{T}):=&~\mathbb{Z}E^0;\\
C_1(\mT):=&~\mathbb{Z}(F_1^0\sqcup f_2^0\sqcup E^1);\\
C_2(\mT):=&~\mathbb{Z}(G^0\sqcup F_1^1\sqcup F_2^1);\\
C_3(\mT):=&~\mathbb{Z}G^1.
\end{align*}
We set $C_n(\mT):=0$ for all $n\ge 4$. Now, we define the \emph{homology differentials} or the \emph{boundary maps} as the group homomorphisms $\partial_i^\mT:C_i(\mT)\longrightarrow C_{i-1}(\mT)$ for $i\ge 1$, where 
\begin{align*}
\partial_1^\mT(x):=
	\left\{
	\begin{array}{lll}
		(h_1^0-g_1^0)(x)  & \mbox{if } x\in F_1^0,\\
		(h_2^0-g_2^0)(x)  & \mbox{if } x\in F_2^0,\\
        (r_E-s_E)(x)  & \mbox{if } x\in E^1;
	\end{array}
	\right.
\end{align*}
\[
\partial_2^\mT(x):=
	\left\{
	\begin{array}{lll}
		(p_2^0+q_1^0-p_1^0-q_2^0)(x)  & \mbox{if } x\in G^0,\\
		(s_{F_1}+h_1^1-g_1^1-r_{F_1})(x)  & \mbox{if } x\in F_1^1,\\
        (s_{F_2}+h_2^1-g_2^1-r_{F_2})(x)  & \mbox{if } x\in F_2^1;
	\end{array}
	\right.
\]
\[
\partial_3^\mT:=r_G+p_2^1+q_1^1-s_G-p_1^1-q_2^1
\]
and for all $i\ge 4$, $\partial_i^\mT=\theta$ (the zero homomorphism) for obvious reason. 
\begin{lem}\label{lem chain complex}
Let $\mT$ be any three-dimensional textile system. Then the sequence of abelian groups and homomorphisms \[0\overset{\partial_0^\mT=\theta}{\longleftarrow} C_0(\mT)\overset{\partial_1^\mT}{\longleftarrow} C_1(\mT)\overset{\partial_2^\mT}{\longleftarrow} C_2(\mT)\overset{\partial_3^\mT}{\longleftarrow}C_3(\mT)\overset{\theta}{\longleftarrow} 0\] forms a chain complex. 
\end{lem}
\begin{proof}
We need to show that $\image(\partial_i^\mT)\subseteq \Ker(\partial_{i-1}^\mT)$ for all $i\ge 1$. It suffices to show that $\image(\partial_3^\mT)\subseteq \Ker(\partial_2^\mT)$ and $\image(\partial_2^\mT)\subseteq \Ker(\partial_1^\mT)$ as the other implications are obvious. Let $x\in G^1$. Then 
\allowdisplaybreaks
\begin{align*}
&~\partial_2^\mT(\partial_3^\mT(x))\\
=&~\partial_2^\mT(r_G(x)+p_2^1(x)+q_1^1(x)-s_G(x)-p_1^1(x)-q_2^1(x))\\
=& \left(p_2^0(r_G(x))+q_1^0(r_G(x))-p_1^0(r_G(x))-q_2^0(r_G(x))\right)+\left(s_{F_2}(p_2^1(x))+h_2^1(p_2^1(x))-g_2^1(p_2^1(x))-r_{F_2}(p_2^1(x))\right)+\\
&\left(s_{F_1}(q_1^1(x))+h_1^1(q_1^1(x))-g_1^1(q_1^1(x))-r_{F_1}(q_1^1(x))\right)-\left(p_2^0(s_G(x))+q_1^0(s_G(x))-p_1^0(s_G(x))-q_2^0(s_G(x))\right)-\\
&\left(s_{F_1}(p_1^1(x))+h_1^1(p_1^1(x))-g_1^1(p_1^1(x))-r_{F_1}(p_1^1(x))\right)-\left(s_{F_2}(q_2^1(x))+h_2^1(q_2^1(x))-g_2^1(q_2^1(x))-r_{F_2}(q_2^1(x))\right)\\
=& \left(p_2^0(r_G(x))-r_{F_2}(p_2^1(x))\right)+\left(q_1^0(r_G(x))-r_{F_1}(q_1^1(x))\right)+\left(r_{F_1}(p_1^1(x))-p_1^0(r_G(x))\right)+\left(r_{F_2}(q_2^1(x))-q_2^0(r_G(x))\right)\\
&+\left(s_{F_2}(p_2^1(x))-p_2^0(s_G(x))\right)+\left(h_2^1(p_2^1(x))-g_1^1(q_1^1(x))\right)+\left(g_1^1(p_1^1(x))-g_2^1(p_2^1(x))\right)+\left(s_{F_1}(q_1^1(x))-q_1^0(s_G(x))\right)\\
&+\left(h_1^1(q_1^1(x))-h_2^1(q_2^1(x))\right)+\left(p_1^0(s_G(x))-s_{F_1}(p_1^1(x))\right)+\left(q_2^0(s_G(x))-s_{F_2}(q_2^1(x))\right)+\left(h_1^1(p_1^1(x))-g_2^1(q_2^1(x))\right)\\
=&~ 0
\end{align*}
using the defining conditions of a three-dimensional textile. Hence, $\image(\partial_3^\mT)\subseteq \Ker(\partial_2^\mT)$. Now, for any $x\in G^0$, 
\begin{align*}
&~\partial_1^\mT(\partial_2^\mT(x))\\
=&~\partial_1^\mT(p_2^0(x)+q_1^0(x)-p_1^0(x)-q_2^0(x))\\
=&\left(h_2^0(p_2^0(x))-g_2^0(p_2^0(x))\right)+\left(h_1^0(q_1^0(x))-g_1^0(q_1^0(x))\right)-\left(h_1^0(p_1^0(x))-g_1^0(p_1^0(x))\right)-\left(h_2^0(q_2^0(x))-g_2^0(q_2^0(x))\right)\\
=&\left(h_2^0(p_2^0(x))-g_1^0(q_1^0(x))\right)+\left(g_1^0(p_1^0(x))-g_2^0(p_2^0(x))\right)+\left(h_1^0(q_1^0(x))-h_2^0(q_2^0(x))\right)+\left(g_2^0(q_2^0(x))-h_1^0(p_1^0(x))\right)\\
=&~0.
\end{align*}
Similarly, one can check that $\partial_1^\mT(\partial_2^\mT(x))=0$ for all $x\in F_1^1\sqcup F_2^1$. Therefore, $\image(\partial_2^\mT)\subseteq \Ker(\partial_1^\mT)$. Hence the lemma. 
\end{proof}
In view of the above lemma, we can now define the homology of a three-dimensional textile. 
\begin{dfn}\label{def homology of 3D textile}
Suppose $\mT$ is any three-dimensional textile system. The chain complex of Lemma \ref{lem chain complex} is denoted by $(C_*(\mT),\partial_*^\mT)$. The \emph{$n$-th integral homology} of $\mT$ is defined as \[H_n(\mT):=\Ker(\partial_n^\mT)/\image(\partial_{n+1}^\mT)\] for all $n\in \mathbb{N}$. 
\end{dfn}
We now state the main result of this section. 
\begin{thm}\label{th the homologies are isomorphic}
Suppose $\mathcal{T}$ is a three-dimensional textile system presented by the diagram 
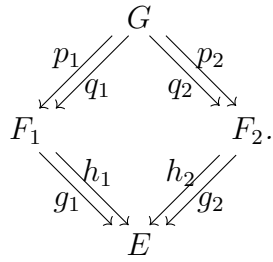
\[
\begin{tikzpicture}[scale=1]
\node[inner sep=2.5pt, circle] (B) at (-1.5,0) {$F_1$};
\node[inner sep=2.5pt, circle] (C) at (1.5,0) {$F_2$.};	
\node[inner sep=2.5pt, circle] (D) at (0,-1.5) {$E$};
\node[inner sep=2.5pt, circle] (A) at (0,1.5) {$G$}; 

\draw[transform canvas={xshift=0.6ex}, ->] (A) -- (B);
\draw[transform canvas={xshift=-0.6ex}, ->] (A) -- (B);

\draw[transform canvas={xshift=0.6ex}, ->] (A) -- (C);
\draw[transform canvas={xshift=-0.6ex}, ->] (A) -- (C);

\draw[transform canvas={xshift=0.6ex}, ->] (C) -- (D);
\draw[transform canvas={xshift=-0.6ex}, ->] (C) -- (D);

\draw[transform canvas={xshift=0.6ex}, ->] (B) -- (D);
\draw[transform canvas={xshift=-0.6ex}, ->] (B) -- (D);

\node at (-0.55,0.55) {$q_1$};
\node at (-0.95,0.95) {$p_1$};

\node at (0.55,0.55) {$q_2$};
\node at (0.95,0.95) {$p_2$};

\node at (-0.55,-0.55) {$h_1$};
\node at (-0.95,-0.95) {$g_1$};

\node at (0.55,-0.55) {$h_2$};
\node at (0.95,-0.95) {$g_2$};
\end{tikzpicture}
\]

Suppose $q_1$ has $s$-path lifting and the conditions $(i)$--$(iii)$ of Theorem \ref{th 3D textile to 3-graph} are satisfied. Let $\Lambda_\mT$ be the associated $3$-graph. Then the chain complex $(C_*(\mT),\partial_*^\mT)$ is isomorphic to the chain complex $(C_*(\Lambda_\mT),\partial_*^{\Lambda_\mT})$. Consequently, $H_n(\mT)\cong H_n(\Lambda_\mT)$ for all $n\in \mathbb{N}$.
\end{thm}
\begin{proof}
Note that \[C_0(\Lambda_\mT)=\mathbb{Z}\Lambda_\mT^0=\mathbb{Z}E^0=C_0(\mT)\] and \[C_1(\Lambda_\mT)=\mathbb{Z}(\Lambda_\mT^{\mathbf{e}_1}\sqcup \Lambda_\mT^{\mathbf{e}_2}\sqcup \Lambda_\mT^{\mathbf{e}_3})=\mathbb{Z}(F_1^0\sqcup F_2^0\sqcup E^1)=C_1(\mT).\] Now, consider the maps $\phi_{12}:G^0\longrightarrow \Lambda_\mT^{\mathbf{e}_1+\mathbf{e}_2}$, $\phi_{23}:F_1^1\longrightarrow \Lambda_\mT^{\mathbf{e}_2+\mathbf{e}_3}$ and $\phi_{13}:F_2^1\longrightarrow \Lambda_\mT^{\mathbf{e}_1+\mathbf{e}_3}$ defined by \[\phi_{12}(x):=p_2^0(x)q_1^0(x),~~\phi_{23}(x):=s_{F_1}(x)h_1^1(x),~~\phi_{13}(x):=s_{F_2}(x)h_2^1(x).\] In view of the conditions $(i)$--$(iii)$, it is easy to observe that the above three maps are bijections and thus give rise to group isomorphisms $\Tilde{\phi}_{12}$, $\Tilde{\phi}_{23}$ and $\Tilde{\phi}_{13}$ between the respective free abelian groups. Now, we define $\Phi:C_2(\mT)\longrightarrow C_2(\Lambda_\mT)$ as $\Phi=\Tilde{\phi}_{12}\oplus \Tilde{\phi}_{23} \oplus \Tilde{\phi}_{13}$. Then $\Phi$ is clearly an isomorphism. Define 
\begin{align*}
\psi:G^1 &\longrightarrow \Lambda_\mT^{\mathbf{e}_1+\mathbf{e}_2+\mathbf{e}_3}\\
x &\longmapsto p_2^0(s_G(x))q_1^0(s_G(x))h_1^1(q_1^1(x))=p_2^0(s_G(x))s_{F_1}(q_1^1(x))h_1^1(q_1^1(x)).
\end{align*}
The map is well-defined by the definition of $\Lambda_\mT$. Let $\gamma\in \Lambda_\mT^{\mathbf{e}_1+\mathbf{e}_2+\mathbf{e}_3}$. Suppose $y:=\phi_{12}^{-1}(\gamma(0,\mathbf{e}_1+\mathbf{e}_2))$ and $z:=\phi_{23}^{-1}(\gamma(\mathbf{e}_1,\mathbf{e}_1+\mathbf{e}_2+\mathbf{e}_3))$. Then \[q_1^0(y)=\gamma(\mathbf{e}_1,\mathbf{e}_1+\mathbf{e}_2)=s_{F_1}(z)\] and so by the $s$-path lifting of $q_1$. there exists $x\in G^1$ with $q_1^1(x)=z$ and $s_G(x)=y$. It follows that $\psi(x)=p_2^0(s_G(x))q_1^0(s_G(x))h_1^1(q_1^1(x))=p_2^0(y)q_1^0(y)h_1^1(z)=\gamma(0,\mathbf{e}_1)\gamma(\mathbf{e}_1,\mathbf{e}_1+\mathbf{e}_2)\gamma(\mathbf{e}_1+\mathbf{e}_2,\mathbf{e}_1+\mathbf{e}_2+\mathbf{e}_3)=\gamma$. This shows that $\psi$ is surjective. For injectivity, note that if $\psi(x)=\psi(y)$, then \[s_G(x)=\phi_{12}^{-1}(\psi(x)(0,\mathbf{e}_1+\mathbf{e}_2))=\phi_{12}^{-1}(\psi(y)(0,\mathbf{e}_1+\mathbf{e}_2))=s_G(y)\] and \[q_1^1(x)=\phi_{23}^{-1}(\psi(x)(\mathbf{e}_1,\mathbf{e}_1+\mathbf{e}_2+\mathbf{e}_3))=\phi_{23}^{-1}(\psi(y)(\mathbf{e}_1,\mathbf{e}_1+\mathbf{e}_2+\mathbf{e}_3))=q_1^1(y).\] The uniqueness of the $s$-path lifting of $q_1$ now implies, $x=y$. Thus, the map $\psi$ is bijective and induces an isomorphism $\Psi:C_3(\mT)\longrightarrow C_3(\Lambda_\mT)$. We will be done if we can show that the following diagram of chain complexes commute.
\[
\begin{tikzcd}
0 & C_0(\mathcal{T}) \arrow[l] \arrow[d, equal] & C_1(\mathcal{T}) \arrow[l, "\partial_1^\mT"'] \arrow[d, equal] & C_2(\mathcal{T}) \arrow[l, "\partial_2^\mT"'] \arrow[d, "\Phi"] & C_3(\mathcal{T}) \arrow[l, "\partial_3^\mT"'] \arrow[d, "\Psi"] & 0 \arrow[l] \\
0 & C_0(\Lambda_{\mathcal{T}}) \arrow[l]         & C_1(\Lambda_{\mathcal{T}}) \arrow[l, "\partial_1^{\Lambda_{\mathcal{T}}}"'] & C_2(\Lambda_{\mathcal{T}}) \arrow[l, "\partial_2^{\Lambda_{\mathcal{T}}}"'] & C_3(\Lambda_{\mathcal{T}}) \arrow[l, "\partial_3^{\Lambda_{\mathcal{T}}}"'] & 0 \arrow[l]
\end{tikzcd}
\]
We only check that $\partial_3^{\Lambda_\mT}\circ \Psi=\Phi\circ \partial_3^\mT$ and leave the other verifications to the reader. Let $x\in G^1$. Then 
\begin{align*}
\partial_3^{\Lambda_\mT}(\Psi(x))=&\displaystyle{\sum_{l=0}^{1}}\left(\displaystyle{\sum_{i=1}^{3}}(-1)^{i+l}F_i^l(\psi(x))\right)\\
=&-\psi(x)(0,\eb+\ec)+\psi(x)(0,\ea+\ec)-\psi(x)(0,\ea+\eb)\\
&+\psi(x)(\ea,\ea+\eb+\ec)-\psi(x)(\eb,\ea+\eb+\ec)+\psi(x)(\ec,\ea+\eb+\ec)\\
=&-p_1^0(s_G(x))g_2^1(q_2^1(x))+p_2^0(s_G(x))g_1^1(q_1^1(x))-p_2^0(s_G(x))q_1^0(s_G(x))\\
&+q_1^0(s_G(x))h_1^1(q_1^1(x))-q_2^0(s_G(x))h_1^1(q_1^1(x))+p_2^0(r_G(x))q_1^0(r_G(x))\\
=&~ \phi_{12}(r_G(x))+\phi_{23}(q_1^1(x))+\phi_{13}(p_2^1(x))-\phi_{12}(s_G(x))-\phi_{23}(p_1^1(x))-\phi_{13}(q_2^1(x))\\
=&~\Phi(r_G(x)+q_1^1(x)+p_2^1(x)-s_G(x)-p_1^1(x)-q_2^1(x))=\Phi(\partial_3^\mT(x)).
\end{align*}
\end{proof}

\textbf{Acknowledgments:} The authors would like to thank David Pask for sharing his notes which provided the background material on textile systems and their homology. This work was initiated during visits by the second and third authors to Western Sydney University in October 2025. They thank the first author for his warm hospitality. The first author acknowledges Australian Research Council Discovery Project DP230103184.

\end{document}